\documentclass[reqno]{amsart}
\usepackage{amsmath, amsfonts, amsthm, amssymb, setspace, textcomp,
	fullpage, bbm, multirow}
\newtheorem{theorem}{Theorem}[section]
\newtheorem{lemma}[theorem]{Lemma}

\newtheorem{corollary}[theorem]{Corollary}
\newtheorem{proposition}[theorem]{Proposition}

\theoremstyle{definition}

\theoremstyle{remark}

\numberwithin{equation}{section}

\newcommand{\spt}[2]{\mbox{\normalfont spt}_{#1}\Parans{#2}}

\newcommand{\Parans}[1]{\left(#1\right)}
\newcommand{\CBrackets}[1]{\left\{#1\right\}}
\newcommand{\SBrackets}[1]{\left[#1\right]}

\newcommand{\PieceTwo}[4]
{
	\left\{
   	\begin{array}{ll}
      	#1 & #3 \\
       	#2 & #4
     	\end{array}
	\right.
}
\newcommand{\aqprod}[3]{\Parans{#1;#2}_{#3}}
\newcommand{\jacprod}[2]{\SBrackets{#1;#2}_{\infty}}

\newcommand{\GEta}[3]{\eta_{#1,#2}\Parans{#3}}

\newcommand{\TwoPhiOne}[5]{
	{_2\phi_1}
	\left(\hspace{-5pt}
	\begin{array}{ccc}
		#1, &\hspace{-5pt}#2  & \hspace{-5pt}\multirow{2}{*}{; $#4$, $#5$}
		\\
		& \hspace{-5pt}#3 &
	\end{array}\hspace{-5pt}
	\right)
}

\author{CHRIS JENNINGS-SHAFFER}
\address{Department of Mathematics, Oregon State University\\
Corvallis, Oregon 97330, USA
\endgraf jennichr@math.oregonstate.edu}

\keywords{Number theory, Andrews' spt-functon, partitions, partition pairs, 
smallest parts function, congruences, Bailey pairs, Bailey's Lemma,
conjugate Bailey pairs}

\subjclass[2010]{Primary 11P81, 11P83, 33D15}

\title{Exotic Bailey-Slater SPT-Functions III: 
Bailey Pairs from Groups B, F, G, and J}

\allowdisplaybreaks
\begin{document}

\allowdisplaybreaks

\begin{abstract}
We continue to investigate spt-type functions that arise from Bailey pairs.
In this third paper on the subject, we proceed to introduce additional
spt-type functions.
We prove simple Ramanujan type congruences for these functions 
which can be explained by a spt-crank-type function. 
The spt-crank-type functions are actually defined first, with the 
spt-type functions coming from setting $z=1$ in the spt-crank-type functions.
We find some of the spt-crank-type functions to have interesting representations
as single series, some of which reduce to infinite products. 
Additionally we find dissections of the other spt-crank-type functions when 
$z$ is a certain root of unity. Both methods are used to
explain congruences for the spt-type functions. 
Our series formulas require Bailey's Lemma and conjugate Bailey pairs.
Our dissection formulas follow from Bailey's Lemma and dissections of known 
ranks and cranks. 

\end{abstract}

\maketitle

\section{Introduction}
\allowdisplaybreaks

We proceed with the study of spt-crank-type functions that the author began in
\cite{JenningsShaffer} and continued in \cite{GarvanJennings2}.  We begin with 
a brief introduction. We recall
a partition of $n$ is a non-increasing sequence of positive integers that sum to
$n$. For example, the partitions of $4$ are $4$, $3+1$, 
$2+2$, $2+1+1$, and $1+1+1+1$. We have Andrews smallest parts function from
\cite{Andrews}, $\spt{}{n}$, as the 
weighted count on partitions given by counting a partition
by the number of times the smallest part appears. From the partitions of $4$ we see that
$\spt{}{4}=10$. In this article we consider variations of the smallest parts functions.
We use the standard product notation
\begin{align*}
	\aqprod{z}{q}{n} 
		&= \prod_{j=0}^{n-1} (1-zq^j)
	,
	&\aqprod{z}{q}{\infty} 
		&= \prod_{j=0}^\infty (1-zq^j)
	,\\
	\aqprod{z_1,\dots,z_k}{q}{n} 
		&= \aqprod{z_1}{q}{n}\dots\aqprod{z_k}{q}{n}
	,
	&\aqprod{z_1,\dots,z_k}{q}{\infty} 
		&= \aqprod{z_1}{q}{\infty}\dots\aqprod{z_k}{q}{\infty},
	\\
	\jacprod{z}{q} 
		&= \aqprod{z,q/z}{q}{\infty}
	,	
	&\jacprod{z_1,\dots,z_k}{q} &= \jacprod{z_1}{q}\dots\jacprod{z_k}{q}
.
\end{align*}
We recall that a pair of sequences $(\alpha,\beta)$ is a Bailey pair relative
to $(a,q)$ if
\begin{align*}
	\beta_n = \sum_{k=0}^n \frac{\alpha_k}{\aqprod{q}{q}{n-k}\aqprod{aq}{q}{n+k}}
.
\end{align*}
One may consult \cite{Andrews2} for a history of Bailey pairs and
Bailey's Lemma.
Motivated by the prototype spt-crank functions
of \cite{AGL} and \cite{GarvanJennings} for partitions and overpartitions, we 
consider an spt-crank-type function to be a function of the form
\begin{align*}
	\frac{P(q)}{\aqprod{z,z^{-1}}{q}{\infty}}\sum_{n=1}^\infty \aqprod{z,z^{-1}}{q}{n}q^n\beta_n
,
\end{align*}
where $P(q)$ is some product and $\beta$ comes from a Bailey pair 
relative to $(1,q)$. We consider an spt-type function to be the $z=1$ case of an
spt-crank-type function. That is, for a Bailey pair $(\alpha^X,\beta^X)$ 
relative to $(1,q)$ we have the spt-crank-type and spt-type functions given by
\begin{align*}
	S_{X}(z,q) 
	&= 
		\frac{P_X(q)}{\aqprod{z,z^{-1}}{q}{\infty}}
		\sum_{n=1}^\infty \aqprod{z,z^{-1}}{q}{n} q^n \beta^X_n 
		=
		\sum_{n=1}^\infty \sum_{m=-\infty}^\infty M_X(m,n)z^mq^n
	,\\
	S_{X}(q) 
	&=
		S_{X}(1,q)
		=
		\sum_{n=1}^\infty \spt{X}{n}q^n	
	.
\end{align*}

The author is in the process of studying interesting spt-crank-type and
spt-type functions. This article introduces the last
of the spt-crank-type and spt-type functions arising from Bailey pairs in 
\cite{Slater1} and \cite{Slater2} that possess simple linear congruences
of the form $\spt{X}{pn+b}\equiv 0\pmod{p}$, where $p$ is an odd prime.
In \cite{JenningsShaffer} the author introduced the spt-crank-type functions
$S_{A1}(z,q)$, $S_{A3}(z,q)$, $S_{A5}(z,q)$, and $S_{A7}(z,q)$ which correspond
to the Bailey pairs $A(1)$, $A(3)$, $A(5)$, and $A(7)$ of \cite{Slater1}. 
In \cite{GarvanJennings2} Garvan and the author introduced the spt-crank-type
functions $S_{C1}(z,q)$, $S_{C5}(z,q)$, $S_{E2}(z,q)$, and $S_{E4}(z,q)$.
These spt-type functions satisfy many linear congruences, in particular,
$\spt{A1}{3n}\equiv\spt{A3}{3n+1}\equiv\spt{E2}{3n}\equiv\spt{E4}{3n+1}
\equiv 0\pmod{3}$,
$\spt{A3}{5n+1}\equiv\spt{A5}{5n+4}\equiv\spt{A7}{5n+1}\equiv\spt{C1}{5n+3}
\equiv\spt{C5}{5n+3}\equiv 0\pmod{5}$, and
$\spt{A5}{7n+1}\equiv 0\pmod{7}$.
Here we consider the Bailey pairs $B(2)$, $F(3)$, $G(4)$, and the entry just above
$G(4)$ from \cite{Slater1} and $J(1)$, $J(2)$, and $J(3)$ from \cite{Slater2}.

We prove simple linear congruences for the $\spt{X}{n}$ by considering
$S_{X}(\zeta,q)$, where $\zeta$ is a root of unity.
For $t$ a positive integer we define
\begin{align*}
	M_X(k,t,n) &= \sum_{m\equiv k \pmod{t}}M_X(m,n)
.
\end{align*}
We note that
\begin{align*}
	\spt{X}{n} &= \sum_{k=0}^{t-1} M_X(k,t,n).
\end{align*}
When $\zeta_t$ is a $t^{th}$ root of unity, we have
\begin{align*}
	S_X(\zeta_t,q) 
	&=
	\sum_{n=1}^\infty \left(\sum_{k=0}^{t-1}M_X(k,t,n)\zeta_t^k\right)q^n
.
\end{align*}
The last equation is of great importance because if $t$ is prime and 
$\zeta_t$ is a primitive $t^{th}$ root of unity, then the minimal polynomial
for $\zeta_t$ is $1+x+x^2+\dots+x^{t-1}$. Thus if the coefficient of $q^N$ in
$S_X(\zeta_t,q)$ is zero, then $\sum_{k=0}^{t-1}M_X(k,t,N)\zeta_t^k$ is zero 
and so $M_X(0,t,N)=M_X(1,t,N)=\dots=M_X(t-1,t,N)$. But then we would have that
$\spt{X}{N}=t\cdot M_X(0,t,N)$ and so
if $M_X(0,t,N)$ is an integer then
 clearly $\spt{X}{N}\equiv 0\pmod{t}$.
That is to say, if the coefficient of $q^N$ in $S_X(\zeta_t,q)$ is zero, then
$\spt{X}{N}\equiv 0 \pmod{t}$. Thus not only do we have the congruence
$\spt{X}{N}\equiv 0 \pmod{t}$, but also the stronger combinatorial result that
all of the $M_{X}(r,t,N)$ are equal.

In \cite{JenningsShaffer} the author found dissection formulas for
the $S_{Ai}(z,q)$ when $z$ is the appropriate root of unity to establish the
various congruences. In \cite{GarvanJennings2} Garvan and the author similarly
found dissection formulas for the $S_{Ci}(z,q)$ and $S_{Ei}(z,q)$ when $z$ is
the appropriate root of unity. The main difference between these two papers is
that the $S_{Ci}(z,q)$ and $S_{Ei}(z,q)$ could be expressed in terms of 
functions with known dissections, whereas the $S_{Ai}(z,q)$ could not.
Additionally in \cite{GarvanJennings2}, we found
interesting series representations for the $S_{Ci}(z,q)$, $S_{Ei}(z,q)$, and 
the $S_{Ai}(z,q)$ that are valid for all values of $z$, rather than just a fixed
root of unity. These series representations were a combination of single series
representations that showed that some of the spt-crank-type functions could be 
written just in terms of infinite products, and double series representations
that could be written as so called Hecke-Rogers type double sums. 

In the next section we define the new spt-crank-type and spt-type functions and state 
our main results, which are 
congruences for the various spt-type functions, single series representations
for some of the spt-crank-type functions,
and dissection formulas for the other spt-crank-type functions.

\section{Preliminaries and Statement of Results}

To begin we use the Bailey pair $B(2)$ from \cite{Slater1}
and $J(1$), $J(2)$, and $J(3)$ from \cite{Slater2}. Each of
these is a Bailey pair relative to $(1,q)$ and in all cases
$\alpha_0=\beta_0=1$.
\begin{align*}
	\beta^{B2}_n &= \frac{q^n}{\aqprod{q}{q}{n}}
	,&
	\alpha^{B2}_n &= 
	\left\{\begin{array}{ll}
		1 & \mbox{ if } n=0
		\\
		(-1)^n q^{3(n^2-n)/2 } (1+q^{3n}) &\mbox{ if } n\ge 1
	\end{array}\right.
	,\\
	\beta^{J1}_n &= \frac{\aqprod{q^3}{q^3}{n-1}}
		{\aqprod{q}{q}{2n-1}\aqprod{q}{q}{n}}
	,&
	\alpha^{J1}_n &=
	\left\{\begin{array}{ll}
		0	& \mbox{ if } n = 3k-1
		\\
		(-1)^{k}q^{\frac{9k^2-3k}{2}}(1+q^{3k}) 		& \mbox{ if } n = 3k
		\\
		0	& \mbox{ if } n = 3k+1
	\end{array}\right.
	,\\
	\beta^{J2}_n &= \frac{\aqprod{q^3}{q^3}{n-1}}
		{\aqprod{q}{q}{2n}\aqprod{q}{q}{n-1}}
	,&
	\alpha^{J2}_n &=
	\left\{\begin{array}{ll}
		(-1)^{k-1}  q^{\frac{9k^2 - 9k + 2}{2}}		& \mbox{ if } n = 3k-1
		\\
		(-1)^{k}q^{\frac{9k^2-3k}{2}}(1+q^{3k}) 		& \mbox{ if } n = 3k
		\\
		(-1)^{k+1}  q^{\frac{9k^2 + 9k + 2}{2}}		& \mbox{ if } n = 3k+1
	\end{array}\right.
	,\\
	\beta^{J3}_n &= \frac{q^n\aqprod{q^3}{q^3}{n-1}}
		{\aqprod{q}{q}{2n}\aqprod{q}{q}{n-1}}
	,&
	\alpha^{J3}_n &=
	\left\{\begin{array}{ll}
		(-1)^{k-1}  q^{\frac{9k^2 - 3k }{2}}			& \mbox{ if } n = 3k-1
		\\
		(-1)^{k} q^{\frac{9k^2-3k}{2}}(1+q^{3k})		& \mbox{ if } n = 3k
		\\
		(-1)^{k+1}  q^{\frac{9k^2 + 3k }{2}}			& \mbox{ if } n = 3k+1
	\end{array}\right.
.
\end{align*}
We note these Bailey pairs from group $J$ also appear as unlabeled Bailey pairs 
on page 467 of 
\cite{Slater1}. 
Additionally, we use the following Bailey pairs relative to $(1,q^2)$,
from \cite{Slater1}:
\begin{align*}
	\beta^{F3}_n &= \frac{q^{-n}}{\aqprod{q,q^2}{q^2}{n}}
	,&
	\alpha^{F3}_n &= 
	\left\{\begin{array}{ll}
		1 & \mbox{ if } n=0
		\\
		q^n+q^{-n} & \mbox{ if } n \ge 1
	\end{array}\right.,
	\\
	\beta^{G4}_n &= \frac{(-1)^nq^{n^2}}{\aqprod{q^4}{q^4}{n}\aqprod{-q}{q^2}{n}}
	,&
	\alpha^{G4}_n &= 
	\left\{\begin{array}{ll}
		1 & \mbox{ if } n=0
		\\
		(-1)^nq^{n(n-1)/2}(1+q^n) & \mbox{ if } n \ge 1
	\end{array}\right.,
	\\
	\beta^{AG4}_n &= \frac{(-1)^nq^{n^2-2n}}{\aqprod{q^4}{q^4}{n}\aqprod{-q}{q^2}{n}}
	,&
	\alpha^{AG4}_n &= 
	\left\{\begin{array}{ll}
		1 & \mbox{ if } n=0
		\\
		(-1)^nq^{n(n-3)/2}(1+q^{3n}) & \mbox{ if } n \ge 1
	\end{array}\right.
.
\end{align*}
The Bailey pair $AG(4)$ is the entry just above $G(4)$ in \cite{Slater1}.
For each Bailey pair we define a two variable spt-crank-type series as follows,
\begin{align*}
	S_{B2}(z,q)
	&=
	\frac{\aqprod{q}{q}{\infty}}{\aqprod{z,z^{-1}}{q}{\infty}}	
	\sum_{n=1}^\infty \frac{\aqprod{z,z^{-1}}{q}{n} q^{2n}}
		{\aqprod{q}{q}{n}}
	,\\
	S_{F3}(z,q)
	&=
	\frac{\aqprod{q}{q}{\infty}}{\aqprod{z,z^{-1}}{q^2}{\infty}}
	\sum_{n=1}^\infty \frac{\aqprod{z,z^{-1}}{q^2}{n} q^{n}}
		{\aqprod{q}{q}{2n}}
	,\\
	S_{G4}(z,q)
	&=
	\frac{\aqprod{q^4}{q^4}{\infty}\aqprod{-q}{q^2}{\infty}}
		{\aqprod{z,z^{-1}}{q^2}{\infty}}
	\sum_{n=1}^\infty \frac{\aqprod{z,z^{-1}}{q^2}{n}(-1)^nq^{n^2+2n}}
		{\aqprod{q^4}{q^4}{n}\aqprod{-q}{q^2}{n}}
	,\\
	S_{AG4}(z,q)
	&=
	\frac{\aqprod{q^4}{q^4}{\infty}\aqprod{-q}{q^2}{\infty}}
		{\aqprod{z,z^{-1}}{q^2}{\infty}}
	\sum_{n=1}^\infty \frac{\aqprod{z,z^{-1}}{q^2}{n}(-1)^nq^{n^2}}
		{\aqprod{q^4}{q^4}{n}\aqprod{-q}{q^2}{n}}
	,\\
	S_{J1}(z,q)
	&=
	\frac{\aqprod{q}{q}{\infty}^2}{\aqprod{q^3}{q^3}{\infty}\aqprod{z,z^{-1}}{q}{\infty}}
	\sum_{n=1}^\infty \frac{\aqprod{z,z^{-1}}{q}{n}\aqprod{q^3}{q^3}{n-1}q^n}
		{\aqprod{q}{q}{2n-1}\aqprod{q}{q}{n}}
	,\\
	S_{J2}(z,q)
	&=
	\frac{\aqprod{q}{q}{\infty}^2}{\aqprod{q^3}{q^3}{\infty}\aqprod{z,z^{-1}}{q}{\infty}}
	\sum_{n=1}^\infty \frac{\aqprod{z,z^{-1}}{q}{n}\aqprod{q^3}{q^3}{n-1}q^n}
		{\aqprod{q}{q}{2n}\aqprod{q}{q}{n-1}}
	,\\
	S_{J3}(z,q)
	&=
	\frac{\aqprod{q}{q}{\infty}^2}{\aqprod{q^3}{q^3}{\infty}\aqprod{z,z^{-1}}{q}{\infty}}
	\sum_{n=1}^\infty \frac{\aqprod{z,z^{-1}}{q}{n}\aqprod{q^3}{q^3}{n-1}q^{2n}}
		{\aqprod{q}{q}{2n}\aqprod{q}{q}{n-1}}
.
\end{align*}
While it is not true that $J(1)=J(2)+J(3)$, because $\beta^{J1}_{0}=1$, 
we do have $\beta^{J1}_n=\beta^{J2}_n+\beta^{J3}_n$ for $n\ge 1$ and so
$S_{J1}(z,q) = S_{J2}(z,q)+S_{J3}(z,q)$.

Next we define the corresponding spt-type functions.
For $B2$, $J1$, $J2$, and $J3$ we just set $z=1$ and simplify the products,
but for $F3$, $G4$, and $AG4$ we make some additional rearrangements.
\begin{align*}
	S_{B2}(q)
	=
	\sum_{n=1}^\infty \spt{B2}{n}q^n
	&=
	\sum_{n=1}^\infty \frac{q^{2n}}{(1-q^n)^2\aqprod{q^{n+1}}{q}{\infty}}
	,\\
	S_{J1}(q)
	=
	\sum_{n=1}^\infty \spt{J1}{n}q^n
	&=
	\sum_{n=1}^\infty \frac{q^n}{(1-q^n)^2\aqprod{q^{n+1}}{q}{n-1}\aqprod{q^{3n}}{q^3}{\infty}}	
	,\\
	S_{J2}(q)
	=
	\sum_{n=1}^\infty \spt{J2}{n}q^n
	&=
	\sum_{n=1}^\infty \frac{q^n}{\aqprod{q^n}{q}{n+1}\aqprod{q^{3n}}{q^3}{\infty}}	
	,\\
	S_{J3}(q)
	=
	\sum_{n=1}^\infty \spt{J3}{n}q^n
	&=
	\sum_{n=1}^\infty \frac{q^{2n}}{\aqprod{q^n}{q}{n+1}\aqprod{q^{3n}}{q^3}{\infty}}	
	,\\
	S_{F3}(q)	
	=
	\sum_{n=1}^\infty \spt{F3}{n}q^n
	&=
	\sum_{n=1}^\infty
	\frac{q^n\aqprod{q^{2n+1}}{q^2}{\infty}}
		{(1-q^{2n})^2\aqprod{q^{2n+2}}{q^2}{\infty}}
		\\
		&=
		\sum_{n=1}^\infty
		\frac{q^n}{(1-q^{2n})^2\aqprod{q^{n+1}}{q}{n}\aqprod{q^{2n+2}}{q^2}{\infty}}
		\frac{\aqprod{q^{n+1}}{q}{\infty}}{\aqprod{q^{2n+2}}{q^2}{\infty}}
	,\\
	S_{G4}(q)
	=
	\sum_{n=1}^\infty \spt{G4}{n}q^n
	&=
		\sum_{n=1}^\infty	
		\frac{(-1)^nq^{n^2+2n}\aqprod{-q^{2n+1}}{q^2}{\infty}}
			{(1-q^{2n})^2\aqprod{q^{2n+2}}{q^2}{n+1}\aqprod{q^{2n+2}}{q^2}{\infty}\aqprod{q^{4n+6}}{q^4}{\infty}}
		\\
		&=
		\sum_{n=1}^\infty	
		\frac{(-1)^nq^{n^2+2n}}
			{(1-q^{2n})^2\aqprod{q^{n+1}}{q}{n}\aqprod{q^{2n+2}}{q^2}{n+1}\aqprod{q^{4n+6}}{q^4}{\infty}}
		\frac{\aqprod{q^{n+1}}{q}{n}\aqprod{-q^{2n+1}}{q^2}{\infty}}
			{\aqprod{q^{2n+2}}{q^2}{\infty}}
	,\\
	S_{AG4}(q)
	=
	\sum_{n=1}^\infty \spt{G4}{n}q^n
	&=
		\sum_{n=1}^\infty	
		\frac{(-1)^nq^{n^2}\aqprod{-q^{2n+1}}{q^2}{\infty}}
			{(1-q^{2n})^2\aqprod{q^{2n+2}}{q^2}{n+1}\aqprod{q^{2n+2}}{q^2}{\infty}\aqprod{q^{4n+6}}{q^4}{\infty}}
		\\
		&=
		\sum_{n=1}^\infty	
		\frac{(-1)^nq^{n^2}}
			{(1-q^{2n})^2\aqprod{q^{n+1}}{q}{n}\aqprod{q^{2n+2}}{q^2}{n+1}\aqprod{q^{4n+6}}{q^4}{\infty}}
		\frac{\aqprod{q^{n+1}}{q}{n}\aqprod{-q^{2n+1}}{q^2}{\infty}}
			{\aqprod{q^{2n+2}}{q^2}{\infty}}
	.
\end{align*}
We recall that an overpartition is a partition in which a part may be overlined
the first time it appears; overpartitions can be identified with partition
pairs $(\pi_1,\pi_2)$ where $\pi_2$ is restricted to having distinct parts.
For $\pi$ either a partition or an overpartition,
we let $s(\pi)$ denote the smallest part of $\pi$,
$spt(\pi)$ denote the number of times $s(\pi)$ occurs,
and $\#(\pi)$ denote the number of parts of $\pi$.
For overpartitions 
we let a superscript $n$ in these operators mean the restriction to the
non-overlined parts and a superscript $o$ mean the restriction to the overlined parts.
For example, $\#^o(\pi)$ is the number of overlined parts of the overpartition
$\pi$ and $s^n(\pi)$ is the smallest non-overlined part.
We can now give the combinatorial interpretation of the various 
spt-type functions.

We see $\spt{B2}{n}$ is the number of partitions $\pi$ of $n$ weighted by
the number of times $s(\pi)$ appears past the first occurrence. From this 
interpretation we see that $\spt{B2}{n}=\spt{}{n}-p(n)$.
We see $\spt{J1}{n}$ is the number of partitions $\pi$ of $n$ 
weighted by the number of times $s(\pi)$ appears,
where the allowed parts are those from $s(\pi)$ to $2s(\pi)-1$ and those that
 are divisible by 
$3$ and at least $3s(\pi)$.
We see $\spt{J2}{n}$ is the number of partitions $\pi$ of $n$ where 
the parts are those from $s(\pi)$ to $2s(\pi)$ and those that are divisible by $3$
and at least $3s(\pi)$.
Similarly 
$\spt{J3}{n}$ is the number of partitions $\pi$ of $n$ where 
the smallest part appears at least twice and
the parts are those from $s(\pi)$ to $2s(\pi)$ and those that are divisible by $3$
and at least $3s(\pi)$.

From the generating functions we see that $\spt{J1}{n}=\spt{J2}{n}+\spt{J3}{n}$,
as pointed out by the referee is it also not difficult to explain this combinatorially.
Suppose we fix $n$ and let $Ji(n)$ denote the set of partitions counted by
$\spt{Ji}{n}$. We have $J1(n)$ is the set of partitions
$\pi$ of $n$ where no part $\pi_i$ satisfies $2s(\pi)\le\pi_i<3s(\pi)$,
and if a part $\pi_i\ge 3s(\pi)$ then $3$ divides $\pi_i$.
Similarly $J2(n)$ is the set of partitions
$\pi$ of $n$ where no part $\pi_i$ satisfies $2s(\pi)<\pi_i<3s(\pi)$,
and if a part $\pi_i\ge 3s(\pi)$ then $3$ divides $\pi_i$.
Lastly $J3(n)$ is the set of partitions
$\pi$ of $n$ where $s(\pi)$ appears at least twice,
no part $\pi_i$ satisfies $2s(\pi)\le\pi_i<3s(\pi)$,
and if a part $\pi_i\ge 3s(\pi)$ then $3$ divides $\pi_i$.
Given a partition $\pi\in J2(n)$, we obtain a partition of $J1(n)$
by taking the part $2s(\pi)$ and writing it as $s(\pi)+s(\pi)$, so the 
smallest part appears two more times for each time $2s(\pi)$ appeared;
this function is clearly onto and is exactly 
$\left\lfloor\frac{spt(\pi)+1}{2}\right\rfloor$-to-one.
Given a partition $\pi\in J3(n)$, we obtain a partition of $J1(n)$
in the same way,
but now we miss the elements of $J1(n)$ whose smallest part appears
exactly once and this function is
$\left\lfloor\frac{spt(\pi)}{2}\right\rfloor$-to-one.  
Since 
$\left\lfloor\frac{spt(\pi)+1}{2}\right\rfloor + \left\lfloor\frac{spt(\pi)}{2}\right\rfloor
=spt(\pi)$, we see $\spt{J1}{n}=\spt{J2}{n}+\spt{J3}{n}$

For $S_{F3}(q)$, we first note that 
\begin{align*}
	\frac{q^n}{(1-q^{2n})^2}
	&=
	q^n + 2q^{3n} + 3q^{5n} + 4q^{7n} + 5q^{9n} +\dots
.
\end{align*}
We let $F3$ denote the set of pairs $(\pi_1,\pi_2)$,
where $\pi_1$ is a partition with $spt(\pi_1)$ odd and 
the parts that are at least $2s(\pi)$ must even; and $\pi_2$ is an overpartition 
where all non-overlined parts are even, $s^n(\pi_2) \ge 2s(\pi_1)+2$, 
and $s^o(\pi_2)\ge s(\pi_1)+1$.
Then we see that $\spt{F3}{n}$ is the number of partition pairs $(\pi_1,\pi_2)$
of $n$ from $F3$, weighted by 
$(-1)^{\#^o(\pi_2)}\left(\frac{spt(\pi_1)+1}{2}\right)$.

For $S_{G4}(q)$, we first note that
\begin{align*}
	\frac{q^{n^2+2n}}{(1-q^{2n})^2}
	&=
	q^{n(n+2)} + 2q^{n(n+4)} + 3q^{n(n+6)} +  4q^{n(n+8)}+\dots
.
\end{align*}
We let $G4$ be the set of pairs $(\pi_1,\pi_2)$ where
$\pi_1$ is a partition
such that $spt(\pi_1) \ge s(\pi_1)+2$,
$spt(\pi_1)+s(\pi_1)$ is even,
parts larger than $2s(\pi_1)$ must be even,
and parts larger than $4s(\pi_1)$ must be congruent to $2\pmod{4}$;
and $\pi_2$ is an overpartition
with all non-overlined parts even,
$s^n(\pi_2) \ge 2s(\pi_1)+2$,
$s^o(\pi_2) \ge s(\pi)+1$, and 
overlined parts that are at least $2s(\pi_1)+1$ are odd.
For an overpartition $\pi$, we let $k_m(\pi)$ denote the number 
of overlined parts of $\pi$ that less than $2m+1$.
Then $\spt{G4}{n}$ is the number of partition pairs of $n$ from
$G4$, weighted by
$(-1)^{s(\pi_1) + k_{s(\pi_1)}(\pi_2) } \left(\frac{spt(\pi_1)-s(\pi)}{2}\right)$.

For $S_{AG4}(q)$, we first note that
\begin{align*}
	\frac{q^{n^2}}{(1-q^{2n})^2}
	&=
	q^{n(n)} + 2q^{n(n+2)} + 3q^{n(n+4)} +  4q^{n(n+6)}+\dots
.
\end{align*}
We let $AG4$ be the set of pairs $(\pi_1,\pi_2)$ where
$\pi_1$ is a partition
such that $spt(\pi_1) \ge s(\pi_1)$,
$spt(\pi_1)+s(\pi_1)$ is even,
parts larger than $2s(\pi_1)$ must be even,
and parts larger than $4s(\pi_1)$ must be congruent to $2\pmod{4}$;
and $\pi_2$ is an overpartition
with all non-overlined parts even,
$s^n(\pi_2) \ge 2s(\pi_1)+2$,
$s^o(\pi_2) \ge s(\pi)+1$, and 
overlined parts that are at least $2s(\pi_1)+1$ are odd.
Then $\spt{AG4}{n}$ is the number of partition pairs of $n$ from
$AG4$, weighted by
$(-1)^{s(\pi_1) + k_{s(\pi_1)}(\pi_2) } \left(\frac{spt(\pi_1)-s(\pi)+2}{2}\right)$.

These functions satisfy the following congruences.
\begin{theorem}\label{TheoremTheCongruences}
\begin{align*}
	\spt{F3}{3n}\equiv 0\pmod{3}
	,\\
	\spt{J1}{3n+2}\equiv 0\pmod{3}
	,\\
	\spt{J2}{3n}\equiv 0\pmod{3}
	,\\
	\spt{J3}{3n+1}\equiv 0\pmod{3}
	,\\
	\spt{B2}{5n+1}\equiv 0\pmod{5}
	,\\
	\spt{B2}{5n+4}\equiv 0\pmod{5}
	,\\
	\spt{F3}{5n}\equiv 0\pmod{5}
	,\\
	\spt{F3}{5n+4}\equiv 0\pmod{5}
	,\\
	\spt{G4}{5n+4}\equiv 0\pmod{5}
	,\\
	\spt{AG4}{5n+4}\equiv 0\pmod{5}
	,\\
	\spt{B2}{7n+1}\equiv 0\pmod{7}
	,\\
	\spt{B2}{7n+5}\equiv 0\pmod{7}
	,\\
	\spt{F3}{7n}\equiv 0\pmod{7}
	,\\
	\spt{F3}{7n+4}\equiv 0\pmod{7}
	,\\
	\spt{F3}{7n+6}\equiv 0\pmod{7}
.
\end{align*}
\end{theorem}
That $\spt{J1}{3n+2}\equiv 0$ is actually known. In \cite{Patkowski} 
Patkowski considered this smallest parts function and proved that 
$\spt{J1}{3n+2}\equiv 0$. Although that proof is dependent on Bailey's Lemma,
the proof is not through a spt-crank-type function as we have here.
Since $\spt{B2}{n}=\spt{}{n}-p(n)$, the congruences 
$\spt{B2}{5n+4}\equiv 0\pmod{5}$ and $\spt{B2}{7n+5}\equiv 0\pmod{7}$
also follow from the fact that both $\spt{}{n}$ and $p(n)$ satisfy these
congruences.
We use the spt-crank-type functions to prove the congruences of Theorem \ref{TheoremTheCongruences}
 as 
explained in the introduction. This will be as a corollary to the following
two theorems.
\begin{theorem}\label{TheoremSingleSeries}
\begin{align}
	\label{TheoremSeriesForJ1}
	&(1+z)\aqprod{z,z^{-1},q}{q}{\infty}S_{J1}(z,q)
	\nonumber\\
	&=\sum_{j=2}^\infty 
		\frac{(1-z^{j-1})(1-z^j)z^{1-j}(-1)^{j+1}q^{\frac{j(j-1)}{2}}
			(1-q^j - q^{2j-2} + q^{4j-3} + q^{5j-2} - q^{6j-3} )}
		{(1-q^{3j-3})(1-q^{3j})}
	,\\
	\label{TheoremSeriesForJ2}
	&(1+z)\aqprod{z,z^{-1},q}{q}{\infty}S_{J2}(z,q)
	\nonumber\\
	&=\sum_{j=2}^\infty 
		\frac{(1-z^{j-1})(1-z^j)z^{1-j}(-1)^{j+1}q^{\frac{j(j-1)}{2}}
			(1-q^{j-1} - q^{2j} + q^{4j-1} + q^{5j-3} - q^{6j-3} )}
		{(1-q^{3j-3})(1-q^{3j})}
	,\\
	\label{TheoremSeriesForJ3}
	&(1+z)\aqprod{z,z^{-1},q}{q}{\infty}S_{J3}(z,q)
	\nonumber\\
	&=\sum_{j=2}^\infty 
		\frac{(1-z^{j-1})(1-z^j)z^{1-j}(-1)^{j+1}q^{\frac{j(j-1)}{2}}
			(q^{j-1} - q^{j} - q^{2j-2} + q^{2j} + q^{4j-3} - q^{4j-1} - q^{5j-3} + q^{5j-2})}
		{(1-q^{3j-3})(1-q^{3j})}
	,\\
	\label{TheoremSeriesForF3}
	&(1+z)\aqprod{z,z^{-1},q}{q^2}{\infty}S_{F3}(z,q)
		=\sum_{j=-\infty}^\infty 
		(1-z^{j-1})(1-z^j)z^{1-j}(-1)^{j+1}q^{(j-1)^2}
	,\\
	\label{TheoremSeriesForG4}
	&(1+z)\aqprod{z,z^{-1}}{q^2}{\infty}S_{G4}(z,q)
	=
		\sum_{j=-\infty}^\infty (1-z^{j-1})(1-z^j)z^{1-j}q^{2j^2-j}
	,\\
	\label{TheoremSeriesForAG4}
	&(1+z)\aqprod{z,z^{-1}}{q^2}{\infty}S_{AG4}(z,q)
	=	
		\sum_{j=-\infty}^\infty (1-z^{j-1})(1-z^j)z^{1-j}q^{2j^2+j}
	.
\end{align}
\end{theorem}
\begin{theorem}\label{TheoremDissections}
\begin{align}
	\label{Dissection5B2}
	&S_{B2}(\zeta_5,q)
	\nonumber\\
	&=
		1
		-\frac{\aqprod{q^{25}}{q^{25}}{\infty}\jacprod{q^{10}}{q^{25}}}
			{\jacprod{q^{5}}{q^{25}}^2}
		+
		(1-\zeta_5-\zeta_5^4)q^5\frac{1}{\aqprod{q^{25}}{q^{25}}{\infty}}
			\sum_{n=-\infty}^\infty \frac{(-1)^nq^{75n(n+1)/2}}{1-q^{25n+5}}
		+
		q^2\frac{\aqprod{q^{25}}{q^{25}}{\infty}}
			{\jacprod{q^{10}}{q^{25}}}
		\nonumber\\&\quad
		+
		(\zeta_5+\zeta_5^4)
		q^3\frac{\aqprod{q^{25}}{q^{25}}{\infty}\jacprod{q^{5}}{q^{25}}}
			{\jacprod{q^{10}}{q^{25}}^2}
		+
		(\zeta_5+\zeta_5^4)q^8\frac{1}{\aqprod{q^{25}}{q^{25}}{\infty}}
		\sum_{n=-\infty}^\infty \frac{(-1)^nq^{75n(n+1)/2}}{1-q^{25n+10}}
	,\\
	\label{Dissection7B2}
	&S_{B2}(\zeta_7,q)
	\nonumber\\
	&=
		\zeta_7+\zeta_7^6
		-
		(\zeta_7+\zeta_7^6)	
		\frac{\aqprod{q^{49}}{q^{49}}{\infty}\jacprod{q^{21}}{q^{49}}}
			{\jacprod{q^{7},q^{14}}{q^{49}}}
		+
		(-1+\zeta_7+\zeta_7^6)
		q^7\frac{1}{\aqprod{q^{49}}{q^{49}}{\infty}}
			\sum_{n=-\infty}^\infty \frac{(-1)^nq^{147n(n+1)/2}}{1-q^{49n+7}}
		\nonumber\\&\quad
		+
		q^2\frac{\aqprod{q^{49}}{q^{49}}{\infty}\jacprod{q^{14}}{q^{49}}}
			{\jacprod{q^{7},q^{21}}{q^{49}}}	
		+
		(1+\zeta_7^2+\zeta_7^5)q^{16}\frac{1}{\aqprod{q^{49}}{q^{49}}{\infty}}
			\sum_{n=-\infty}^\infty \frac{(-1)^nq^{147n(n+1)/2}}{1-q^{49n+21}}
		\nonumber\\&\quad
		+
		(\zeta_7 +\zeta_7^6)
		q^3\frac{\aqprod{q^{49}}{q^{49}}{\infty}}
			{\jacprod{q^{14}}{q^{49}}}
		+
		(1 + \zeta_7 + \zeta_7^2 + \zeta_7^5 +\zeta_7^6)
		q^4\frac{\aqprod{q^{49}}{q^{49}}{\infty}}
			{\jacprod{q^{21}}{q^{49}}}
		\nonumber\\&\quad
		-		
		(\zeta_7^2 + \zeta_7^5)
		q^{13}\frac{1}{\aqprod{q^{49}}{q^{49}}{\infty}}
			\sum_{n=-\infty}^\infty \frac{(-1)^nq^{147n(n+1)/2}}{1-q^{49n+14}}
		+
		q^6\frac{\aqprod{q^{49}}{q^{49}}{\infty}\jacprod{q^{7}}{q^{49}}}
			{\jacprod{q^{14},q^{21}}{q^{49}}}
	,\\
	\label{Dissection3F3}
	&S_{F3}(\zeta_3,q)
	=
		q\frac{\aqprod{q^{18}}{q^{18}}{\infty}\aqprod{q^9}{q^9}\infty{}}
			{\aqprod{q^{6}}{q^{6}}{\infty}}
		+
		q^2\frac{\aqprod{q^{18}}{q^{18}}{\infty}^4\aqprod{q^3}{q^3}{\infty}}
			{\aqprod{q^{9}}{q^{9}}{\infty}^2\aqprod{q^6}{q^6}{\infty}^2}
	,\\
	\label{Dissection5F3}
	&S_{F3}(\zeta_5,q)
	\nonumber\\
	&=
		q\frac{\aqprod{q^{25}}{q^{25}}{\infty}}
			{\jacprod{q^{10}}{q^{50}}}
		+
		\frac{3+\zeta_5+\zeta_5^4}{5}
		q^2\frac{\aqprod{q^{50}}{q^{50}}{\infty}\jacprod{q^{15}}{q^{50}}}
			{\aqprod{q^{25}}{q^{50}}{\infty}\jacprod{q^{10}}{q^{50}}}
		-	
		\frac{1+2\zeta_5+2\zeta_5^4}{5}
		q^2\frac{\aqprod{q^{25}}{q^{25}}{\infty}\jacprod{q^{10}}{q^{25}}}
			{\jacprod{q^5}{q^{25}}\jacprod{q^{20}}{q^{50}}}
		\nonumber\\&\quad
		+
		\frac{3+\zeta_5+\zeta_5^4}{5}
		q^2\frac{\aqprod{q^{25}}{q^{25}}{\infty}\jacprod{q^5}{q^{25}}}
			{\jacprod{q^{10}}{q^{25}}\jacprod{q^{10}}{q^{50}}}
		-
		\frac{1+2\zeta_5+2\zeta_5^4}{5}
		q^7\frac{\aqprod{q^{50}}{q^{50}}{\infty}\jacprod{q^{5}}{q^{50}}}
			{\aqprod{q^{25}}{q^{50}}{\infty}\jacprod{q^{20}}{q^{50}}}
		\nonumber\\&\quad
		+
		(1+\zeta_5+\zeta_5^4)
		q^3\frac{\aqprod{q^{25}}{q^{25}}{\infty}}
			{\jacprod{q^{20}}{q^{50}}}
	,\\
	\label{Dissection7F3}
	&S_{F3}(\zeta_7,q)
	\nonumber\\
	&=
		\frac{18 + 9\zeta_7 + 3\zeta_7^2 + 3\zeta_7^5 + 9\zeta_7^6}{7}
		q\frac{\aqprod{q^{98}}{q^{98}}{\infty}\jacprod{q^{35}}{q^{98}}}
			{\aqprod{q^{49}}{q^{98}}{\infty}\jacprod{q^{14}}{q^{98}}}
		-
		\frac{5 + 6\zeta_7 + 2\zeta_7^2 + 2\zeta_7^5 + 6\zeta_7^6}{7}
		q\frac{\aqprod{q^{14}}{q^{14}}{\infty}}
		{\aqprod{q^{7}}{q^{14}}{\infty}}
		\nonumber\\&\quad
		+
		\frac{2 + \zeta_7 - 2\zeta_7^2 - 2\zeta_7^5 + \zeta_7^6}{7}
		q^8\frac{\aqprod{q^{98}}{q^{98}}{\infty}\jacprod{q^{21}}{q^{98}}}
			{\aqprod{q^{49}}{q^{98}}{\infty}\jacprod{q^{28}}{q^{98}}}
		-
		\frac{2 + \zeta_7 - 2\zeta_7^2 - 2\zeta_7^5 + \zeta_7^6}{7}
		q\frac{\aqprod{q^{49}}{q^{49}}{\infty}\jacprod{q^{35}}{q^{98}}}
			{\jacprod{q^{21}}{q^{49}}}
		\nonumber\\&\quad
		-
		\frac{4 + 2\zeta_7 + 3\zeta_7^2 + 3\zeta_7^5 + 2\zeta_7^6}{7}
		q\frac{\aqprod{q^{49}}{q^{49}}{\infty}\jacprod{q^{21}}{q^{98}}}
			{\jacprod{q^{7}}{q^{49}}}
		+
		q^2\frac{\aqprod{q^{49}}{q^{49}}{\infty}\jacprod{q^{14}}{q^{49}}}
			{\jacprod{q^{7}}{q^{49}}\jacprod{q^{28}}{q^{98}}}
		\nonumber\\&\quad
		+
		q^3(1+\zeta_7+\zeta_7^6)\frac{\aqprod{q^{49}}{q^{49}}{\infty}\jacprod{q^{7}}{q^{49}}}
			{\jacprod{q^{21}}{q^{49}}\jacprod{q^{14}}{q^{98}}}	
		+
		q^5(1+\zeta_7+\zeta_7^2+\zeta_7^5+\zeta_7^6)\frac{\aqprod{q^{49}}{q^{49}}{\infty}\jacprod{q^{21}}{q^{49}}}
			{\jacprod{q^{14}}{q^{49}}\jacprod{q^{42}}{q^{98}}}
	,\\
	\label{Dissection5G4}
	&S_{G4}(\zeta_5,q)
	\nonumber\\
	&=
		-(1+\zeta_5+\zeta_5^4)q^{10}
			\frac{\aqprod{q^{100}}{q^{100}}{\infty}\jacprod{q^{10}}{q^{200}}}
			{\jacprod{q^{10}}{q^{50}}\jacprod{q^{5}}{q^{100}}}
		-
		(\zeta_5+\zeta_5^4)q^{5}
			\frac{\aqprod{q^{50}}{q^{50}}{\infty}}
			{\aqprod{q^{25}}{q^{50}}{\infty}\jacprod{q^{20}}{q^{50}}}
		\nonumber\\&\quad
		-
		q^{6}
			\frac{\aqprod{q^{100}}{q^{100}}{\infty}\jacprod{q^{30}}{q^{200}}}
			{\jacprod{q^{10}}{q^{50}}\jacprod{q^{15}}{q^{100}}}
		-
		q^{12}
			\frac{\aqprod{q^{100}}{q^{100}}{\infty}\jacprod{q^{10}}{q^{200}}}
			{\jacprod{q^{20}}{q^{50}}\jacprod{q^{5}}{q^{100}}}
		-
		q^{3}
			\frac{\aqprod{q^{50}}{q^{50}}{\infty}}
			{\aqprod{q^{25}}{q^{50}}{\infty}\jacprod{q^{10}}{q^{50}}}
		\nonumber\\&\quad
		-
		(\zeta_5+\zeta_5^4)q^{8}
			\frac{\aqprod{q^{100}}{q^{100}}{\infty}\jacprod{q^{30}}{q^{200}}}
			{\jacprod{q^{20}}{q^{50}}\jacprod{q^{15}}{q^{100}}}
	,\\
	\label{Dissection5AG4}
	&S_{AG4}(\zeta_5,q)
	\nonumber\\
	&=
		-q^{10}
			\frac{\aqprod{q^{100}}{q^{100}}{\infty}\jacprod{q^{10}}{q^{200}}}
			{\jacprod{q^{10}}{q^{50}}\jacprod{q^{5}}{q^{100}}}
		-
		q\frac{\aqprod{q^{100}}{q^{100}}{\infty}\jacprod{q^{70}}{q^{200}}}
			{\jacprod{q^{10}}{q^{50}}\jacprod{q^{35}}{q^{100}}}
		\nonumber\\&\quad
		-
		(1+\zeta_5+\zeta_5^4)q^{6}
			\frac{\aqprod{q^{100}}{q^{100}}{\infty}\jacprod{q^{30}}{q^{200}}}
			{\jacprod{q^{10}}{q^{50}}\jacprod{q^{15}}{q^{100}}}
		-
		(\zeta_5+\zeta_5^4)q^{12}
			\frac{\aqprod{q^{100}}{q^{100}}{\infty}\jacprod{q^{10}}{q^{200}}}
			{\jacprod{q^{20}}{q^{50}}\jacprod{q^{5}}{q^{100}}}
		\nonumber\\&\quad
		-
		(\zeta_5+\zeta_5^4)q^{3}
			\frac{\aqprod{q^{100}}{q^{100}}{\infty}\jacprod{q^{70}}{q^{200}}}
			{\jacprod{q^{20}}{q^{50}}\jacprod{q^{35}}{q^{100}}}
		-q^8
			\frac{\aqprod{q^{100}}{q^{100}}{\infty}\jacprod{q^{30}}{q^{200}}}
			{\jacprod{q^{20}}{q^{50}}\jacprod{q^{15}}{q^{100}}}
.
\end{align}
\end{theorem}
We note the identities of Theorems \ref{TheoremSingleSeries} and 
\ref{TheoremDissections} are inherently different. In Theorem \ref{TheoremDissections}
 we have an identity 
for $z=\zeta_\ell$, a primitive $\ell^{th}$ root of unity and we have an 
explicit formula for
each term of the $\ell$-dissection. In Theorem \ref{TheoremSingleSeries}
we have an identity for general $z$,
but if we set $z=\zeta_\ell$, we are able to determine some but not necessary 
all of the terms in the $\ell$-dissection.

With these two Theorems we will show that the
coefficients of the following terms are zero:
$q^{3n}$ in $S_{F3}(\zeta_3,q)$,
$q^{3n+2}$ in $S_{J1}(\zeta_3,q)$,
$q^{3n}$ in $S_{J2}(\zeta_3,q)$,
$q^{3n+1}$ in $S_{J3}(\zeta_3,q)$,
$q^{5n+1}$ in $S_{B2}(\zeta_5,q)$,
$q^{5n+4}$ in $S_{B2}(\zeta_5,q)$,
$q^{5n}$ in $S_{F3}(\zeta_5,q)$,
$q^{5n+4}$ in $S_{F3}(\zeta_5,q)$,
$q^{5n+4}$ in $S_{G4}(\zeta_5,q)$,
$q^{5n+4}$ in $S_{AG4}(\zeta_5,q)$,
$q^{7n+1}$ in $S_{B2}(\zeta_7,q)$,
$q^{7n+5}$ in $S_{B2}(\zeta_7,q)$,
$q^{7n}$ in $S_{F3}(\zeta_7,q)$,
$q^{7n+4}$ in $S_{F3}(\zeta_7,q)$, and
$q^{7n+6}$ in $S_{F3}(\zeta_7,q)$.
As explained in the introduction, this gives the following corollary which also
establishes the congruences of Theorem \ref{TheoremTheCongruences}.
\begin{corollary}\label{TheoremEqualCranks}
For $n\ge 0$,
\begin{align*}
	M_{F3}(0,3,3n)&=M_{F3}(1,3,3n)=M_{F3}(2,3,3n)=\frac{1}{3}\spt{F3}{3n}
	,\\
	M_{J1}(0,3,3n+2)&=M_{J1}(1,3,3n+2)=M_{J1}(2,3,3n+2)=\frac{1}{3}\spt{J1}{3n+2}
	,\\
	M_{J2}(0,3,3n)&=M_{J2}(1,3,3n)=M_{J2}(2,3,3n)=\frac{1}{3}\spt{J2}{3n}
	,\\
	M_{J3}(0,3,3n+1)&=M_{J3}(1,3,3n+1)=M_{J3}(2,3,3n+1)=\frac{1}{3}\spt{J3}{3n+1}
	,\\
	M_{B2}(0,5,5n+1)&=M_{B2}(1,5,5n+1)=M_{B2}(2,5,5n+1)
		=M_{B2}(3,5,5n+1)=M_{B2}(4,5,5n+1)
		\\&=\frac{1}{5}\spt{B2}{5n+1}
	,\\
	M_{B2}(0,5,5n+4)&=M_{B2}(1,5,5n+4)=M_{B2}(2,5,5n+4)
		=M_{B2}(3,5,5n+4)=M_{B2}(4,5,5n+4)
		\\&=\frac{1}{5}\spt{B2}{5n+4}
	,\\
	M_{F3}(0,5,5n)&=M_{F3}(1,5,5n)=M_{F3}(2,5,5n)
		=M_{F3}(3,5,5n)=M_{F3}(4,5,5n)
		\\&=\frac{1}{5}\spt{F3}{5n}
	,\\
	M_{F3}(0,5,5n+4)&=M_{F3}(1,5,5n+4)=M_{F3}(2,5,5n+4)
		=M_{F3}(3,5,5n+4)=M_{F3}(4,5,5n+4)
		\\&=\frac{1}{5}\spt{F3}{5n+4}
	,\\
	M_{G4}(0,5,5n+4)&=M_{G4}(1,5,5n+4)=M_{G4}(2,5,5n+4)
		=M_{G4}(3,5,5n+4)=M_{G4}(4,5,5n+4)
		\\&=\frac{1}{5}\spt{G4}{5n+4}
	,\\
	M_{AG4}(0,5,5n+4)&=M_{AG4}(1,5,5n+4)=M_{AG4}(2,5,5n+4)
		=M_{AG4}(3,5,5n+4)=M_{AG4}(4,5,5n+4)
		\\&=\frac{1}{5}\spt{AG4}{5n+4}
	,\\
	M_{B2}(0,7,7n+1)&=M_{B2}(1,7,7n+1)=M_{B2}(2,7,7n+1)=M_{B2}(3,7,7n+1)=M_{B2}(4,7,7n+1)
		\\&=M_{B2}(5,7,7n+1)=M_{B2}(6,7,7n+1)=\frac{1}{7}\spt{B2}{7n+1}
	,\\
	M_{B2}(0,7,7n+5)&=M_{B2}(1,7,7n+5)=M_{B2}(2,7,7n+5)=M_{B2}(3,7,7n+5)=M_{B2}(4,7,7n+5)
		\\&=M_{B2}(5,7,7n+5)=M_{B2}(6,7,7n+5)=\frac{1}{7}\spt{B2}{7n+5}
	,\\
	M_{F3}(0,7,7n)&=M_{F3}(1,7,7n)=M_{F3}(2,7,7n)=M_{F3}(3,7,7n)
		=M_{F3}(4,7,7n)\\&=M_{F3}(5,7,7n)=M_{F3}(6,7,7n)=\frac{1}{7}\spt{F3}{7n}
	,\\
	M_{F3}(0,7,7n+4)&=M_{F3}(1,7,7n+4)=M_{F3}(2,7,7n+4)=M_{F3}(3,7,7n+4)
		=M_{F3}(4,7,7n+4)\\&=M_{F3}(5,7,7n+4)=M_{F3}(6,7,7n+4)=\frac{1}{7}\spt{F3}{7n+4}
	,\\
	M_{F3}(0,7,7n+6)&=M_{F3}(1,7,7n+6)=M_{F3}(2,7,7n+6)=M_{F3}(3,7,7n+6)
		=M_{F3}(4,7,7n+6)\\&=M_{F3}(5,7,7n+6)=M_{F3}(6,7,7n+6)=\frac{1}{7}\spt{F3}{7n+6}
	.
\end{align*}
\end{corollary}

We note (\ref{TheoremSeriesForJ1}) follows from adding (\ref{TheoremSeriesForJ2}) 
and (\ref{TheoremSeriesForJ3}).
Theorem \ref{TheoremSingleSeries} also 
lets us easily deduce the following product identities for 
$S_{F3}(z,q)$,
$S_{G4}(z,q)$, and $S_{AG4}(z,q)$.
\begin{corollary}
\begin{align*}
	S_{F3}(z,q)
	&=
	\frac{\aqprod{zq,z^{-1}q,q^2}{q^2}{\infty}}
		{\aqprod{z,z^{-1},q}{q^2}{\infty}}
	-
		\frac{\aqprod{q}{q}{\infty}}
			{\aqprod{z,z^{-1}}{q^2}{\infty}}
	,\\
	S_{G4}(z,q)
	&=
	\frac{z\aqprod{-z^{-1}q,-zq^3,q^4}{q^4}{\infty}}
		{(1+z)\aqprod{z,z^{-1}}{q^2}{\infty}}
	+
	\frac{\aqprod{-zq,-z^{-1}q^3,q^4}{q^4}{\infty}}
		{(1+z)\aqprod{z,z^{-1}}{q^2}{\infty}}
	-
	\frac{\aqprod{q^2}{q^2}{\infty}}{\aqprod{q,z,z^{-1}}{q^2}{\infty}}
	,\\
	S_{AG4}(z,q)
	&=
	\frac{z\aqprod{-zq,-z^{-1}q^3,q^4}{q^4}{\infty}}
		{(1+z)\aqprod{z,z^{-1}}{q^2}{\infty}}
	+
	\frac{\aqprod{-z^{-1}q,-zq^3,q^4}{q^4}{\infty}}
		{(1+z)\aqprod{z,z^{-1}}{q^2}{\infty}}
	-
	\frac{\aqprod{q^2}{q^2}{\infty}}{\aqprod{q,z,z^{-1}}{q^2}{\infty}}
.
\end{align*}
\end{corollary}
These follow by rearranging the series in 
Theorem \ref{TheoremSingleSeries} and applying the Jacobi triple product 
identity. For example,
\begin{align*}
	(1+z)\aqprod{z,z^{-1},q}{q^2}{\infty}S_{F3}(z,q)
	&=
		\sum_{j=-\infty}^\infty 
		(1-z^{j-1})(1-z^j)z^{1-j}(-1)^{j+1}q^{(j-1)^2}
	\\
	&=
		\sum_{j=-\infty}^\infty 
		(z^{1-j}+z^j)(-1)^{j+1}q^{(j-1)^2}
		-
		(1+z)\sum_{j=-\infty}^\infty 
		(-1)^{j+1}q^{(j-1)^2}
	\\
	&=
		\sum_{j=-\infty}^\infty 
		(z^{j-1}+z^j)(-1)^{j+1}q^{(j-1)^2}
		-
		(1+z)\frac{\aqprod{q}{q}{\infty}^2}{\aqprod{q^2}{q^2}{\infty}}
	\\
	&=
		(1+z)\sum_{j=-\infty}^\infty 
		z^j(-1)^{j}q^{j^2}
		-
		(1+z)\frac{\aqprod{q}{q}{\infty}^2}{\aqprod{q^2}{q^2}{\infty}}
	\\
	&=
		(1+z)\aqprod{zq,z^{-1}q,q^2}{q^2}{\infty}
		-
		(1+z)\frac{\aqprod{q}{q}{\infty}^2}{\aqprod{q^2}{q^2}{\infty}}
.
\end{align*}
The identities for $S_{G4}(z,q)$ and $S_{AG4}(z,q)$ are similar.

We summarize the results of this article in the following table:

\begin{tabular}{|c|c|c|c|c|}
	\hline
	Bailey pair & linear congruence & single series & product  & dissection  
	\\
	$X$ &  mod $p$ & identity for $S_X(z,q)$ & identity for $S_X(z,q)$ & identity for $S_X(\zeta_p,q)$
	\\
	\hline
	$B2$ & $p=5,7$ 		& No 	& No 	& Yes
	\\
	$F3$ & $p=3,5,7$ 	& Yes & Yes & Yes
	\\
	$G4$ & $p=5$			& Yes & Yes & Yes
	\\
	$AG4$ & $p=5$			& Yes & Yes & Yes
	\\
	$J1$ & $p=3$			& Yes & No & No
	\\
	$J2$ & $p=3$			& Yes & No & No
	\\
	$J3$ & $p=3$			&Yes & No & No
	\\
	\hline
\end{tabular}

\noindent
In Section 3 we prove the series identities in Theorem 
\ref{TheoremSingleSeries}. 
In Section 4 we prove the dissections for
$S_{B2}(\zeta_5,q)$ and $S_{B2}(\zeta_7,q)$.
In Section 5 we prove the dissections for $S_{F3}(\zeta_3,q)$, 
$S_{F3}(\zeta_5,q)$, and $S_{F3}(\zeta_7,q)$.
In Section 6 we sketch a proof that is independent of Theorem
\ref{TheoremSingleSeries} for the dissections for
$S_{AG4}(\zeta_5,q)$ and $S_{AG4}(\zeta_7,q)$. 
In Section 7 we use Theorems
\ref{TheoremSingleSeries} and \ref{TheoremDissections}
to prove Corollary \ref{TheoremEqualCranks}.
In Section 8 we give some concluding remarks, in particular we discuss some
additional Bailey pairs from \cite{Slater1} whose
spt-crank-type functions reduce to previous
functions after a change of variables.

\section{Proof of Series Identities}

The proof of these identities is to verify that the coefficients of each
power of $z$ on the left hand side and right hand side agree. This depends on
a identity of Garvan from \cite{Garvan2} to determine the coefficients of the 
powers of $z$ in the left hand side of the identities in Theorem 
\ref{TheoremSingleSeries}, and a variant of Bailey's lemma applied to one of
two general Bailey pairs to transform the coefficients of the powers of $z$.
The following is Proposition 4.1 of \cite{Garvan2},
\begin{align}\label{GarvanProp41}
	\frac{(1+z)\aqprod{z,z^{-1}}{q}{n}}{\aqprod{q}{q}{2n}}
	&=
	\sum_{j=-n}^{n+1}
	\frac{ (-1)^{j+1}(1-q^{2j-1})z^jq^{\frac{j(j-3)}{2}+1}}
	{\aqprod{q}{q}{n+j}\aqprod{q}{q}{n-j+1}}
.
\end{align}

We recall a limiting case of Bailey's Lemma \cite{Bailey} gives that if $(\alpha,\beta)$ is a Bailey 
pair relative to $(a,q)$ then
\begin{align*}
	\sum_{n=0}^\infty \aqprod{\rho_1,\rho_2}{q}{n} 
		\left(\frac{aq}{\rho_1\rho_2} \right)^n \beta_n
	&=
	\frac{\aqprod{aq/\rho_1,aq/\rho_2}{q}{\infty}}{\aqprod{aq,aq/(\rho_1\rho_2)}{q}{\infty}}
	\sum_{n=0}^\infty \frac{
		\aqprod{\rho_1,\rho_2}{q}{n} 
		\left(\frac{aq}{\rho_1\rho_2} \right)^n \alpha_n
		}{\aqprod{aq/\rho_1,aq/\rho_2}{q}{n}}	
.  
\end{align*}
For one of the variants of Bailey's Lemma, we need the conjugate
Bailey pair in the following lemma. We recall that a pair of sequences 
$(\delta,\gamma)$ is a conjugate Bailey pair relative to $(a,q)$ if
\begin{align*}
	\gamma_n &= \sum_{j=n}^\infty \frac{\delta_j}{\aqprod{q}{q}{j-n}\aqprod{aq}{q}{j+n}}
.
\end{align*}
Different conjugate Bailey pairs give rise to different variants of Bailey's Lemma 
because Bailey's Transform states that if $(\alpha,\beta)$ is a Bailey pair
relative to $(a,q)$ and $(\delta,\gamma)$ is a conjugate Bailey pair relative
to $(a,q)$ then
\begin{align*}
	\sum_{n=0}^\infty \beta_n\delta_n &= \sum_{n=0}^\infty \alpha_n\gamma_n
.
\end{align*}
\begin{lemma}\label{LemmaNewConjugateBaileyPairs}
The following is a conjugate Bailey pair relative to 
$(a,q)$,
\begin{align*}
	\delta^1_n
	&=
	\aqprod{z\sqrt{a}q^{-\frac{1}{2}}, z^{-1}\sqrt{a}q^{-\frac{1}{2}}}{q}{n}q^n
	,\\
	\gamma^1_n
	&=
	\frac{q^n \aqprod{z\sqrt{a}q^{-\frac{1}{2}}, z^{-1}\sqrt{a}q^{-\frac{1}{2}}}{q}{\infty}
			( 1- (z+z^{-1})\sqrt{a}q^{n+\frac{1}{2}} + aq^{2n})}
		{\aqprod{q,aq}{q}{\infty}(1-z\sqrt{a}q^{n-\frac{1}{2}})
		(1-z\sqrt{a}q^{n+\frac{1}{2}})(1-z^{-1}\sqrt{a}q^{n-\frac{1}{2}})
		(1-z^{-1}\sqrt{a}q^{n+\frac{1}{2}})}
.
\end{align*}
With $z=\omega$, a primitive third root of unity, and $a\not=q$ this
conjugate Bailey pair becomes
\begin{align*}
	\delta^2_n
	&=
	\frac{\aqprod{a^{\frac{3}{2}}q^{-\frac{3}{2}}}{q^3}{n}q^n}
		{\aqprod{\sqrt{a}q^{-\frac{1}{2}}}{q}{n}}
	,
	&\gamma^2_n
	&=
	\frac{q^n \aqprod{a^{\frac{3}{2}}q^{-\frac{3}{2}}}{q^3}{\infty}
		(1-\sqrt{a}q^{n-\frac{1}{2}})
		(1-\sqrt{a}q^{n+\frac{1}{2}})
		\left( 1+ \sqrt{a}q^{n+\frac{1}{2}} + aq^{2n} \right)
	}
	{\aqprod{q,aq,\sqrt{a}q^{-\frac{1}{2}}}{q}{\infty}
		(1-a^{\frac{3}{2}}q^{3n-\frac{3}{2}})(1-a^{\frac{3}{2}}q^{3n+\frac{3}{2}})}
	.
\end{align*}
\end{lemma}
\begin{proof}
We are to show that 
\begin{align*}
	&\frac{q^n \aqprod{z\sqrt{a}q^{-\frac{1}{2}}, z^{-1}\sqrt{a}q^{-\frac{1}{2}}}{q}{\infty}
		( 1- (z+z^{-1}) \sqrt{a}q^{n+\frac{1}{2}} + aq^{2n} )}
		{\aqprod{q,aq}{q}{\infty}(1-z\sqrt{a}q^{n-\frac{1}{2}})
		(1-z\sqrt{a}q^{n+\frac{1}{2}})(1-z^{-1}\sqrt{a}q^{n-\frac{1}{2}})
		(1-z^{-1}\sqrt{a}q^{n+\frac{1}{2}})}
	\\
	&=
	\sum_{j=n}^\infty 
		\frac{ \aqprod{z\sqrt{a}q^{-\frac{1}{2}}, z^{-1}\sqrt{a}q^{-\frac{1}{2}}}{q}{j}q^j }
		{\aqprod{q}{q}{j-n}\aqprod{aq}{q}{j+n}}
.
\end{align*}
Other than elementary rearrangements, we only need Heine's Transformation, which
can be found as Corollary 2.3 in \cite{AndrewsBook}.
We recall Heine's Transformation is
\begin{align*}
	\TwoPhiOne{a}{b}{c}{q}{z}
	&=
	\frac{\aqprod{c/b,bz}{q}{\infty}}{\aqprod{c,z}{q}{\infty}}
	\TwoPhiOne{abz/c}{b}{bz}{q}{c/b}
.
\end{align*}
We have
\begin{align*}
	&\sum_{j=n}^\infty 
		\frac{ \aqprod{z\sqrt{a}q^{-\frac{1}{2}}, z^{-1}\sqrt{a}q^{-\frac{1}{2}}}{q}{j}q^j }
		{\aqprod{q}{q}{j-n}\aqprod{aq}{q}{j+n}}
	\\
	&=
	\sum_{j=0}^\infty 
		\frac{ \aqprod{z\sqrt{a}q^{-\frac{1}{2}}, z^{-1}\sqrt{a}q^{-\frac{1}{2}}}{q}{j+n}q^{j+n} }
		{\aqprod{q}{q}{j}\aqprod{aq}{q}{j+2n}}
	\\
	&=
	\frac{q^n \aqprod{z\sqrt{a}q^{-\frac{1}{2}}, z^{-1}\sqrt{a}q^{-\frac{1}{2}}}{q}{n} }
		{\aqprod{aq}{q}{2n}}
	\sum_{j=0}^\infty 
		\frac{ \aqprod{z\sqrt{a}q^{n-\frac{1}{2}}, z^{-1}\sqrt{a}q^{n-\frac{1}{2}}}{q}{j}q^{j} }
		{\aqprod{q}{q}{j}\aqprod{aq^{2n+1}}{q}{j}}
	\\
	&=
	\frac{q^n \aqprod{z\sqrt{a}q^{-\frac{1}{2}}, z^{-1}\sqrt{a}q^{-\frac{1}{2}}}{q}{n} }
		{\aqprod{aq}{q}{2n}}
	\TwoPhiOne{z\sqrt{a}q^{n-\frac{1}{2}}}{z^{-1}\sqrt{a}q^{n-\frac{1}{2}}}{aq^{2n+1}}{q}{q}
	\\
	&=
	\frac{q^n \aqprod{z\sqrt{a}q^{-\frac{1}{2}}, z^{-1}\sqrt{a}q^{-\frac{1}{2}}}{q}{n} }
		{\aqprod{aq}{q}{2n}}
	\frac{\aqprod{z\sqrt{a}q^{n+\frac{3}{2}}, z^{-1}\sqrt{a}q^{n+\frac{1}{2}}}{q}{\infty}}
		{\aqprod{aq^{2n+1},q}{q}{\infty}}
	\TwoPhiOne{q^{-1}}{z^{-1}\sqrt{a}q^{n-\frac{1}{2}}}{z^{-1}\sqrt{a}q^{n+\frac{1}{2}}}{q}
		{z\sqrt{a}q^{n+\frac{3}{2}}}
	\\
	&=
	\frac{q^n \aqprod{z\sqrt{a}q^{-\frac{1}{2}}, z^{-1}\sqrt{a}q^{-\frac{1}{2}}}{q}{\infty}}
		{\aqprod{aq,q}{q}{\infty}(1-z\sqrt{a}q^{n-\frac{1}{2}})
			(1-z\sqrt{a}q^{n+\frac{1}{2}})(1-z^{-1}\sqrt{a}q^{n-\frac{1}{2}})}
	\left(
		1+ 
		\frac{(1-q^{-1})(1-z^{-1}\sqrt{a}q^{n-\frac{1}{2}})z\sqrt{a}q^{n+\frac{3}{2}}}
		{(1-q)(1-z^{-1}\sqrt{a}q^{n+\frac{1}{2}})}
	\right)
	\\
	&=
	\frac{q^n \aqprod{z\sqrt{a}q^{-\frac{1}{2}}, z^{-1}\sqrt{a}q^{-\frac{1}{2}}}{q}{\infty}
		\left( 1-  (z+z^{-1})\sqrt{a}q^{n+\frac{1}{2}} + aq^{2n} \right)	
		}
		{\aqprod{aq,q}{q}{\infty}(1-z\sqrt{a}q^{n-\frac{1}{2}})
		(1-z\sqrt{a}q^{n+\frac{1}{2}})(1-z^{-1}\sqrt{a}q^{n-\frac{1}{2}})
		(1-z^{-1}\sqrt{a}q^{n+\frac{1}{2}})}
.
\end{align*}

\end{proof}

\begin{lemma}
If $(\alpha,\beta)$ is a Bailey pair relative to $(a,q)$ then
\begin{align}
	\label{LemmaBailey1}
	&\sum_{n=0}^\infty 
	\aqprod{\sqrt{a}}{q}{n}(-1)^n a^{\frac{n}{2}}q^{\frac{n(n+1)}{2}}\beta_n
	=
	\frac{\aqprod{\sqrt{a}q}{q}{\infty}}{\aqprod{aq}{q}{\infty}}
	\sum_{n=0}^\infty
	\frac{(1-\sqrt{a})(-1)^na^{\frac{n}{2}}q^{\frac{n(n+1)}{2}}\alpha_n}{(1-\sqrt{a}q^n)}
	,\\
	\label{LemmaBailey2}
	&\sum_{n=0}^\infty 
	\aqprod{\sqrt{aq}}{q}{n}(-1)^n a^{\frac{n}{2}}q^{\frac{n^2}{2}}\beta_n
	=
	\frac{\aqprod{\sqrt{aq}}{q}{\infty}}{\aqprod{aq}{q}{\infty}}
	\sum_{n=0}^\infty
	(-1)^na^{\frac{n}{2}}q^{\frac{n^2}{2}}\alpha_n
	,\\
	\label{LemmaBailey3}
	&\sum_{n=0}^\infty
	\aqprod{z\sqrt{a}q^{-\frac{1}{2}}, z^{-1}\sqrt{a}q^{-\frac{1}{2}}}{q}{n}q^n\beta_n
		\nonumber\\
		&=
		\frac{\aqprod{z\sqrt{a}q^{-\frac{1}{2}}, z^{-1}\sqrt{a}q^{-\frac{1}{2}}}{q}{\infty}}
		{\aqprod{q,aq}{q}{\infty}}
		\sum_{n=0}^\infty
		\frac{ (1- (z+z^{-1})\sqrt{a}q^{n+\frac{1}{2}} + aq^{2n})q^n \alpha_n}
		{(1-z\sqrt{a}q^{n-\frac{1}{2}}) (1-z\sqrt{a}q^{n+\frac{1}{2}})
		(1-z^{-1}\sqrt{a}q^{n-\frac{1}{2}}) (1-z^{-1}\sqrt{a}q^{n+\frac{1}{2}})}
	,\\
	\label{LemmaBailey4}
	&\sum_{n=0}^\infty
	\aqprod{z\sqrt{a}q^{-\frac{1}{2}}, z^{-1}\sqrt{a}q^{-\frac{1}{2}}}{q}{n}q^{2n}\beta_n
		\nonumber\\
		&=
		\frac{\aqprod{z\sqrt{a}q^{-\frac{1}{2}}, z^{-1}\sqrt{a}q^{-\frac{1}{2}}}{q}{\infty}}
		{\aqprod{q,aq}{q}{\infty}}	
		\sum_{n=0}^\infty
		\frac{(1-q)q^{2n}  \alpha_n}
		{
		(1-z\sqrt{a}q^{n-\frac{1}{2}})(1-z\sqrt{a}q^{n+\frac{1}{2}})
		(1-z^{-1}\sqrt{a}q^{n-\frac{1}{2}})(1-z^{-1}\sqrt{a}q^{n+\frac{1}{2}})
		}
	,\\
	\label{LemmaBailey5}
	&\sum_{n=0}^\infty
	\frac{\aqprod{a^{\frac{3}{2}}q^{-\frac{3}{2}}}{q^3}{n}q^n\beta_n}
		{\aqprod{\sqrt{a}q^{-\frac{1}{2}}}{q}{n}}
		=
		\frac{\aqprod{a^{\frac{3}{2}}q^{-\frac{3}{2}}}{q^3}{\infty}}
			{\aqprod{q,aq,\sqrt{a}q^{-\frac{1}{2}}}{q}{\infty}}
		\sum_{n=0}^\infty
		\frac{
			(1-\sqrt{a}q^{n-\frac{1}{2}})(1-\sqrt{a}q^{n+\frac{1}{2}})
			( 1+ \sqrt{a}q^{n+\frac{1}{2}} + aq^{2n} )q^n \alpha_n}
		{(1-a^{\frac{3}{2}}q^{3n-\frac{3}{2}})(1-a^{\frac{3}{2}}q^{3n+\frac{3}{2}})}
	,\\
	\label{LemmaBailey6}
	&\sum_{n=0}^\infty	
	\frac{\aqprod{a^{\frac{3}{2}}q^{-\frac{3}{2}}}{q^3}{n}q^{2n}\beta_n}
		{\aqprod{\sqrt{a}q^{-\frac{1}{2}}}{q}{n}}
		=
		\frac{\aqprod{a^{\frac{3}{2}}q^{-\frac{3}{2}}}{q^3}{\infty}}
		{\aqprod{q,aq,\sqrt{a}q^{-\frac{1}{2}}}{q}{\infty}}
		\sum_{n=0}^\infty
		\frac{ (1-q)(1-\sqrt{a}q^{n-\frac{1}{2}})(1-\sqrt{a}q^{n+\frac{1}{2}})q^{2n}\alpha_n}
		{(1-a^{\frac{3}{2}}q^{3n-\frac{3}{2}})(1-a^{\frac{3}{2}}q^{3n+\frac{3}{2}})}
.
\end{align}
If $(\alpha,\beta)$ is a Bailey pair relative to $(a,q^2)$ then
\begin{align}
	\label{LemmaBailey7}
	&\sum_{n=0}^\infty
	\aqprod{-\sqrt{a}}{q}{2n}q^n\beta_n
		=
		\frac{\aqprod{-\sqrt{a}q}{q}{\infty}}{\aqprod{aq^2,q}{q^2}{\infty}}
		\sum_{n=0}^\infty \frac{(1+\sqrt{a})q^n\alpha_n}{(1+\sqrt{a}q^{2n})}
.
\end{align}
\end{lemma}
\begin{proof}
Equation (\ref{LemmaBailey1}) follows from Bailey's Lemma by letting
$\rho_1=\sqrt{a}$ and $\rho_2\rightarrow\infty$. Equation
(\ref{LemmaBailey2}) follows from Bailey's Lemma by letting
$\rho_1=\sqrt{aq}$ and $\rho_2\rightarrow\infty$.
Equation (\ref{LemmaBailey7}) follows from Bailey's Lemma by letting
$q\mapsto q^2$, $\rho_1=-\sqrt{a}$, and $\rho_2=-q\sqrt{a}$.
Equations (\ref{LemmaBailey3}) and (\ref{LemmaBailey5}) 
are Bailey's
Transform with the conjugate Bailey pairs of Lemma 
\ref{LemmaNewConjugateBaileyPairs}.
Equation (\ref{LemmaBailey4}) follows from Bailey's Lemma by letting
$\rho_1 = z\sqrt{a}q^{-\frac{1}{2}}$ and $\rho_2 = z^{-1}\sqrt{a}q^{-\frac{1}{2}}$.
Lastly, equation (\ref{LemmaBailey6}) follows by letting
$z=\omega$, a primitive third root of unity, in  (\ref{LemmaBailey4}).
\end{proof}

We use the following Bailey pairs relative to $(a,q)$,
\begin{align}\label{FirstBaileyPair}
	\beta^*_n(a,q) &= \frac{1}{\aqprod{aq,q}{q}{n}}
	,
	&\alpha^*_n(a,q) &= \PieceTwo{1}{0}{n=0}{n\ge 1}
	,
	\\
	\label{SecondBaileyPair}
	\beta^{**}_n(a,q) &= \frac{1}{\aqprod{aq^2,q}{q}{n}}
	,
	&\alpha^{**}_n(a,q) &= 
		\left\{
   		\begin{array}{ll}
      		1 & n=0 \\
       		-aq & n=1 \\
				0 & n\ge 2 
     		\end{array}
		\right.
.
\end{align}
That these are Bailey pairs relative to $(a,q)$ follows immediately from the
definition of a Bailey pair.


\begin{proof}[Proof of (\ref{TheoremSeriesForJ2})]
By (\ref{GarvanProp41}) we have that
\begin{align*}
	(1+z)\aqprod{z,z^{-1}}{q}{\infty}S_{J2}(z,q)
	&=
		\frac{\aqprod{q}{q}{\infty}^2}{\aqprod{q^3}{q^3}{\infty}}
		\sum_{n=1}^\infty 
		\frac{\aqprod{q^3}{q^3}{n-1}q^n(1+z)\aqprod{z,z^{-1}}{q}{n}}
			{\aqprod{q}{q}{2n}\aqprod{q}{q}{n-1}}
	\\
	&=
		\frac{\aqprod{q}{q}{\infty}^2}{\aqprod{q^3}{q^3}{\infty}}
		\sum_{n=1}^\infty 
		\frac{\aqprod{q^3}{q^3}{n-1}q^n}{\aqprod{q}{q}{n-1}}
		\sum_{j=-n}^{n+1}
		\frac{(-1)^{j+1}(1-q^{2j-1})z^jq^{\frac{j(j-3)}{2}+1}}
			{\aqprod{q}{q}{n+j}\aqprod{q}{q}{n-j+1}}
	.
\end{align*}
We note the coefficients of $z^{-j}$ and $z^{j+1}$ are then the same in
$(1+z)\aqprod{z,z^{-1}}{q}{\infty}S_{J2}(z,q)$, so we need only determine
the coefficients of $z^j$ for $j\ge 1$.
For $j\ge 2$ we see the coefficient of $z^j$ in 
$(1+z)\aqprod{z,z^{-1}}{q}{\infty}S_{J2}(z,q)$ is given by
\begin{align*}
	&\frac{\aqprod{q}{q}{\infty}^2}{\aqprod{q^3}{q^3}{\infty}}
	\sum_{n=j-1}^\infty 
	\frac{\aqprod{q^3}{q^3}{n-1} 
		(-1)^{j+1}(1-q^{2j-1})q^{n+\frac{j(j-3)}{2}+1}}
		{\aqprod{q}{q}{n-1}\aqprod{q}{q}{n+j}\aqprod{q}{q}{n-j+1}}
	\\
	&=
	\frac{\aqprod{q}{q}{\infty}^2}{\aqprod{q^3}{q^3}{\infty}}
	\sum_{n=0}^\infty 
	\frac{\aqprod{q^3}{q^3}{n+j-2} 
		(-1)^{j+1}(1-q^{2j-1})q^{n+\frac{j(j-1)}{2}}}
		{\aqprod{q}{q}{n+j-2}\aqprod{q}{q}{n+2j-1}\aqprod{q}{q}{n}}
	\\	
	&=
	\frac{\aqprod{q}{q}{\infty}^2(-1)^{j+1}\aqprod{q^3}{q^3}{j-2}(1-q^{2j-1})q^{\frac{j(j-1)}{2}}}
		{\aqprod{q^3}{q^3}{\infty}\aqprod{q}{q}{j-2}\aqprod{q}{q}{2j-1}}
	\sum_{n=0}^\infty 
	\frac{\aqprod{q^{3j-3}}{q^3}{n} q^n}
		{\aqprod{q^{j-1}}{q}{n}\aqprod{q^{2j}}{q}{n}\aqprod{q}{q}{n}}
	\\	
	&=
	\frac{\aqprod{q}{q}{\infty}^2(-1)^{j+1}\aqprod{q^3}{q^3}{j-2}(1-q^{2j-1})q^{\frac{j(j-1)}{2}}}
		{\aqprod{q^3}{q^3}{\infty}\aqprod{q}{q}{j-2}\aqprod{q}{q}{2j-1}}
	\sum_{n=0}^\infty
	\frac{\aqprod{q^{3j-3}}{q^3}{n} q^n\beta^*_n(q^{2j-1},q)}
		{\aqprod{q^{j-1}}{q}{n}}
	.
\end{align*}
We now apply (\ref{LemmaBailey5}) so that the coefficient of $z^j$ in 
$(1+z)\aqprod{z,z^{-1}}{q}{\infty}S_{J2}(z,q)$ is given by
\begin{align*}	
	&
	\frac{\aqprod{q}{q}{\infty}^2(-1)^{j+1}\aqprod{q^3}{q^3}{j-2}(1-q^{2j-1})q^{\frac{j(j-1)}{2}}
		\aqprod{q^{3j-3}}{q^3}{\infty}(1-q^{j-1})(1-q^j)(1+q^j+q^{2j-1})}
		{\aqprod{q^3}{q^3}{\infty}\aqprod{q}{q}{j-2}\aqprod{q}{q}{2j-1}
			\aqprod{q,q^{2j},q^{j-1}}{q}{\infty}(1-q^{3j-3})(1-q^{3j})
		}
	\\	
	&=
	\frac{(-1)^{j+1}(1-q^{2j-1})q^{\frac{j(j-1)}{2}} (1-q^{j-1})(1-q^j)(1+q^j+q^{2j-1})}
		{\aqprod{q}{q}{\infty}(1-q^{3j-3})(1-q^{3j})}
	\\	
	&=
	\frac{(-1)^{j+1}(1-q^{2j-1})q^{\frac{j(j-1)}{2}} (1+q^j+q^{2j-1})}
		{\aqprod{q}{q}{\infty}(1-\omega q^{j-1})(1-\omega^{-1}q^{j-1})(1-\omega q^j)(1-\omega^{-1}q^j)}
.
\end{align*}
The calculations are similar for the coefficient of $z$, except that
we use (\ref{LemmaBailey3}). In particular, we have
that the coefficient of $z$ in 
$(1+z)\aqprod{z,z^{-1}}{q}{\infty}S_{J2}(z,q)$
is given by
\begin{align*}
	&\frac{\aqprod{q}{q}{\infty}^2}{\aqprod{q^3}{q^3}{\infty}}
	\sum_{n=1}^\infty 
	\frac{\aqprod{q^3}{q^3}{n-1}q^{n}(1-q)}
		{\aqprod{q}{q}{n-1}\aqprod{q}{q}{n+1}\aqprod{q}{q}{n}}
	\\
	&=\frac{\aqprod{q}{q}{\infty}^2}{\aqprod{q^3}{q^3}{\infty}}
	\sum_{n=1}^\infty 
	\frac{\aqprod{\omega q, \omega^{-1}q}{q}{n-1}q^n(1-q)}
		{\aqprod{q}{q}{n+1}\aqprod{q}{q}{n}}
	\\
	&=
	\frac{\aqprod{q}{q}{\infty}^2}{\aqprod{q^3}{q^3}{\infty}}
	\sum_{n=0}^\infty 
	\frac{\aqprod{\omega q, \omega^{-1}q}{q}{n-1}q^n(1-q)}
		{\aqprod{q}{q}{n+1}\aqprod{q}{q}{n}}
	-\frac{\aqprod{q}{q}{\infty}^2}{3\aqprod{q^3}{q^3}{\infty}}
	\\
	&=
	\frac{\aqprod{q}{q}{\infty}^2}{\aqprod{q^3}{q^3}{\infty}(1-\omega)(1-\omega^{-1})}
	\sum_{n=0}^\infty 
	\frac{\aqprod{\omega, \omega^{-1}}{q}{n}q^n}
		{\aqprod{q^2}{q}{n}\aqprod{q}{q}{n}}
	-\frac{\aqprod{q}{q}{\infty}^2}{3\aqprod{q^3}{q^3}{\infty}}
	\\
	&=
	\frac{\aqprod{q}{q}{\infty}^2}{\aqprod{q^3}{q^3}{\infty}(1-\omega)(1-\omega^{-1})}
	\sum_{n=0}^\infty 
	\aqprod{\omega, \omega^{-1}}{q}{n}q^n\beta^*_n(q,q)
	-\frac{\aqprod{q}{q}{\infty}^2}{3\aqprod{q^3}{q^3}{\infty}}
	\\
	&=
	\frac{\aqprod{q}{q}{\infty}^2 \aqprod{\omega,\omega^{-1}}{q}{\infty}(1+2q)}
		{\aqprod{q^3}{q^3}{\infty}\aqprod{q^2,q}{q}{\infty}
		(1-\omega)^2(1-\omega^{-1})^2		
		(1-\omega q)(1-\omega^{-1}q)}
	-\frac{\aqprod{q}{q}{\infty}^2}{3\aqprod{q^3}{q^3}{\infty}}
	\\	
	&=
	\frac{(1-q)(1+2q)}
		{\aqprod{q}{q}{\infty}(1-\omega)(1-\omega^{-1})(1-\omega q)(1-\omega^{-1}q)}
	-\frac{\aqprod{q}{q}{\infty}^2}{3\aqprod{q^3}{q^3}{\infty}}
.
\end{align*}
Thus,
\begin{align*}
	&(1+z)\aqprod{z,z^{-1}}{q}{\infty}S_{J2}(z,q)
	\\
	&=
	-\frac{(1+z)}{3}\frac{\aqprod{q}{q}{\infty}^2}{\aqprod{q^3}{q^3}{\infty}}
	+
	\frac{1}{\aqprod{q}{q}{\infty}}
	\sum_{j=1}^\infty 
	\frac{(z^j+z^{1-j})(-1)^{j+1}q^{\frac{j(j-1)}{2}}(1-q^{2j-1})(1+q^j+q^{2j-1})}  
		{(1-\omega q^{j-1})(1-\omega^{-1}q^{j-1})(1-\omega q^{j})(1-\omega^{-1}q^{j})}
.
\end{align*}
Here setting $z=1$ yields
\begin{align*}
	\frac{1}{3}\frac{\aqprod{q}{q}{\infty}^2}{\aqprod{q^3}{q^3}{\infty}}
	&=
	\frac{1}{\aqprod{q}{q}{\infty}}
	\sum_{j=1}^\infty 
	\frac{(-1)^{j+1}q^{\frac{j(j-1)}{2}}(1-q^{2j-1})(1+q^j+q^{2j-1})}  
		{(1-\omega q^{j-1})(1-\omega^{-1}q^{j-1})(1-\omega q^{j})(1-\omega^{-1}q^{j})}
,
\end{align*}
so in fact
\begin{align*}
	&(1+z)\aqprod{z,z^{-1}}{q}{\infty}S_{J2}(z,q)
	\\
	&=\frac{1}{\aqprod{q}{q}{\infty}}
	\sum_{j=1}^\infty 
	\frac{(1-z^j)(1-z^{j-1})z^{1-j}(-1)^{j+1}q^{\frac{j(j-1)}{2}}(1-q^{2j-1})(1+q^j+q^{2j-1})}  
		{(1-\omega q^{j-1})(1-\omega^{-1}q^{j-1})(1-\omega q^{j})(1-\omega^{-1}q^{j})}
	\\
	&=
	\frac{1}{\aqprod{q}{q}{\infty}}
	\sum_{j=2}^\infty 
	\frac{(1-z^j)(1-z^{j-1})z^{1-j}(-1)^{j+1}q^{\frac{j(j-1)}{2}}(1-q^{2j-1})(1+q^j+q^{2j-1})}  
		{(1-\omega q^{j-1})(1-\omega^{-1}q^{j-1})(1-\omega q^{j})(1-\omega^{-1}q^{j})}
	\\
	&=
	\frac{1}{\aqprod{q}{q}{\infty}}
	\sum_{j=2}^\infty 
	\frac{(1-z^j)(1-z^{j-1})z^{1-j}(-1)^{j+1}q^{\frac{j(j-1)}{2}}(1-q^{2j-1})(1+q^j+q^{2j-1})
		(1-q^{j-1})(1-q^j)}  
		{(1-q^{3j-3})(1-q^{3j})}
	\\
	&=
	\frac{1}{\aqprod{q}{q}{\infty}}
	\sum_{j=2}^\infty 
	\frac{(1-z^j)(1-z^{j-1})z^{1-j}(-1)^{j+1}q^{\frac{j(j-1)}{2}}( 1-q^{j-1}-q^{2j}+q^{4j-1}+q^{5j-3}-q^{6j-3})}  
		{(1-q^{3j-3})(1-q^{3j})}
.
\end{align*}
\end{proof}

\begin{proof}[Proof of (\ref{TheoremSeriesForJ3})]
By (\ref{GarvanProp41}) we have that
\begin{align*}
	(1+z)\aqprod{z,z^{-1}}{q}{\infty}S_{J3}(z,q)
	&=
		\frac{\aqprod{q}{q}{\infty}^2}{\aqprod{q^3}{q^3}{\infty}}
		\sum_{n=1}^\infty 
		\frac{\aqprod{q^3}{q^3}{n-1}q^{2n}(1+z)\aqprod{z,z^{-1}}{q}{n}}
			{\aqprod{q}{q}{2n}\aqprod{q}{q}{n-1}}
	\\
	&=
		\frac{\aqprod{q}{q}{\infty}^2}{\aqprod{q^3}{q^3}{\infty}}
		\sum_{n=1}^\infty 
		\frac{ \aqprod{q^3}{q^3}{n-1}q^{2n}}{\aqprod{q}{q}{n-1}}
		\sum_{j=-n}^{n+1}
		\frac{(-1)^{j+1}(1-q^{2j-1})z^jq^{\frac{j(j-3)}{2}+1}}
			{\aqprod{q}{q}{n+j}\aqprod{q}{q}{n-j+1}}
	.
\end{align*}
We note the coefficients of $z^{-j}$ and $z^{j+1}$ are then the same in
$(1+z)\aqprod{z,z^{-1}}{q}{\infty}S_{J3}(z,q)$, so we need only determine
the coefficients of $z^j$ for $j\ge 1$.
The proof is now the same as it was for $S_{J2}(z,q)$, except that we use
(\ref{LemmaBailey6}).
For $j\ge 2$ we see the coefficient of $z^j$ in 
$(1+z)\aqprod{z,z^{-1}}{q}{\infty}S_{J3}(z,q)$ is given by
\begin{align*}
	&\frac{\aqprod{q}{q}{\infty}^2}{\aqprod{q^3}{q^3}{\infty}}
	\sum_{n=j-1}^\infty 
	\frac{\aqprod{q^3}{q^2}{n-1} 
		(-1)^{j+1}(1-q^{2j-1})q^{2n+\frac{j(j-3)}{2}+1}}
		{\aqprod{q}{q}{n-1}\aqprod{q}{q}{n+j}\aqprod{q}{q}{n-j+1}}
	\\
	&=
	\frac{\aqprod{q}{q}{\infty}^2}{\aqprod{q^3}{q^3}{\infty}}
	\sum_{n=0}^\infty 
	\frac{\aqprod{q^3}{q^3}{n+j-2} 
		(-1)^{j+1}(1-q^{2j-1})q^{2n+\frac{j(j+1)}{2} -1}}
		{\aqprod{q}{q}{n+j-2}\aqprod{q}{q}{n+2j-1}\aqprod{q}{q}{n}}
	\\	
	&=
	\frac{\aqprod{q}{q}{\infty}^2(-1)^{j+1}\aqprod{q^3}{q^3}{j-2}(1-q^{2j-1})q^{\frac{j(j+1)}{2}-1}}
		{\aqprod{q^3}{q^3}{\infty}\aqprod{q}{q}{j-2}\aqprod{q}{q}{2j-1}}
	\sum_{n=0}^\infty 
	\frac{\aqprod{q^{3j-3}}{q^3}{n} q^{2n}}
		{\aqprod{q^{j-1}}{q}{n}\aqprod{q^{2j}}{q}{n}\aqprod{q}{q}{n}}
	\\	
	&=
	\frac{\aqprod{q}{q}{\infty}^2(-1)^{j+1}\aqprod{q^3}{q^3}{j-2}(1-q^{2j-1})q^{\frac{j(j+1)}{2}-1}}
		{\aqprod{q^3}{q^3}{\infty}\aqprod{q}{q}{j-2}\aqprod{q}{q}{2j-1}}
	\sum_{n=0}^\infty 
	\frac{\aqprod{q^{3j-3}}{q^3}{n} q^{2n}\beta^*_n(q^{2j-1},q)}
		{\aqprod{q^{j-1}}{q}{n}}
	\\
	&=
		\frac{\aqprod{q}{q}{\infty}^2(-1)^{j+1}\aqprod{q^3}{q^3}{j-2}
			(1-q^{2j-1})q^{\frac{j(j+1)}{2} -1}
			\aqprod{q^{3j-3}}{q^3}{\infty}(1-q)(1-q^{j-1})(1-q^{j})
		}
		{\aqprod{q^3}{q^3}{\infty}	\aqprod{q}{q}{j-2}\aqprod{q}{q}{2j-1}
			\aqprod{q,q^{2j},q^{j-1}}{q}{\infty}(1-q^{3j-3})(1-q^{3j})}
	\\
	&=
	\frac{(-1)^{j+1}(1-q^{2j-1})(1-q)(1-q^{j-1})(1-q^j)q^{\frac{j(j+1)}{2} -1}}
	{\aqprod{q}{q}{\infty}(1-q^{3j-3})(1-q^{3j})}
	\\
	&=
	\frac{(-1)^{j+1}(1-q^{2j-1})(1-q)q^{\frac{j(j+1)}{2} -1}}
	{\aqprod{q}{q}{\infty}(1-\omega q^{j-1})(1-\omega^{-1}q^{j-1})
		(1-\omega q^{j})(1-\omega q^{j-1})  }
.
\end{align*}
The calculations are similar for the coefficient of $z$, but we
use (\ref{LemmaBailey4}). In particular, we have
that the coefficient of $z$ in 
$(1+z)\aqprod{z,z^{-1}}{q}{\infty}S_{J3}(z,q)$
is given by
\begin{align*}
	&\frac{\aqprod{q}{q}{\infty}^2}{\aqprod{q^3}{q^3}{\infty}}
	\sum_{n=1}^\infty 
	\frac{\aqprod{q^3}{q^3}{n-1}q^{2n}(1-q)}
		{\aqprod{q}{q}{n-1}\aqprod{q}{q}{n+1}\aqprod{q}{q}{n}}
	\\
	&=\frac{\aqprod{q}{q}{\infty}^2}{\aqprod{q^3}{q^3}{\infty}}
	\sum_{n=1}^\infty 
	\frac{\aqprod{\omega q, \omega^{-1}q}{q}{n-1}q^{2n}(1-q)}
		{\aqprod{q}{q}{n+1}\aqprod{q}{q}{n}}
	\\
	&=
	\frac{\aqprod{q}{q}{\infty}^2}{\aqprod{q^3}{q^3}{\infty}}
	\sum_{n=0}^\infty 
	\frac{\aqprod{\omega q, \omega^{-1}q}{q}{n-1}q^{2n}(1-q)}
		{\aqprod{q}{q}{n+1}\aqprod{q}{q}{n}}
	-\frac{\aqprod{q}{q}{\infty}^2}{3\aqprod{q^3}{q^3}{\infty}}
	\\
	&=
	\frac{\aqprod{q}{q}{\infty}^2}{\aqprod{q^3}{q^3}{\infty}(1-\omega)(1-\omega^{-1})}
	\sum_{n=0}^\infty 
	\frac{\aqprod{\omega, \omega^{-1}}{q}{n}q^{2n}}
		{\aqprod{q^2}{q}{n}\aqprod{q}{q}{n}}
	-\frac{\aqprod{q}{q}{\infty}^2}{3\aqprod{q^3}{q^3}{\infty}}
	\\
	&=
	\frac{\aqprod{q}{q}{\infty}^2}{\aqprod{q^3}{q^3}{\infty}(1-\omega)(1-\omega^{-1})}
	\sum_{n=0}^\infty \aqprod{\omega,\omega^{-1}}{q}{n}q^{2n}\beta^*_n(q,q)
	-\frac{\aqprod{q}{q}{\infty}^2}{3\aqprod{q^3}{q^3}{\infty}}
	\\
	&=
	\frac{\aqprod{q}{q}{\infty}^2 \aqprod{\omega,\omega^{-1}}{q}{\infty}(1-q)}
		{\aqprod{q^3}{q^3}{\infty}\aqprod{q,q^2}{q}{\infty}(1-\omega )^2(1-\omega^{-1})^2(1-\omega q)(1-\omega^{-1} q)}
	-\frac{\aqprod{q}{q}{\infty}^2}{3\aqprod{q^3}{q^3}{\infty}}
	\\	
	&=
	\frac{(1-q)^2}
		{\aqprod{q}{q}{\infty}(1-\omega)(1-\omega^{-1})(1-\omega q)(1-\omega^{-1}q)}
	-\frac{\aqprod{q}{q}{\infty}^2}{3\aqprod{q^3}{q^3}{\infty}}
.
\end{align*}
Thus,
\begin{align*}
	&(1+z)\aqprod{z,z^{-1}}{q}{\infty}S_{J3}(z,q)
	\\
	&=
	-\frac{(1+z)}{3}\frac{\aqprod{q}{q}{\infty}^2}{\aqprod{q^3}{q^3}{\infty}}
	+
	\frac{1}{\aqprod{q}{q}{\infty}}
	\sum_{j=1}^\infty 
	\frac{(z^j+z^{1-j})(-1)^{j+1}q^{\frac{j(j+1)}{2}-1}(1-q^{2j-1})(1-q)}  
		{(1-\omega q^{j-1})(1-\omega^{-1}q^{j-1})(1-\omega q^{j})(1-\omega^{-1}q^{j})}
.
\end{align*}
Here setting $z=1$ yields
\begin{align*}
	\frac{1}{3}\frac{\aqprod{q}{q}{\infty}^2}{\aqprod{q^3}{q^3}{\infty}}
	&=
	\frac{1}{\aqprod{q}{q}{\infty}}
	\sum_{j=1}^\infty 
	\frac{(-1)^{j+1}q^{\frac{j(j+1)}{2}-1}(1-q^{2j-1})(1-q)}  
		{(1-\omega q^{j-1})(1-\omega^{-1}q^{j-1})(1-\omega q^{j})(1-\omega^{-1}q^{j})}
,
\end{align*}
so in fact
\begin{align*}
	&(1+z)\aqprod{z,z^{-1},q}{q}{\infty}S_{J3}(z,q)
	\\
	&=
	\sum_{j=1}^\infty 
	\frac{(1-z^j)(1-z^{j-1})z^{1-j}(-1)^{j+1}q^{\frac{j(j+1)-1}{2}}(1-q^{2j-1})(1-q)}  
		{(1-\omega q^{j-1})(1-\omega^{-1}q^{j-1})(1-\omega q^{j})(1-\omega^{-1}q^{j})}
	\\
	&=
	\sum_{j=2}^\infty 
	\frac{(1-z^j)(1-z^{j-1})z^{1-j}(-1)^{j+1}q^{\frac{j(j+1)}{2}-1}
			(1-q-q^{j-1}-q^{j+1} + q^{3j-2} - q^{4j-2} - q^{3j} + q^{4j-1})}  
		{(1-q^{3j-3})(1-q^{3j})}
	\\
	&=
	\sum_{j=2}^\infty 
	\frac{(1-z^j)(1-z^{j-1})z^{1-j}(-1)^{j+1}q^{\frac{j(j-1)}{2}}
			(q^{j-1}-q^{j}-q^{2j-2}+q^{2j}+q^{4j-3}-q^{4j-1} -q^{5j-3} + q^{5j-2})}  
		{(1-q^{3j-3})(1-q^{3j})}
.
\end{align*}
\end{proof}

\begin{proof}[Proof of (\ref{TheoremSeriesForF3})]
By (\ref{GarvanProp41}) we have that
\begin{align*}
	(1+z)\aqprod{z,z^{-1}}{q^2}{\infty}
	S_{F3}(z,q)
	&=
	\aqprod{q}{q}{\infty}
	\sum_{n=1}^\infty	
	\frac{(1+z)\aqprod{z,z^{-1}}{q^2}{n}q^{n}}
		{\aqprod{q}{q}{2n}}
	\\
	&=
	\aqprod{q}{q}{\infty}
	\sum_{n=1}^\infty	
	\frac{(1+z)\aqprod{z,z^{-1}}{q^2}{n}q^{n}\aqprod{-q}{q}{2n}}
		{\aqprod{q^2}{q^2}{2n}}
	\\
	&=
	\aqprod{q}{q}{\infty}
	\sum_{n=1}^\infty	q^{n}\aqprod{-q}{q}{2n}
	\sum_{j=-n}^{n+1} \frac{(-1)^{j+1}(1-q^{4j-2})z^jq^{j(j-3)+2}}
		{\aqprod{q^2}{q^2}{n+j}\aqprod{q^2}{q^2}{n-j+1}}	
	.
\end{align*}
We note the coefficients of $z^{-j}$ and $z^{j+1}$ are then the same in
$(1+z)\aqprod{z,z^{-1}}{q^2}{\infty}S_{F3}(z,q)$, so we need only determine
the coefficients of $z^j$ for $j\ge 1$.
This time we will use (\ref{LemmaBailey7}).
For $j\ge 2$, the coefficient of $z^j$ in 
$(1+z)\aqprod{z,z^{-1}}{q^2}{\infty}S_{F3}(z,q)$
is given by
\begin{align*}
	&
	\aqprod{q}{q}{\infty}
	\sum_{n=j-1}^\infty \frac{q^n\aqprod{-q}{q}{2n}(-1)^{j+1}(1-q^{4j-2})q^{j(j-3)+2}}
		{\aqprod{q^2}{q^2}{n+j}\aqprod{q^2}{q^2}{n-j+1}}
	\\
	&=
	\aqprod{q}{q}{\infty}(-1)^{j+1}(1-q^{4j-2})q^{j(j-3)+2}
	\sum_{n=0}^\infty \frac{q^{n+j-1}\aqprod{-q}{q}{2n+2j-2}}
		{\aqprod{q^2}{q^2}{n+2j-1}\aqprod{q^2}{q^2}{n}}
	\\
	&=
	\frac{\aqprod{q}{q}{\infty}(-1)^{j+1}(1-q^{4j-2})q^{(j-1)^2}\aqprod{-q}{q}{2j-2}}
		{\aqprod{q^2}{q^2}{2j-1}}
	\sum_{n=0}^\infty \frac{q^{n}\aqprod{-q^{2j-1}}{q}{2n}}
		{\aqprod{q^{4j}}{q^2}{n}\aqprod{q^2}{q^2}{n}}
	\\
	&=
	\frac{\aqprod{q}{q}{\infty}(-1)^{j+1}(1-q^{4j-2})q^{(j-1)^2}\aqprod{-q}{q}{2j-2}}
		{\aqprod{q^2}{q^2}{2j-1}}
	\sum_{n=0}^\infty q^{n}\aqprod{-q^{2j-1}}{q}{2n}\beta^*_n(q^{4j-2},q^2)
	\\
	&=
	\frac{\aqprod{q}{q}{\infty}(-1)^{j+1}(1-q^{4j-2})q^{(j-1)^2}
		\aqprod{-q}{q}{2j-2}\aqprod{-q^{2j}}{q}{\infty}}
	{\aqprod{q^2}{q^2}{2j-1}\aqprod{q^{4j},q}{q^2}{\infty}}
	\\
	&=
	\frac{(-1)^{j+1}(1-q^{2j-1})q^{(j-1)^2}}{\aqprod{q}{q^2}{\infty}}
.
\end{align*}
The calculations for the coefficient of $z$ are similar and we still use
(\ref{LemmaBailey7}). In particular, the coefficient of $z$ in
$(1+z)\aqprod{z,z^{-1}}{q^2}{\infty}S_{F3}(z,q)$ is given by
\begin{align*}
	\aqprod{q}{q}{\infty}
	\sum_{n=1}^\infty \frac{q^n\aqprod{-q}{q}{2n}(1-q^{2})}
		{\aqprod{q^2}{q^2}{n+1}\aqprod{q^2}{q^2}{n}}
	&=
	\aqprod{q}{q}{\infty}
	\sum_{n=0}^\infty \frac{q^n\aqprod{-q}{q}{2n}}
		{\aqprod{q^4}{q^2}{n}\aqprod{q^2}{q^2}{n}}
	-\aqprod{q}{q}{\infty}
	\\
	&=
	\aqprod{q}{q}{\infty}
	\sum_{n=0}^\infty q^n\aqprod{-q}{q}{2n}\beta^*_n(q^2,q^2)
	-\aqprod{q}{q}{\infty}
	\\
	&=
	\frac{\aqprod{q}{q}{\infty}\aqprod{-q^2}{q}{\infty}}
	{\aqprod{q^4,q}{q^2}{\infty}}
	\\
	&=
	\frac{(1-q)}{\aqprod{q}{q^2}{\infty}}
.
\end{align*}
Thus
\begin{align*}
	&(1+z)\aqprod{z,z^{-1},q}{q^2}{\infty}S_{F3}(z,q)
	\\
	&=
		-(1+z)\aqprod{q}{q}{\infty}\aqprod{q}{q^2}{\infty}
		+
		\sum_{j=1}^\infty (z^j+z^{1-j})(-1)^{j+1}(1-q^{2j-1})q^{(j-1)^2}
	\\
	&=
		-(1+z)\aqprod{q}{q}{\infty}\aqprod{q}{q^2}{\infty}
		+
		\sum_{j=1}^\infty (z^j + z^{1-j}) (-1)^{j+1} q^{(j-1)^2}
		+
		\sum_{j=1}^\infty (z^j + z^{1-j}) (-1)^{j+1} q^{j^2}
	\\
	&=
		-(1+z)\aqprod{q}{q}{\infty}\aqprod{q}{q^2}{\infty}
		+
		\sum_{j=-\infty}^\infty (z^j + z^{1-j}) (-1)^{j+1} q^{(j-1)^2}
.
\end{align*}
However, we have by Gauss that
\begin{align*}
	\aqprod{q}{q}{\infty}\aqprod{q}{q^2}{\infty}
	&=
	\sum_{j=-\infty}^\infty (-1)^jq^{j^2}	
	=
	\sum_{j=-\infty}^\infty (-1)^{j+1}q^{(j-1)^2}	
.
\end{align*}
We then have that
\begin{align*}
	(1+z)\aqprod{z,z^{-1},q}{q^2}{\infty}S_{F3}(z,q)
	&=
		\sum_{j=-\infty}^\infty (z^j + z^{1-j}-1-z) (-1)^{j+1} q^{(j-1)^2}
	\\
	&=
		\sum_{j=-\infty}^\infty (1-z^{j-1})(1-z^j)z^{1-j} (-1)^{j+1} q^{(j-1)^2}
	.
\end{align*}

\end{proof}

\begin{proof}[Proof of (\ref{TheoremSeriesForG4})]
By (\ref{GarvanProp41}) we have that
\begin{align*}
	(1+z)\aqprod{z,z^{-1}}{q^2}{\infty}
	S_{G4}(z,q)
	&=
	\frac{\aqprod{q^2}{q^2}{\infty}}{\aqprod{q}{q^2}{\infty}}
	\sum_{n=1}^\infty	
	\frac{(1+z)\aqprod{z,z^{-1}}{q^2}{n}(-1)^nq^{n^2+2n}}
		{\aqprod{-q}{q^2}{n}\aqprod{q^4}{q^4}{n}}
	\\
	&=
	\frac{\aqprod{q^2}{q^2}{\infty}}{\aqprod{q}{q^2}{\infty}}
	\sum_{n=1}^\infty	
	\frac{(1+z)\aqprod{z,z^{-1}}{q^2}{n}(-1)^nq^{n^2+2n}\aqprod{q}{q^2}{n}}
		{\aqprod{q^2}{q^2}{2n}}
	\\
	&=
	\frac{\aqprod{q^2}{q^2}{\infty}}{\aqprod{q}{q^2}{\infty}}
	\sum_{n=1}^\infty	
	(-1)^nq^{n^2+2n}\aqprod{q}{q^2}{n}
	\sum_{j=-n}^{n+1} \frac{(-1)^{j+1}(1-q^{4j-2})z^jq^{j(j-3)+2}}
		{\aqprod{q^2}{q^2}{n+j}\aqprod{q^2}{q^2}{n-j+1}}	
	.
\end{align*}
We note the coefficients of $z^{-j}$ and $z^{j+1}$ are then the same in
$(1+z)\aqprod{z,z^{-1}}{q^2}{\infty}S_{G4}(z,q)$, so we need only determine
the coefficients of $z^j$ for $j\ge 1$.
This time we will use (\ref{LemmaBailey1}).
For $j\ge 2$, the coefficient of $z^j$ in 
$(1+z)\aqprod{z,z^{-1}}{q^2}{\infty}S_{G4}(z,q)$
is given by
\begin{align*}
	&
	\frac{\aqprod{q^2}{q^2}{\infty}}{\aqprod{q}{q^2}{\infty}}
	\sum_{n=j-1}^\infty	
	\frac{(-1)^{j+n+1}q^{n^2+2n}\aqprod{q}{q^2}{n}(1-q^{4j-2})q^{j(j-3)+2}}
	{\aqprod{q^2}{q^2}{n+j}\aqprod{q^2}{q^2}{n-j+1}}
	\\
	&=
	\frac{\aqprod{q^2}{q^2}{\infty}(1-q^{4j-2})q^{2j^2-3j+1}}
		{\aqprod{q}{q^2}{\infty}}	
	\sum_{n=0}^\infty	
	\frac{(-1)^{n}q^{n^2+2jn}\aqprod{q}{q^2}{n+j-1}}
	{\aqprod{q^2}{q^2}{n+2j-1}\aqprod{q^2}{q^2}{n}}
	\\
	&=
	\frac{\aqprod{q^2}{q^2}{\infty} \aqprod{q}{q^2}{j-1}(1-q^{4j-2})q^{2j^2-3j+1}}
	{\aqprod{q}{q^2}{\infty}\aqprod{q^2}{q^2}{2j-1}}
	\sum_{n=0}^\infty	
	\frac{(-1)^{n}q^{n^2+2jn}\aqprod{q^{2j-1}}{q^2}{n}}
	{\aqprod{q^{4j}}{q^2}{n}\aqprod{q^2}{q^2}{n}}
	\\
	&=
	\frac{\aqprod{q^2}{q^2}{\infty} \aqprod{q}{q^2}{j-1}(1-q^{4j-2})q^{2j^2-3j+1}}
	{\aqprod{q}{q^2}{\infty}\aqprod{q^2}{q^2}{2j-1}}
	\sum_{n=0}^\infty	
	(-1)^{n}q^{n^2+2jn}\aqprod{q^{2j-1}}{q^2}{n}\beta^*_n(q^{4j-2},q^2)
	\\
	&=
	\frac{\aqprod{q^2}{q^2}{\infty} \aqprod{q}{q^2}{j-1}(1-q^{4j-2})q^{2j^2-3j+1}}
	{\aqprod{q}{q^2}{\infty}\aqprod{q^2}{q^2}{2j-1}}
	\frac{\aqprod{q^{2j+1}}{q^2}{\infty}}{\aqprod{q^{4j}}{q^2}{\infty}}
	\\
	&=
	(1+q^{2j-1})q^{2j^2-3j+1}
.
\end{align*}
The calculations for the coefficient of $z$ are similar and we still use
(\ref{LemmaBailey1}). In particular, the coefficient of $z$ in
$(1+z)\aqprod{z,z^{-1}}{q^2}{\infty}S_{G4}(z,q)$ is given by
\begin{align*}
	\frac{\aqprod{q^2}{q^2}{\infty}}{\aqprod{q}{q^2}{\infty}}
	\sum_{n=1}^\infty \frac{(-1)^nq^{n^2+2n}\aqprod{q}{q^2}{n}(1-q^2)}
		{\aqprod{q^2}{q^2}{n+1}\aqprod{q^2}{q^2}{n}}
	&=
	\frac{\aqprod{q^2}{q^2}{\infty}}{\aqprod{q}{q^2}{\infty}}
	\sum_{n=0}^\infty \frac{(-1)^nq^{n^2+2n}\aqprod{q}{q^2}{n}}
		{\aqprod{q^4}{q^2}{n}\aqprod{q^2}{q^2}{n}}
	-
	\frac{\aqprod{q^2}{q^2}{\infty}}{\aqprod{q}{q^2}{\infty}}
	\\
	&=
	\frac{\aqprod{q^2}{q^2}{\infty}}{\aqprod{q}{q^2}{\infty}}
	\sum_{n=0}^\infty (-1)^nq^{n^2+2n}\aqprod{q}{q^2}{n}\beta^*_n(q^2;q^2)
	-
	\frac{\aqprod{q^2}{q^2}{\infty}}{\aqprod{q}{q^2}{\infty}}
	\\	
	&=
	\frac{\aqprod{q^2}{q^2}{\infty}}{\aqprod{q}{q^2}{\infty}}
	\frac{\aqprod{q^3}{q^2}{\infty}}{\aqprod{q^4}{q^2}{\infty}}
	-
	\frac{\aqprod{q^2}{q^2}{\infty}}{\aqprod{q}{q^2}{\infty}}
	\\
	&=
	1+q	
	-
	\frac{\aqprod{q^2}{q^2}{\infty}}{\aqprod{q}{q^2}{\infty}}
	.
\end{align*}
Thus
\begin{align*}
	(1+z)\aqprod{z,z^{-1}}{q^2}{\infty}S_{G4}(z,q)
	&=
	-(1+z)\frac{\aqprod{q^2}{q^2}{\infty}}{\aqprod{q}{q^2}{\infty}}
	+
	\sum_{j=1}^\infty (z^j + z^{1-j})(1+q^{2j-1})q^{2j^2-3j+1}
	\\
	&=
	-(1+z)\frac{\aqprod{q^2}{q^2}{\infty}}{\aqprod{q}{q^2}{\infty}}
	+
	\sum_{j=-\infty}^\infty z^j(1+q^{2j-1})q^{2j^2-3j+1}
	\\
	&=
	-(1+z)\frac{\aqprod{q^2}{q^2}{\infty}}{\aqprod{q}{q^2}{\infty}}
	+
	\sum_{j=-\infty}^\infty z^jq^{2j^2-3j+1}
	+
	\sum_{j=-\infty}^\infty z^jq^{2j^2-j}
	\\
	&=
	-(1+z)\frac{\aqprod{q^2}{q^2}{\infty}}{\aqprod{q}{q^2}{\infty}}
	+
	\sum_{j=-\infty}^\infty z^{1-j}q^{2j^2-j}
	+
	\sum_{j=-\infty}^\infty z^jq^{2j^2-j}
	\\
	&=
	-(1+z)\frac{\aqprod{q^2}{q^2}{\infty}}{\aqprod{q}{q^2}{\infty}}
	+
	\sum_{j=-\infty}^\infty (z^j+z^{1-j})q^{2j^2-j}
.
\end{align*}
However, we have by Gauss that
\begin{align}
	\label{IdGaussProduct}
	\frac{\aqprod{q^2}{q^2}{\infty}}{\aqprod{q}{q^2}{\infty}}
	&=
	\frac{1}{2}\sum_{n=-\infty}^\infty q^{n(n+1)/2}
	\nonumber\\
	&=
	\frac{1}{2}\left(
		\sum_{n=-\infty}^\infty q^{n(2n+1)}
		+
	\sum_{n=-\infty}^\infty q^{(2n-1)n}
	\right)
	\nonumber\\
	&= 
	\sum_{n=-\infty}^\infty q^{2n^2-n}
.
\end{align}
We then have that
\begin{align*}
	(1+z)\aqprod{z,z^{-1}}{q^2}{\infty}S_{G4}(z,q)
	&=
	\sum_{j=-\infty}^\infty (z^j+z^{1-j}-1-z)q^{2j^2-j}
	\\
	&=
	\sum_{j=-\infty}^\infty (1-z^{j-1})(1-z^j)z^{1-j}q^{2j^2-j}
.
\end{align*}
\end{proof}

\begin{proof}[Proof of (\ref{TheoremSeriesForAG4})]
We have by (\ref{GarvanProp41}) that
\begin{align*}
	(1+z)\aqprod{z,z^{-1}}{q^2}{\infty}
	S_{AG4}(z,q)
	&=
	\frac{\aqprod{q^2}{q^2}{\infty}}{\aqprod{q}{q^2}{\infty}}
	\sum_{n=1}^\infty	
	\frac{(1+z)\aqprod{z,z^{-1}}{q^2}{n}(-1)^nq^{n^2}}
		{\aqprod{-q}{q^2}{n}\aqprod{q^4}{q^4}{n}}
	\\
	&=
	\frac{\aqprod{q^2}{q^2}{\infty}}{\aqprod{q}{q^2}{\infty}}
	\sum_{n=1}^\infty	
	\frac{(1+z)\aqprod{z,z^{-1}}{q^2}{n}(-1)^nq^{n^2}\aqprod{q}{q^2}{n}}
		{\aqprod{q^2}{q^2}{2n}}
	\\
	&=
	\frac{\aqprod{q^2}{q^2}{\infty}}{\aqprod{q}{q^2}{\infty}}
	\sum_{n=1}^\infty	
	(-1)^nq^{n^2}\aqprod{q}{q^2}{n}
	\sum_{j=-n}^{n+1} \frac{(-1)^{j+1}(1-q^{4j-2})z^jq^{j(j-3)+2}}
		{\aqprod{q^2}{q^2}{n+j}\aqprod{q^2}{q^2}{n-j+1}}	
	.
\end{align*}
We note the coefficients of $z^{-j}$ and $z^{j+1}$ are then the same in
$(1+z)\aqprod{z,z^{-1}}{q^2}{\infty}S_{AG4}(z,q)$, so we need only determine
the coefficients of $z^j$ for $j\ge 1$.
This time we use (\ref{LemmaBailey2}).
For $j\ge 2$, the coefficient of $z^j$ in 
$(1+z)\aqprod{z,z^{-1}}{q^2}{\infty}S_{AG4}(z,q)$
is given by
\begin{align*}
	&
	\frac{\aqprod{q^2}{q^2}{\infty}}{\aqprod{q}{q^2}{\infty}}
	\sum_{n=j-1}^\infty	
	\frac{(-1)^{j+n+1}q^{n^2}\aqprod{q}{q^2}{n}(1-q^{4j-2})q^{j(j-3)+2}}
	{\aqprod{q^2}{q^2}{n+j}\aqprod{q^2}{q^2}{n-j+1}}
	\\
	&=
	\frac{\aqprod{q^2}{q^2}{\infty}(1-q^{4j-2})q^{2j^2-5j+3}}
		{\aqprod{q}{q^2}{\infty}}	
	\sum_{n=0}^\infty	
	\frac{(-1)^{n}q^{n^2+2jn-2n}\aqprod{q}{q^2}{n+j-1}}
	{\aqprod{q^2}{q^2}{n+2j-1}\aqprod{q^2}{q^2}{n}}
	\\
	&=
	\frac{\aqprod{q^2}{q^2}{\infty} \aqprod{q}{q^2}{j-1}(1-q^{4j-2})q^{2j^2-5j+3}}
	{\aqprod{q}{q^2}{\infty}\aqprod{q^2}{q^2}{2j-1}}
	\sum_{n=0}^\infty	
	\frac{(-1)^{n}q^{n^2+2jn-2n}\aqprod{q^{2j-1}}{q^2}{n}}
	{\aqprod{q^{4j}}{q^2}{n}\aqprod{q^2}{q^2}{n}}
	\\
	&=
	\frac{\aqprod{q^2}{q^2}{\infty} \aqprod{q}{q^2}{j-1}(1-q^{4j-2})q^{2j^2-5j+3}}
	{\aqprod{q}{q^2}{\infty}\aqprod{q^2}{q^2}{2j-1}}
	\sum_{n=0}^\infty	
	(-1)^{n}q^{n^2+2jn-2n}\aqprod{q^{2j-1}}{q^2}{n}\beta^{**}_n(q^{4j-4},q^2)
	\\
	&=
	\frac{\aqprod{q^2}{q^2}{\infty} \aqprod{q}{q^2}{j-1}(1-q^{4j-2})q^{2j^2-5j+3}}
	{\aqprod{q}{q^2}{\infty}\aqprod{q^2}{q^2}{2j-1}}
	\frac{\aqprod{q^{2j-1}}{q^2}{\infty}}{\aqprod{q^{4j-2}}{q^2}{\infty}}
	\sum_{n=0}^\infty (-1)^n q^{n^2+2jn-2n}\alpha^{**}_n(q^{4j-4},q^2)
	\\
	&=
	q^{2j^2-5j+3}(1+q^{6j-3})	
.
\end{align*}
Similarly, the coefficient of $z$ in
$(1+z)\aqprod{z,z^{-1}}{q^2}{\infty}S_{AG4}(z,q)$ 
is given by
\begin{align*}
	\frac{\aqprod{q^2}{q^2}{\infty}}{\aqprod{q}{q^2}{\infty}}
	\sum_{n=1}^\infty \frac{(-1)^nq^{n^2}\aqprod{q}{q^2}{n}(1-q^2)}
		{\aqprod{q^2}{q^2}{n+1}\aqprod{q^2}{q^2}{n}}
	&=
	\frac{\aqprod{q^2}{q^2}{\infty}}{\aqprod{q}{q^2}{\infty}}
	\sum_{n=0}^\infty \frac{(-1)^nq^{n^2}\aqprod{q}{q^2}{n}}
		{\aqprod{q^4}{q^2}{n}\aqprod{q^2}{q^2}{n}}
	-
	\frac{\aqprod{q^2}{q^2}{\infty}}{\aqprod{q}{q^2}{\infty}}
	\\
	&=
	\frac{\aqprod{q^2}{q^2}{\infty}}{\aqprod{q}{q^2}{\infty}}
	\sum_{n=0}^\infty (-1)^nq^{n^2}\aqprod{q}{q^2}{n}\beta^{**}_n(1;q^2)
	-
	\frac{\aqprod{q^2}{q^2}{\infty}}{\aqprod{q}{q^2}{\infty}}
	\\
	&=
	\frac{\aqprod{q^2}{q^2}{\infty}}{\aqprod{q}{q^2}{\infty}}
	\frac{\aqprod{q}{q^2}{\infty}}{\aqprod{q^2}{q^2}{\infty}}
	(1+q^3)
	-
	\frac{\aqprod{q^2}{q^2}{\infty}}{\aqprod{q}{q^2}{\infty}}
	\\
	&=
	1+q^3
	-
	\frac{\aqprod{q^2}{q^2}{\infty}}{\aqprod{q}{q^2}{\infty}}
	.
\end{align*}
Along with (\ref{IdGaussProduct}), this gives
\begin{align*}
	(1+z)\aqprod{z,z^{-1}}{q^2}{\infty}S_{AG4}(z,q)
	&=
	-(1+z)\frac{\aqprod{q^2}{q^2}{\infty}}{\aqprod{q}{q^2}{\infty}}
	+
	\sum_{j=1}^\infty (z^j + z^{1-j})(1+q^{6j-3})q^{2j^2-5j+3}
	\\
	&=
	-(1+z)\frac{\aqprod{q^2}{q^2}{\infty}}{\aqprod{q}{q^2}{\infty}}
	+
	\sum_{j=-\infty}^\infty z^j(1+q^{6j-3})q^{2j^2-5j+3}
	\\
	&=
	-(1+z)\frac{\aqprod{q^2}{q^2}{\infty}}{\aqprod{q}{q^2}{\infty}}
	+
	\sum_{j=-\infty}^\infty z^jq^{2j^2-5j+3}
	+
	\sum_{j=-\infty}^\infty z^jq^{2j^2+j}
	\\
	&=
	-(1+z)\frac{\aqprod{q^2}{q^2}{\infty}}{\aqprod{q}{q^2}{\infty}}
	+
	\sum_{j=-\infty}^\infty z^{1-j}q^{2j^2+j}
	+
	\sum_{j=-\infty}^\infty z^jq^{2j^2+j}
	\\
	&=
	\sum_{j=-\infty}^\infty (z^j+z^{1-j}-1-z)q^{2j^2+j}
	\\
	&=
	\sum_{j=-\infty}^\infty (1-z^{j-1})(1-z^j)z^{1-j}q^{2j^2+j}
.
\end{align*}
\end{proof}

\section{Dissections for $S_{B2}(z,q)$}

To begin, by Bailey's Lemma with $\rho_1=z$ and $\rho_2=z^{-1}$ we have that
\begin{align*}
	(1-z)(1-z^{-1})S_{B2}(z,q)
	&=
	\frac{\aqprod{q}{q}{\infty}}{\aqprod{zq,z^{-1}q}{q}{\infty}}
	\sum_{n=0}^\infty \frac{\aqprod{z,z^{-1}}{q}{n} q^{2n}}{\aqprod{q}{q}{n}}
	-\frac{\aqprod{q}{q}{\infty}}{\aqprod{zq,z^{-1}q}{q}{\infty}}
	\\
	&=
	\frac{\aqprod{q}{q}{\infty}}{\aqprod{zq,z^{-1}q}{q}{\infty}}
	\sum_{n=0}^\infty \aqprod{z,z^{-1}}{q}{n} q^{n}\beta^{B2}_n
	-
	\frac{\aqprod{q}{q}{\infty}}{\aqprod{zq,z^{-1}q}{q}{\infty}}
	\\
	&=
	\frac{1}{\aqprod{q}{q}{\infty}}
	\left(1+
		\sum_{n=1}^\infty \frac{(1-z)(1-z^{-1})(-1)^nq^{(3n^2-n)/2}(1+q^{3n})}
		{(1-zq^n)(1-z^{-1}q^n)}
	\right)
	-\frac{\aqprod{q}{q}{\infty}}{\aqprod{zq,z^{-1}q}{q}{\infty}}
.
\end{align*}
While the series term is not the generating function for the rank of 
partitions, it is surprisingly close to it.
We recall the rank of a partition is the largest part minus the number of 
parts. One form of the generating function for the rank of partitions, which is
given on page 64 of \cite{Watson}, is
\begin{align}
	\label{EqRankDef}
	R(z,q)
	&=
	\frac{1}{\aqprod{q}{q}{\infty}}
	\left(
	1+\sum_{n=1}^\infty\frac{(1-z)(1-z^{-1})(-1)^n q^{n(3n+1)/2}(1+q^{n})}
		{(1-zq^n)(1-z^{-1}q^n)}
	\right)
.
\end{align}
We recall the crank of a partition is the largest part, if there are no ones,
and otherwise is the number of parts larger than the number of ones minus the 
number of ones. One form of the generating function for the crank of partitions,
which is given in (7.15) of \cite{Garvan1},
is
\begin{align}
	\label{EqCrankDef}
	C(z,q)
	&=
	\frac{\aqprod{q}{q}{\infty}}{\aqprod{zq,z^{-1}q}{q}{\infty}}
.
\end{align}

\begin{lemma}\label{LemmaB2ToRank}
\begin{align*}
	\frac{1}{\aqprod{q}{q}{\infty}}
	\left(
	1+\sum_{n=1}^\infty\frac{(1-z)(1-z^{-1})(-1)^n q^{n(3n-1)/2}(1+q^{3n})}
		{(1-zq^n)(1-z^{-1}q^n)}
	\right)
	&=
	(z+z^{-1}-1)R(z,q) + (1-z)(1-z^{-1})
.
\end{align*}
\end{lemma}
\begin{proof}
To prove this identity, we multiply both sides by $\aqprod{q}{q}{\infty}$
and expand $\aqprod{q}{q}{\infty}$ into a series by Euler's pentagonal number
theorem. 
We then have
\begin{align*}
	&\aqprod{q}{q}{\infty}((z+z^{-1}-1)R(z,q) + (1-z)(1-z^{-1}))
	\\
	&=
	z+z^{-1}-1 + (1-z)(1-z^{-1})\aqprod{q}{q}{\infty}
	+
	\sum_{n=1}^\infty\frac{(1-z)(1-z^{-1})(-1)^nq^{n(3n+1)/2}(1+q^n)(z+z^{-1}-1)}
		{(1-zq^n)(1-z^{-1}q^n)}
	\\
	&=
	1+
	\sum_{n=1}^\infty (1-z)(1-z^{-1})(-1)^nq^{n(3n-1)/2}(1+q^n)
	\left(
		\frac{q^n(z+z^{-1}-1)}{(1-zq^n)(1-z^{-1}q^n)}
		+1
 	\right)
	\\		
	&=
	1+
	\sum_{n=1}^\infty \frac{(1-z)(1-z^{-1})(-1)^nq^{n(3n-1)/2}(1+q^n)(1-q^n+q^{2n})}
		{(1-zq^n)(1-z^{-1}q^n)}
	\\		
	&=
	1+
	\sum_{n=1}^\infty \frac{(1-z)(1-z^{-1})(-1)^nq^{n(3n-1)/2}(1+q^{3n})}
		{(1-zq^n)(1-z^{-1}q^n)}
	.
\end{align*}
This proves the lemma.
\end{proof}
With Lemma \ref{LemmaB2ToRank} we now have
\begin{align}\label{B2RankCrankId}
	S_{B2}(z,q)
	&=
	\frac{(z+z^{-1}-1)R(z,q)-C(z,q)}{(1-z)(1-z^{-1})}+1
.
\end{align}
Using the rank difference formulas from \cite{AS}, we can deduce the following
dissections for the rank function.
Theorem 4 of \cite{AS} gives the following dissection for $R(\zeta_5,q)$,
which can also be found as Entry 2.1.2 in \cite{AndrewsBerndt3},
\begin{align}\label{RankZeta5Id}
	R(\zeta_5,q)
	&=
	\frac{\aqprod{q^{25}}{q^{25}}{\infty}\jacprod{q^{10}}{q^{25}}}
		{\jacprod{q^{5}}{q^{25}}^2}
	+
	q^5\frac{\zeta_5+\zeta_5^{-1}-2}{\aqprod{q^{25}}{q^{25}}{\infty}}
		\sum_{n=-\infty}^\infty \frac{(-1)^nq^{75n(n+1)/2}}{1-q^{25n+5}}
	+
	q\frac{\aqprod{q^{25}}{q^{25}}{\infty}}{\jacprod{q^5}{q^{25}}}
	\nonumber\\&\quad
	+
	q^2(\zeta_5+\zeta_5^{-1})\frac{\aqprod{q^{25}}{q^{25}}{\infty}}
		{\jacprod{q^{10}}{q^{25}}}
	-
	q^3(\zeta_5+\zeta_5^{-1})\frac{\aqprod{q^{25}}{q^{25}}{\infty}\jacprod{q^5}{q^{25}}}
		{\jacprod{q^{10}}{q^{25}}^2}
	\nonumber\\&\quad
	-
	q^8\frac{2\zeta_5+2\zeta_5^{-1}+1}{\aqprod{q^{25}}{q^{25}}{\infty}}
		\sum_{n=-\infty}^\infty \frac{(-1)^nq^{75n(n+1)/2}}{1-q^{25n+10}}
.
\end{align}
Similarly
Theorem 5 of \cite{AS} gives the following dissection for $R(\zeta_7,q)$,
which is also Entry 2.1.5 of \cite{AndrewsBerndt3}.
\begin{align}\label{RankZeta7Id}
	&R(\zeta_7,q)
	\nonumber\\
	&=
		(1-\zeta_7)(1-\zeta_7^6)
		+
		(-1+\zeta_7+\zeta_7^6)\frac{\aqprod{q^{49}}{q^{49}}{\infty}\jacprod{q^{21}}{q^{49}}}
			{\jacprod{q^{7},q^{14}}{q^{49}}}
		\nonumber\\&\quad
		+
		(2-\zeta_7-\zeta_7^6)q^7\frac{1}{\aqprod{q^{49}}{q^{49}}{\infty}}
			\sum_{n=-\infty}^\infty \frac{(-1)^nq^{147n(n+1)/2}}{1-q^{49n+7}}
		+
		q\frac{\aqprod{q^{49}}{q^{49}}{\infty}}
			{\jacprod{q^{7}}{q^{49}}}
		\nonumber\\&\quad
		+
		(\zeta_7+\zeta_7^6)q^2\frac{\aqprod{q^{49}}{q^{49}}{\infty}\jacprod{q^{14}}{q^{49}}}
			{\jacprod{q^{7},q^{21}}{q^{49}}}	
		+
		(\zeta_7-\zeta_7^2-\zeta_7^5+\zeta_7^6)q^{16}\frac{1}{\aqprod{q^{49}}{q^{49}}{\infty}}
			\sum_{n=-\infty}^\infty \frac{(-1)^nq^{147n(n+1)/2}}{1-q^{49n+21}}
		\nonumber\\&\quad
		+
		(1+\zeta_7^2+\zeta_7^5)q^3\frac{\aqprod{q^{49}}{q^{49}}{\infty}}
			{\jacprod{q^{14}}{q^{49}}}
		-
		(\zeta_7^2+\zeta_7^5)q^4\frac{\aqprod{q^{49}}{q^{49}}{\infty}}
			{\jacprod{q^{21}}{q^{49}}}
		\nonumber\\&\quad
		+
		(1+\zeta_7+2\zeta_7^2+2\zeta_7^5+\zeta_7^6)q^{13}\frac{1}{\aqprod{q^{49}}{q^{49}}{\infty}}
			\sum_{n=-\infty}^\infty \frac{(-1)^nq^{147n(n+1)/2}}{1-q^{49n+14}}
		\nonumber\\&\quad
		+
		(\zeta_7+\zeta_7^2+\zeta_7^5+\zeta_7^6)q^6
			\frac{\aqprod{q^{49}}{q^{49}}{\infty}\jacprod{q^{7}}{q^{49}}}
			{\jacprod{q^{14},q^{21}}{q^{49}}}
.
\end{align}
Next by (3.8) of \cite{Garvan1}
\begin{align}\label{CrankZeta5Id}
	\frac{\aqprod{q}{q}{\infty}}{\aqprod{\zeta_5q,\zeta_5^{-1}q}{q}{\infty}}	
	&=
	\aqprod{q^{25}}{q^{25}}{\infty}
	\left(
		\frac{\jacprod{q^{10}}{q^{25}}}{\jacprod{q^{5}}{q^{25}}^2}
		+(\zeta_5+\zeta_5^4-1)q\frac{1}{\jacprod{q^{5}}{q^{25}}}
		-(\zeta_5+\zeta_5^4+1)q^2\frac{1}{\jacprod{q^{10}}{q^{25}}}
		\right.\nonumber\\&\left.\quad
		-(\zeta_5+\zeta_5^4)q^3\frac{\jacprod{q^{5}}{q^{25}}}{\jacprod{q^{10}}{q^{25}}^2}
	\right)
.
\end{align}
Also by Theorem 5.1 of \cite{Garvan1} we have 
\begin{align}\label{CrankZeta7Id}
	\frac{\aqprod{q}{q}{\infty}}
		{\aqprod{\zeta_7q,\zeta_7^{-1}q}{q}{\infty}}
	&=
	\aqprod{q^{49}}{q^{49}}{\infty}
	\left(
		\frac{\jacprod{q^{21}}{q^{49}}}{\jacprod{q^{7},q^{14}}{q^{49}}}
		+(\zeta_7+\zeta_7^6-1)q\frac{1}{\jacprod{q^{7}}{q^{49}}}
		+(\zeta_7^2+\zeta_7^5)q^2\frac{\jacprod{q^{14}}{q^{49}}}{\jacprod{q^7,q^{21}}{q^{49}}}
		\right.\nonumber\\&\left.\quad
		-(\zeta_7+\zeta_7^2+\zeta_7^5+\zeta_7^6)q^3\frac{1}{\jacprod{q^{14}}{q^{49}}}
		-(\zeta_7+\zeta_7^6)q^4\frac{1}{\jacprod{q^{21}}{q^{49}}}
		\right.\nonumber\\&\left.\quad
		-(\zeta_7^2+\zeta_7^5+1)q^6\frac{\jacprod{q^{7}}{q^{49}}}{\jacprod{q^{14},q^{21}}{q^{49}}}
	\right)
.
\end{align}
We then find (\ref{Dissection5B2}) of Theorem \ref{TheoremDissections} follows by (\ref{B2RankCrankId}),
(\ref{RankZeta5Id}), and (\ref{CrankZeta5Id}) and 
(\ref{Dissection7B2}) of Theorem \ref{TheoremDissections} follows by (\ref{B2RankCrankId}),
(\ref{RankZeta7Id}), and (\ref{CrankZeta7Id}).


\section{Dissections of $S_{F3}(z,q)$}

By Bailey's Lemma with $\rho_1=z$ and $\rho_2=z^{-1}$ we have that
\begin{align*}
	(1-z)(1-z^{-1})S_{F3}(z,q)
	&=
	\frac{\aqprod{q}{q}{\infty}}{\aqprod{zq^2,z^{-1}q^2}{q^2}{\infty}}
	\sum_{n=0}^\infty\frac{\aqprod{z,z^{-1}}{q^2}{n}q^n}{\aqprod{q}{q}{2n}}
	-
	\frac{\aqprod{q}{q}{\infty}}{\aqprod{zq^2,z^{-1}q^2}{q^2}{\infty}}
	\\
	&=
	\frac{\aqprod{q}{q^2}{\infty}}{\aqprod{q^2}{q^2}{\infty}}
	\left(
		1
		+
		\sum_{n=1}^\infty\frac{(1-z)(1-z^{-1})q^n(1+q^{2n})}
		{(1-zq^{2n})(1-z^{-1}q^{2n})}
	\right)
	-
	\frac{\aqprod{q}{q}{\infty}}{\aqprod{zq^2,z^{-1}q^2}{q}{\infty}}
	.
\end{align*}

We first find the dissections for the product term and then proceed with the 
series term. When $z=\zeta_7$ we use the theory of modular
functions, both for the product term and the series term. 
We recall some facts about modular functions as in \cite{Rankin} and use 
the notation in \cite{Berndt} and \cite{Robins}. 
The generalized eta function is defined by
\begin{align*}
	\GEta{\delta}{g}{\tau} 
	&= 
	q^{P(g/\delta)\delta/2 }
	\prod_{\substack{n>0\\n\equiv g \pmod{d}}} (1-q^n)
	\prod_{\substack{n>0\\n\equiv -g \pmod{d}}} (1-q^n)
,
\end{align*}
where $q = e^{2\pi i \tau}$ and
$P(t) = \CBrackets{t}^2-\CBrackets{t}+\frac{1}{6}$. 
So $\GEta{\delta}{0}{\tau}=q^{\delta/12}\aqprod{q^\delta}{q^\delta}{\infty}^2$ and 
$\GEta{\delta}{g}{\tau} = q^{P(g/\delta)\delta/2 }\jacprod{q^g}{q^{\delta}}$
for $0<g<\delta$.
We use Theorem 3 of
\cite{Robins} to determine when a quotient of $\GEta{\delta}{g}{\tau}$
is a modular function with respect to a congruence subgroup 
$\Gamma_1(N)$ and use Theorem 4 of \cite{Robins} to determine the order
at the cusps.
Suppose $f$ is a modular function with respect to the congruence subgroup $\Gamma$ 
of $\Gamma_0(1)$. For $A\in\Gamma_0(1)$ we have a cusp given 
by $\zeta=A^{-1}\infty$. The width of the cusp $W:=W(\Gamma,\zeta)$ is
given by
\begin{align*}
	W(\Gamma,\zeta) &= \min\{ k>0:\pm A^{-1}T^kA\in\Gamma \},
\end{align*}
where $T$ is the translation matrix 
\begin{align*}
	T &= {\Parans{\begin{array}{cc}
				1&1\\
				0&1
			\end{array}}}
.
\end{align*}

If
\begin{align*}
	f(A^{-1}\tau) &= \sum_{m=m_0}^\infty b_mq^{m/W}
\end{align*}
and $b_{m_0}\not=0$, then we say $m_0$ is the order of $f$ at
$\zeta$ with respect to $\Gamma$ and we denote this value 
by $Ord_\Gamma(f;\zeta)$. By
$ord(f;\zeta)$ we mean the invariant order of $f$ at $\zeta$
given by
\begin{align*}
	ord(f;\zeta) = \frac{Ord_\Gamma(f;\zeta)}{W}.
\end{align*}

For $z$ in the upper half plane $\mathcal{H}$, we write 
$ord(f;z)$ for the order of $f$ at $z$ as an analytic function
in $z$. We define the order of $f$ at $z$ with respect to
$\Gamma$ by
\begin{align*}
	Ord_\Gamma(f;z) = \frac{ord(f;z)}{m},
\end{align*}
where $m$ is the order of $z$ as a fixed point of $\Gamma$.

The valence formula for modular functions is as follows.
Suppose a subset $\mathcal{F}$ of 
$\mathcal{H}\cup\{\infty\}\cup\mathbb{Q}$ is a fundamental
region for the action of $\Gamma$ along with a complete set of
inequivalent cusps, if $f$ is not the zero 
function then
\begin{align}
	\sum_{z\in\mathcal{F}}Ord_\Gamma(f;z) &=0.
\end{align}
We can verify an identity between sums of generalized eta quotients as follows.
Suppose we are to show
\begin{align*}
	a_1f_1 + a_2f_2 + \dots + a_kf_k &=	a_{k+1}f_{k+1} + a_{k+2}f_{k+2} + \dots + a_{k+m}f_{k+m}
,\end{align*}
where each $a_i\in\mathbb{C}$ and each $f_i$ is of the form
\begin{align*}
	f_i &= \prod_{j=1}^{n}  \GEta{\delta_j}{g_j}{\tau}^{r_j}   
.
\end{align*}
We verify each $f_i$ is a modular function with respect to a common
$\Gamma_1(N)$, so that $f = a_1f_1+\dots+a_kf_k - a_{k+1}f_{k+1}-\dots-a_{k+m}f_{k+m}$ 
is a modular function
with respect to $\Gamma_1(N)$. Although $f$ may have zeros at points other than
the cusps, the poles must occur only at the cusps.
At each cusp $\zeta$, not equivalent to $\infty$, we compute a lower bound
for $Ord_\Gamma(f;\zeta)$ by taking the minimum of the $Ord_\Gamma(f_i,\zeta)$
, we call this lower bound $B_\zeta$. We then use the $q$-expansion
of $f$ to find that $Ord_\Gamma(f;\infty)$ is larger than 
$-\sum_{\zeta\in\mathcal{C}'} B_\zeta$, where $\mathcal{C}'$ is a set 
of cusps with a representative of each cusp not equivalent to $\infty$. By the
valence formula we have $f\equiv 0$ since 
$\sum_{z\in\mathcal{F}}Ord_\Gamma(f;z) >0$.

\begin{proposition}
\begin{align}
	\label{EqF3Mod3CrankId}
	\frac{\aqprod{q,q^2}{q^2}{\infty}}
		{\aqprod{\zeta_3q^2,\zeta_3^{-1}q^2}{q^2}{\infty}}
	&=
	\frac{\aqprod{q^9}{q^9}{\infty}^4}
	{\aqprod{q^{18}}{q^{18}}{\infty}^2\aqprod{q^3}{q^3}{\infty}}
	-
	q\frac{\aqprod{q^{18}}{q^{18}}{\infty}\aqprod{q^9}{q^9}\infty{}}
	{\aqprod{q^{6}}{q^{6}}{\infty}}
	-
	2q^2\frac{\aqprod{q^{18}}{q^{18}}{\infty}^4\aqprod{q^3}{q^3}{\infty}}
	{\aqprod{q^{9}}{q^{9}}{\infty}^2\aqprod{q^6}{q^6}{\infty}^2}
	,\\
	\label{EqF3Mod5CrankId}
	\frac{\aqprod{q,q^2}{q^2}{\infty}}{\aqprod{\zeta_5q^2,\zeta_5^4q^2}{q^2}{\infty}}
	&=
		\frac{\aqprod{q^{25}}{q^{25}}{\infty}\jacprod{q^{15}}{q^{50}}}
			{\jacprod{q^5}{q^{25}}} 	
		-
		q\frac{\aqprod{q^{25}}{q^{25}}{\infty}}
			{\jacprod{q^{10}}{q^{50}}}
		+		
		(\zeta_5+\zeta_5^4)q^2
		\frac{\aqprod{q^{25}}{q^{25}}{\infty}\jacprod{q^{10}}{q^{25}}}
			{\jacprod{q^5}{q^{25}}\jacprod{q^{20}}{q^{50}}}
		\nonumber\\&\quad
		-
		q^2\frac{\aqprod{q^{25}}{q^{25}}{\infty}\jacprod{q^5}{q^{25}}}
			{\jacprod{q^{10}}{q^{25}}\jacprod{q^{10}}{q^{50}}}
		-
		(\zeta_5+\zeta_5^4)
		q^3\frac{\aqprod{q^{25}}{q^{25}}{\infty}}
			{\jacprod{q^{20}}{q^{50}}}
		\nonumber\\&\quad
		-
		(\zeta_5+\zeta_5^4)q^4
			\frac{\aqprod{q^{25}}{q^{25}}{\infty}\jacprod{q^5}{q^{50}}}
				{\jacprod{q^{10}}{q^{25}}}
	,\\
	\label{EqF3Mod7CrankId}
	\frac{\aqprod{q,q^2}{q^2}{\infty}}{\aqprod{\zeta_7q^2,\zeta_7^{-1}q^2}{q^2}{\infty}}
	&=
		\frac{\aqprod{q^{49}}{q^{49}}{\infty}}{\jacprod{q^{14}}{q^{98}}}
		-
		(\zeta_7^2+\zeta_7^5)q\frac{\aqprod{q^{49}}{q^{49}}{\infty}\jacprod{q^{35}}{q^{98}}}
			{\jacprod{q^{21}}{q^{49}}}
		\nonumber\\&\quad
		+
		(1+\zeta_7^2+\zeta_7^5)q\frac{\aqprod{q^{49}}{q^{49}}{\infty}\jacprod{q^{21}}{q^{98}}}
			{\jacprod{q^{7}}{q^{49}}}
		-
		2q\frac{\aqprod{q^{98}}{q^{98}}{\infty}\jacprod{q^{35}}{q^{98}}}
			{\aqprod{q^{49}}{q^{98}}{\infty}\jacprod{q^{14}}{q^{98}}}
		\nonumber\\&\quad
		-
		(1-\zeta_7-\zeta_7^6)
		q^2\frac{\aqprod{q^{49}}{q^{49}}{\infty}\jacprod{q^{14}}{q^{49}}}
			{\jacprod{q^{7}}{q^{49}}\jacprod{q^{28}}{q^{98}}}
		-
		(\zeta_7+\zeta_7^6)
		q^3\frac{\aqprod{q^{49}}{q^{49}}{\infty}\jacprod{q^{7}}{q^{49}}}
			{\jacprod{q^{21}}{q^{49}}\jacprod{q^{14}}{q^{98}}}
		\nonumber\\&\quad
		+
		(1+\zeta_7^2+\zeta_7^5)
		q^4\frac{\aqprod{q^{49}}{q^{49}}{\infty}}{\jacprod{q^{28}}{q^{98}}}
		-
		(\zeta_7+\zeta_7^2+\zeta_7^5+\zeta_7^6)
		q^5\frac{\aqprod{q^{49}}{q^{49}}{\infty}\jacprod{q^{21}}{q^{49}}}
			{\jacprod{q^{14}}{q^{49}}\jacprod{q^{42}}{q^{98}}}
		\nonumber\\&\quad
		-
		(\zeta_7^2+\zeta_7^5)
		q^6\frac{\aqprod{q^{49}}{q^{49}}{\infty}}{\jacprod{q^{42}}{q^{98}}}	
.
\end{align}
\end{proposition}
\begin{proof}
Equation (\ref{EqF3Mod3CrankId}) follows from
Theorem 2.11 of \cite{GarvanJennings} by replacing $q$ by $-q$ 
and simplifying the products.
Similarly (\ref{EqF3Mod5CrankId}) follows from Theorem 2.12 of 
\cite{GarvanJennings} with $q$ replaced by $-q$.

We recognize the left hand side of (\ref{EqF3Mod7CrankId}) 
as $\aqprod{q}{q^2}{\infty}C(\zeta_7,q^2)$, where $C(z,q)$ is defined 
in (\ref{EqCrankDef}).
For $C(\zeta_7,q^2)$, we use
(\ref{CrankZeta7Id}) with $q$ replaced by $q^2$.
If we divide both sides by 
$\frac{\aqprod{q}{q^2}{\infty}\aqprod{q^{98}}{q^{98}}{\infty}\jacprod{q^{42}}{q^{98}}}
	{\jacprod{q^{14},q^{28}}{q^{98}}}$,
we find that (\ref{EqF3Mod7CrankId}) is equivalent to
\begin{align}
	\label{EqF3Mod7CrankModularMess}
	&1
	+
	(\zeta_7^6+\zeta_7-1)\frac{\GEta{98}{28}{\tau}}{\GEta{98}{42}{\tau}}
	+
	(\zeta_7^5+\zeta_7^2)\frac{\GEta{98}{28}{\tau}^{2}}{\GEta{98}{42}{\tau}^{2}}
	+
	(\zeta_7^4+\zeta_7^3+1)\frac{\GEta{98}{14}{\tau}}{\GEta{98}{42}{\tau}}
	-
	(\zeta_7^6+\zeta_7)\frac{\GEta{98}{14}{\tau}\GEta{98}{28}{\tau}}
		{\GEta{98}{42}{\tau}^{2}}
	\nonumber\\&\quad
	-
	(\zeta_7^5+\zeta_7^2+1)\frac{\GEta{98}{14}{\tau}^{2}}{\GEta{98}{42}{\tau}^{2}}
	\nonumber\\
	&=
	\frac{\GEta{98}{28}{\tau}\GEta{98}{49}{\tau}^{1/2}}
		{\GEta{2}{1}{\tau}^{1/2}\GEta{98}{42}{\tau}}
	-
	(\zeta_7^5+\zeta_7^2)
	\frac{\GEta{98}{14}{\tau}\GEta{98}{28}{\tau}\GEta{98}{35}{\tau}\GEta{98}{49}{\tau}^{1/2}}
		{\GEta{2}{1}{\tau}^{1/2}\GEta{49}{21}{\tau}\GEta{98}{42}{\tau}}
	\nonumber\\&\quad
	+
	(\zeta_7^5+\zeta_7^2+1)
	\frac{\GEta{98}{14}{\tau}\GEta{98}{21}{\tau}\GEta{98}{28}{\tau}\GEta{98}{49}{\tau}^{1/2}}
		{\GEta{2}{1}{\tau}^{1/2}\GEta{49}{7}{\tau}\GEta{98}{42}{\tau}}
	-
	2\frac{\GEta{98}{28}{\tau}\GEta{98}{35}{\tau}}
		{\GEta{2}{1}{\tau}^{1/2}\GEta{98}{42}{\tau}\GEta{98}{49}{\tau}^{1/2}}
	\nonumber\\&\quad
	-
	(-\zeta_7^6-\zeta_7+1)
		\frac{\GEta{49}{14}{\tau}\GEta{98}{14}{\tau}\GEta{98}{49}{\tau}^{1/2}}
		{\GEta{2}{1}{\tau}^{1/2}\GEta{49}{7}{\tau}\GEta{98}{42}{\tau}}
	-
	(\zeta_7^6+\zeta_7)\frac{\GEta{49}{7}{\tau}\GEta{98}{28}{\tau}\GEta{98}{49}{\tau}^{1/2}}
		{\GEta{2}{1}{\tau}^{1/2}\GEta{49}{21}{\tau}\GEta{98}{42}{\tau}}
	\nonumber\\&\quad
	+
	(\zeta_7^5+\zeta_7^2+1)\frac{\GEta{98}{14}{\tau}\GEta{98}{49}{\tau}^{1/2}}
		{\GEta{2}{1}{\tau}^{1/2}\GEta{98}{42}{\tau}}
	-
	(\zeta_7^6+\zeta_7^5+\zeta_7^2+\zeta_7)
		\frac{\GEta{49}{21}{\tau}\GEta{98}{14}{\tau}\GEta{98}{28}{\tau}\GEta{98}{49}{\tau}^{1/2}}
		{\GEta{2}{1}{\tau}^{1/2}\GEta{49}{14}{\tau}\GEta{98}{42}{\tau}^{2}}
	\nonumber\\&\quad
	-
	(\zeta_7^5+\zeta_7^2)\frac{\GEta{98}{14}{\tau}\GEta{98}{28}{\tau}\GEta{98}{49}{\tau}^{1/2}}
		{\GEta{2}{1}{\tau}^{1/2}\GEta{98}{42}{\tau}^{2}}
.
\end{align}
However, by Theorem 3 of \cite{Robins} each individual term of 
(\ref{EqF3Mod7CrankModularMess}) is a modular function with respect to 
$\Gamma_1(98)$. Using Theorem 4 of \cite{Robins} to compute the orders at the 
cusps, as explained previously, we find to prove (\ref{EqF3Mod7CrankModularMess})
that we need only check this identity in the $q$-series expansion past
$q^{210}$. This we do with Maple and so (\ref{EqF3Mod7CrankId}) is true.

\end{proof}

\begin{proposition}
\begin{align}
	\label{EqF3Mod3RankId}
	&\frac{\aqprod{q}{q^2}{\infty}}{\aqprod{q^2}{q^2}{\infty}}
	\left(1+
	\sum_{n=1}^\infty\frac{(1-\zeta_3)(1-\zeta_3^{-1})q^n(1+q^{2n})}
		{(1-\zeta_3q^{2n})(1-\zeta_3^{-1}q^{2n})}
	\right)
	\nonumber\\
	&=
	\frac{\aqprod{q^9}{q^9}{\infty}^4}{\aqprod{q^3}{q^3}{\infty}\aqprod{q^{18}}{q^{18}}{\infty}^2}
	+
	2q\frac{\aqprod{q^9}{q^9}{\infty}\aqprod{q^{18}}{q^{18}}{\infty}}
		{\aqprod{q^6}{q^6}{\infty}}
	+q^2\frac{\aqprod{q^3}{q^3}{\infty}\aqprod{q^{18}}{q^{18}}{\infty}^4}
		{\aqprod{q^6}{q^6}{\infty}^2\aqprod{q^9}{q^9}{\infty}^2}
	,\\
	\label{EqF3Mod5RankId}
	&\frac{\aqprod{q}{q^2}{\infty}}{\aqprod{q^2}{q^2}{\infty}}
	\left(	
		1+\sum_{n=1}^\infty
		\frac{(1-\zeta_5)(1-\zeta_5^{-1})q^n(1+q^{2n})}{(1-\zeta_5q^{2n})(1-\zeta_5q^{2n})}
	\right)
	\nonumber\\
	&=
		\frac{\aqprod{q^{25}}{q^{25}}{\infty}\jacprod{q^{15}}{q^{50}}}
			{\jacprod{q^{5}}{q^{25}}}
		+
		(1-\zeta_5-\zeta_5^4)q\frac{\aqprod{q^{25}}{q^{25}}{\infty}}
			{\jacprod{q^{10}}{q^{50}}}
		+
		q^2\frac{\aqprod{q^{50}}{q^{50}}{\infty}\jacprod{q^{15}}{q^{50}}}
			{\aqprod{q^{25}}{q^{50}}{\infty}\jacprod{q^{10}}{q^{50}}}
		\nonumber\\&\quad
		-
		(\zeta_5+\zeta_5^4)q^7\frac{\aqprod{q^{50}}{q^{50}}{\infty}\jacprod{q^{5}}{q^{50}}}
			{\aqprod{q^{25}}{q^{50}}{\infty}\jacprod{q^{20}}{q^{50}}}
		+
		(1+\zeta_5+\zeta_5^4)q^3\frac{\aqprod{q^{25}}{q^{25}}{\infty}}{\jacprod{q^{20}}{q^{50}}}
		-		
		(\zeta_5+\zeta_5^4)q^4\frac{\aqprod{q^{25}}{q^{25}}{\infty}\jacprod{q^{5}}{q^{50}}}
			{\jacprod{q^{10}}{q^{25}}}
	,\\
	\label{EqF3Mod7RankId}
	&\frac{\aqprod{q}{q^2}{\infty}}{\aqprod{q^2}{q^2}{\infty}}
	\left(	
		1+\sum_{n=1}^\infty
		\frac{(1-\zeta_7)(1-\zeta_7^{-1})q^n(1+q^{2n})}{(1-\zeta_7q^{2n})(1-\zeta_7q^{2n})}
	\right)
	\nonumber\\
	&=
		\frac{\aqprod{q^{49}}{q^{49}}{\infty}}{\jacprod{q^{14}}{q^{98}}}
		+
		q\frac{\aqprod{q^{98}}{q^{98}}{\infty}\jacprod{q^{35}}{q^{98}}}
			{\aqprod{q^{49}}{q^{98}}{\infty}\jacprod{q^{14}}{q^{98}}}
		-
		(\zeta_7+\zeta_7^6)q\frac{\aqprod{q^{14}}{q^{14}}{\infty}}
		{\aqprod{q^{7}}{q^{14}}{\infty}}
		-
		(\zeta_7^2+\zeta_7^5)q^8\frac{\aqprod{q^{98}}{q^{98}}{\infty}\jacprod{q^{21}}{q^{98}}}
			{\aqprod{q^{49}}{q^{98}}{\infty}\jacprod{q^{28}}{q^{98}}}
		\nonumber\\&\quad
		+
		q^2\frac{\aqprod{q^{49}}{q^{49}}{\infty}\jacprod{q^{14}}{q^{49}}}
			{\jacprod{q^7}{q^{49}}\jacprod{q^{28}}{q^{98}}}
		-
		(\zeta_7^2+\zeta_7^5)q^3\frac{\aqprod{q^{49}}{q^{49}}{\infty}\jacprod{q^{7}}{q^{49}}}
			{\jacprod{q^{21}}{q^{49}}\jacprod{q^{14}}{q^{98}}}
		+
		(1+\zeta_7^2+\zeta_7^5)q^4\frac{\aqprod{q^{49}}{q^{49}}{\infty}}
			{\jacprod{q^{28}}{q^{98}}}
		\nonumber\\&\quad
		+
		(1+\zeta_7^2+\zeta_7^5)q^5\frac{\aqprod{q^{49}}{q^{49}}{\infty}\jacprod{q^{21}}{q^{49}}}
			{\jacprod{q^{14}}{q^{49}}\jacprod{q^{42}}{q^{98}}}
		-
		(\zeta_7^2+\zeta_7^5)q^6\frac{\aqprod{q^{49}}{q^{49}}{\infty}}
			{\jacprod{q^{42}}{q^{98}}}
.
\end{align}	
\end{proposition}
We see that (\ref{Dissection3F3}) follows by subtracting
(\ref{EqF3Mod3CrankId}) from (\ref{EqF3Mod3RankId}) and dividing by
$(1-\zeta_3)(1-\zeta_3^{-1})$, (\ref{Dissection5F3}) follows by subtracting
(\ref{EqF3Mod5CrankId}) from (\ref{EqF3Mod5RankId}) and dividing by
$(1-\zeta_5)(1-\zeta_5^{-1})$, and (\ref{Dissection7F3}) follows by subtracting
(\ref{EqF3Mod7CrankId}) from (\ref{EqF3Mod7RankId}) and dividing by
$(1-\zeta_7)(1-\zeta_7^{-1})$.
While we can prove (\ref{EqF3Mod3RankId}) with elementary rearrangements, the
proofs of (\ref{EqF3Mod5RankId}) and (\ref{EqF3Mod7RankId}) will require the
following identities.
We use equation (17.1) from \cite[page 303]{Berndt}, which in our
notation is
\begin{align}\label{LambertSeriesIdentBerndt}
	\frac{q}{ab}\frac{\aqprod{q^2}{q^2}{\infty}^4\jacprod{a,b,ab}{q^2}}
		{\aqprod{q}{q}{\infty}^2\jacprod{aq,bq,abq}{q^2}}
	&=
	\sum_{n=1}^\infty \frac{q^n}{1-q^{2n}}
	\left(\frac{1}{a^nb^n} - \frac{1}{a^n} - \frac{1}{b^n} +a^n + b^n -a^nb^n \right)
.
\end{align}
We also use Theorem 1 of \cite{AndrewsLewisLiu} with $b=a$ and $c=q^{1/2}$,
\begin{align}\label{LambertSeriesIdentALL}
	\frac{\aqprod{q}{q}{\infty}^2\jacprod{a^2,aq^{1/2},aq^{1/2}}{q}}
		{\jacprod{a,a,q^{1/2}, a^2q^{1/2}}{q}}
	&=
		1 
		+ 2\sum_{k=0}^\infty \frac{aq^k}{1-aq^k}
		- 2\sum_{k=1}^\infty \frac{q^k/a}{1-q^k/a}
		- \sum_{k=0}^\infty \frac{a^2q^{k+1/2}}{1-a^2q^{k+1/2}}
		+ \sum_{k=1}^\infty \frac{q^{k-1/2}/a^2}{1-q^{k-1/2}/a^2}
.
\end{align}
Lastly, we will use the following dissection formula for certain quotients of theta
functions.
\begin{lemma}\label{LemmaThetaQuotientDissection}
Let $M$ be a positive integer and 
$|q^{2M}|<|z|<1$, then
\begin{align*}
	\frac{\aqprod{q^{2M}}{q^{2M}}{\infty}^2\jacprod{zq^M}{q^{2M}}}
	{\aqprod{q^{M}}{q^{2M}}{\infty}^2\jacprod{z}{q^{2M}}}
	&=
	\frac{\aqprod{q^{2M^2}}{q^{2M^2}}{\infty}^2}
		{\jacprod{z^M}{q^{2M^2}}}
	\sum_{k \in A}z^k
	\frac{ \jacprod{z^Mq^{M(2k+1)}}{q^{2M^2}}  }
		{\jacprod{q^{M(2k+1)}}{q^{2M^2}}}
,
\end{align*}
where $A$ is any full set of residues modulo $N$ (such as
$A=\{0,1,2,\dots,M-1\}$).
\end{lemma}
\begin{proof}
We recall a specialization of Ramanujan's $_1 \Psi_1$ formula gives
\begin{align*}
	\frac{\aqprod{q}{q}{\infty}^2\jacprod{xy}{q}}{\jacprod{x,y}{q}}
	&=
	\sum_{n=-\infty}^\infty \frac{x^n}{1-yq^n}
\end{align*}
for $|q|<|x|<1$. We let $q\mapsto q^{2M}$, $x=z$, and $y=q^M$ to find that
\begin{align*}
	\frac{\aqprod{q^{2M}}{q^{2M}}{\infty}^2\jacprod{zq^M}{q^{2M}}}
		{\aqprod{q^{M}}{q^{2M}}{\infty}^2\jacprod{z}{q^{2M}}}
	&=
	\sum_{n=-\infty}^\infty \frac{z^n}{1-q^{2Mn+M}}
	\\
	&=
	\sum_{k \in A}
	\sum_{n=-\infty}^\infty \frac{z^{Mn+k}}{1-q^{2M^2n+2Mk+M}}
	\\
	&=
	\sum_{k \in A}z^k
	\sum_{n=-\infty}^\infty \frac{z^{Mn}}{1-q^{2M^2n+2Mk+M}}
	\\
	&=
	\sum_{k \in A}z^k
	\frac{ \jacprod{z^Mq^{M(2k+1)}}{q^{2M^2}}\aqprod{q^{2M^2}}{q^{2M^2}}{\infty}^2  }
		{\jacprod{z^M, q^{M(2k+1)}}{q^{2M^2}}}
.
\end{align*}
\end{proof}
In particular, we can set $z=\pm q^a$ for $1\le a < 2M$.
Similar dissection
formulas for certain quotients of theta functions follow from both the quintuple 
product identity and Theorem 2.1 of \cite{Chan2}.

\begin{proof}[Proof of (\ref{EqF3Mod3RankId})]
We have that
\begin{align*}
	1+
	\sum_{n=1}^\infty \frac{(1-\zeta_3)(1-\zeta_3^{-1})q^{n}(1+q^{2n})}
		{(1-\zeta_3q^{2n})(1-\zeta_3^{-1}q^{2n})}
	&=
	1+
	3\sum_{n=1}^\infty \frac{q^{n}(1+q^{2n})(1-q^{2n})}
		{(1-q^{6n})}
	\\
	&=
	1+
	3\sum_{n=1}^\infty \frac{q^n-q^{5n}}{1-q^{6n}}
	\\
	&=
	1+
	3\sum_{n=1}^\infty q^n E_1(n:6)
	\\
	&=
	\frac{\aqprod{q^2}{q^2}{\infty}^6\aqprod{q^3}{q^3}{\infty}}
	{\aqprod{q}{q}{\infty}^3\aqprod{q^6}{q^6}{\infty}^2}
,
\end{align*}
where
\begin{align*}
	E_{r}(N;m) 
	&=
	\sum_{d\mid N, d\equiv r\pmod{m}}1
	-
	\sum_{d\mid N, d\equiv -r\pmod{m}}1
,
\end{align*}
by (32.42) of \cite{Fine}.
By Gauss and the Jacobi Triple Product Identity
\begin{align*}
	\frac{\aqprod{q^2}{q^2}{\infty}^2}{\aqprod{q}{q}{\infty}}
	&=
	\frac{1}{2}\sum_{n=-\infty}^\infty q^{n(n+1)/2}
	\\
	&=
	\frac{1}{2}\sum_{n=-\infty}^\infty q^{(9n^2+3n)/2}
	+\frac{q}{2}\sum_{n=-\infty}^\infty q^{(9n^2+9n)/2}
	+\frac{q^3}{2}\sum_{n=-\infty}^\infty q^{(9n^2+15n)/2}
	\\
	&=
	\frac{1}{2}\aqprod{-q^3,-q^6,q^9}{q^9}{\infty}
	+\frac{q}{2}\aqprod{-1,-q^9,q^9}{q^9}{\infty}
	+\frac{q^3}{2}\aqprod{-q^{-3},-q^{12},q^9}{q^9}{\infty}
	\\
	&=
	\aqprod{-q^3,-q^6,q^9}{q^9}{\infty}
	+q\aqprod{-q^9,-q^9,q^9}{q^9}{\infty}
.
\end{align*}
Thus
\begin{align*}
	\frac{\aqprod{q^2}{q^2}{\infty}^4}{\aqprod{q}{q}{\infty}^2}
	&=
	\aqprod{-q^3,-q^6,q^9}{q^9}{\infty}^2
	+
	2q
	\aqprod{-q^3,-q^6,-q^9,-q^9,q^9,q^9}{q^9}{\infty}
	+q^2\aqprod{-q^9,-q^9,q^9}{q^9}{\infty}^2
	\\
	&=
	\aqprod{-q^3,-q^6,q^9}{q^9}{\infty}^2
	+
	2q\aqprod{-q^3}{q^3}{\infty}\aqprod{q^9}{q^9}{\infty}\aqprod{q^{18}}{q^{18}}{\infty}
	+
	q^2\aqprod{-q^9}{q^9}{\infty}^2\aqprod{q^{18}}{q^{18}}{\infty}^2
.
\end{align*}
And so
\begin{align*}
	&\frac{\aqprod{q}{q^2}{\infty}}{\aqprod{q^2}{q^2}{\infty}}
	\left(1+
	\sum_{n=1}^\infty\frac{(1-\zeta_3)(1-\zeta_3^{-1})q^n(1+q^{2n})}
		{(1-\zeta_3q^{2n})(1-\zeta_3^{-1}q^{2n})}
	\right)
	\nonumber\\
	&=
	\frac{\aqprod{q^2}{q^2}{\infty}^4\aqprod{q^3}{q^3}{\infty}}
	{\aqprod{q}{q}{\infty}^2\aqprod{q^6}{q^6}{\infty}^2}
	\nonumber\\
	&=
	\frac{\aqprod{q^3}{q^3}{\infty}\aqprod{-q^3,-q^6,q^9}{q^9}{\infty}^2}{\aqprod{q^6}{q^6}{\infty}^2}
	+
	2q\frac{\aqprod{q^9}{q^9}{\infty}\aqprod{q^{18}}{q^{18}}{\infty}}
		{\aqprod{q^6}{q^6}{\infty}}
	+q^2\frac{\aqprod{q^3}{q^3}{\infty}\aqprod{-q^9}{q^9}{\infty}^2\aqprod{q^{18}}{q^{18}}{\infty}^2}
		{\aqprod{q^6}{q^6}{\infty}^2}
	\nonumber\\
	&=
	\frac{\aqprod{q^9}{q^9}{\infty}^4}{\aqprod{q^3}{q^3}{\infty}\aqprod{q^{18}}{q^{18}}{\infty}^2}
	+
	2q\frac{\aqprod{q^9}{q^9}{\infty}\aqprod{q^{18}}{q^{18}}{\infty}}
		{\aqprod{q^6}{q^6}{\infty}}
	+q^2\frac{\aqprod{q^3}{q^3}{\infty}\aqprod{q^{18}}{q^{18}}{\infty}^4}
		{\aqprod{q^6}{q^6}{\infty}^2\aqprod{q^9}{q^9}{\infty}^2}
.
\end{align*}
\end{proof}

\begin{proof}[Proof of (\ref{EqF3Mod5RankId}) ]

To begin, we have
\begin{align*}
	&1+(1-\zeta_5)(1-\zeta_5^{-1})\sum_{n=1}^\infty
	\frac{q^n(1+q^{2n})}{(1-\zeta_5q^{2n})(1-\zeta_5q^{2n})}
	\\
	&=
	1+(1-\zeta_5)(1-\zeta_5^{-1})\sum_{n=1}^\infty
		\frac{q^n(1+q^{2n})(1-q^{2n})(1-\zeta_5^2q^{2n})(1-\zeta_5^3q^{2n})}{1-q^{10n}}
	\\
	&=
	1
		+
		(2-\zeta_5-\zeta_5^4)\sum_{n=1}^\infty
			\frac{q^n-q^{9n}}{1-q^{10n}}
		+
		(1+2\zeta_5+2\zeta_5^4)\sum_{n=1}^\infty
			\frac{q^{3n}-q^{7n}}{1-q^{10n}}
.
\end{align*}
We claim that
\begin{align}
	\label{EqF3Mod5RankProofId1}
	&\frac{\aqprod{q}{q^2}{\infty}}{\aqprod{q^2}{q^2}{\infty}}
	\left(
	1+2\sum_{n=1}^\infty\frac{q^n-q^{9n}}{1-q^{10n}}
		+\sum_{n=1}^\infty\frac{q^{3n}-q^{7n}}{1-q^{10n}}
	\right)
	\nonumber\\
	&=
		\frac{\aqprod{q^{10}}{q^{10}}{\infty}\jacprod{q^4}{q^{10}}}
		{\aqprod{q^5}{q^{10}}{\infty}\jacprod{q}{q^{10}}}
	\nonumber\\
	&=
		\frac{\aqprod{q^{25}}{q^{25}}{\infty}\jacprod{q^{15}}{q^{50}}}
			{\jacprod{q^5}{q^{25}}}
		+
		q\frac{\aqprod{q^{25}}{q^{25}}{\infty}}{\jacprod{q^{10}}{q^{50}}}
		+
		q^2\frac{\aqprod{q^{50}}{q^{50}}{\infty}\jacprod{q^{15}}{q^{50}}}
			{\aqprod{q^{25}}{q^{50}}{\infty}\jacprod{q^{10}}{q^{50}}}
		+
		q^3\frac{\aqprod{q^{25}}{q^{25}}{\infty}}{\jacprod{q^{20}}{q^{50}}}
.
\end{align}
We note the second identity of (\ref{EqF3Mod5RankProofId1}) is just an 
application of Lemma 
\ref{LemmaThetaQuotientDissection}, with $M=5$ and $z=q$ and simplifying
the resulting products. So we need to verify the first identity.
We have
\begin{align*}
	&1+2\sum_{n=1}^\infty\frac{q^n-q^{9n}}{1-q^{10n}}
		+\sum_{n=1}^\infty\frac{q^{3n}-q^{7n}}{1-q^{10n}}
	\\
	&=
		1
		+2\sum_{n=1}^\infty\frac{q^{10n+1}}{1-q^{10n+1}}
		-2\sum_{n=0}^\infty\frac{q^{10n-1}}{1-q^{10n-1}}
		-\sum_{n=0}^\infty\frac{q^{10n+7}}{1-q^{10n+7}}
		+\sum_{n=1}^\infty\frac{q^{10n-7}}{1-q^{10n-7}}
	\\
	&=
	\frac{\aqprod{q^{10}}{q^{10}}{\infty}^2\jacprod{q^2,q^6,q^6}{q^{10}}}
		{\jacprod{q,q,q^5,q^7}{q^{10}}}
,
\end{align*}
where we have used (\ref{LambertSeriesIdentALL}) with $q\mapsto q^{10}$ and
$a=q$. Thus
\begin{align*}
	\frac{\aqprod{q}{q^2}{\infty}}{\aqprod{q^2}{q^2}{\infty}}
	\left(1+2\sum_{n=1}^\infty\frac{q^n-q^{9n}}{1-q^{10n}}
		+\sum_{n=1}^\infty\frac{q^{3n}-q^{7n}}{1-q^{10n}}
	\right)
	&=
	\frac{\aqprod{q}{q^2}{\infty}\aqprod{q^{10}}{q^{10}}{\infty}^2\jacprod{q^2,q^6,q^6}{q^{10}}}
		{\aqprod{q^2}{q^2}{\infty}\jacprod{q,q,q^5,q^7}{q^{10}}}
	\\
	&=	
	\frac{\aqprod{q^{10}}{q^{10}}{\infty}\jacprod{q^4}{q^{10}}}
		{\aqprod{q^5}{q^{10}}{\infty}\jacprod{q}{q^{10}}}
.
\end{align*}
Next we claim that
\begin{align}
	\label{EqF3Mod5RankProofId2}
	&\frac{\aqprod{q}{q^2}{\infty}}{\aqprod{q^2}{q^2}{\infty}}
	\left(
	-\sum_{n=1}^\infty\frac{q^n-q^{9n}}{1-q^{10n}}
		+2\sum_{n=1}^\infty\frac{q^{3n}-q^{7n}}{1-q^{10n}}
	\right)
	\nonumber\\
	&=
	-q\frac{\aqprod{q^{10}}{q^{10}}{\infty}\jacprod{q^2}{q^{10}}}
		{\aqprod{q^{5}}{q^{10}}{\infty}\jacprod{q^3}{q^{10}}}
	\nonumber\\
	&=
		-q\frac{\aqprod{q^{25}}{q^{25}}{\infty}}{\jacprod{q^{10}}{q^{50}}}
		-
		q^7\frac{\aqprod{q^{50}}{q^{50}}{\infty}\jacprod{q^{5}}{q^{50}}}
			{\aqprod{q^{25}}{q^{50}}{\infty}\jacprod{q^{20}}{q^{50}}}
		+
		q^3\frac{\aqprod{q^{25}}{q^{25}}{\infty}}{\jacprod{q^{20}}{q^{50}}}
		-		
		q^4\frac{\aqprod{q^{25}}{q^{25}}{\infty}\jacprod{q^{5}}{q^{50}}}
			{\jacprod{q^{10}}{q^{25}}}
.
\end{align}
We note the second identity of (\ref{EqF3Mod5RankProofId1}) 
follows by Lemma \ref{LemmaThetaQuotientDissection}
with $M=5$ and $z=q^3$. For the first identity, we apply (\ref{LambertSeriesIdentBerndt})
with $q\mapsto q^5$ and $a=b=q^{2}$ to get
\begin{align*}
	\sum_{n=1}^\infty\frac{q^n-q^{9n}}{1-q^{10n}}
		-2\sum_{n=1}^\infty\frac{q^{3n}-q^{7n}}{1-q^{10n}}
	&=
	q\frac{ \aqprod{q^{10}}{q^{10}}{\infty}^4\jacprod{q^{2},q^{2},q^{4}}{q^{10}} }
		{\aqprod{q^5}{q^5}{\infty}^2\jacprod{q^7,q^7,q^9}{q^{10}}}
.
\end{align*}
Thus
\begin{align*}
	\frac{\aqprod{q}{q^2}{\infty}}{\aqprod{q^2}{q^2}{\infty}}
	\left(
	-2\sum_{n=1}^\infty\frac{q^n-q^{9n}}{1-q^{10n}}
		+2\sum_{n=1}^\infty\frac{q^{3n}-q^{7n}}{1-q^{10n}}
	\right)
	&=
	-q\frac{\aqprod{q}{q^2}{\infty}\aqprod{q^{10}}{q^{10}}{\infty}^4\jacprod{q^2,q^2,q^4}{q^{10}} }
		{\aqprod{q^2}{q^2}{\infty}\aqprod{q^5}{q^5}{\infty}^2\jacprod{q^7,q^7,q^9}{q^{10}}}
	\\
	&=
	-q\frac{\aqprod{q^{10}}{q^{10}}{\infty}\jacprod{q^2}{q^{10}} }
		{\aqprod{q^5}{q^{10}}{\infty}\jacprod{q^3}{q^{10}}}
.
\end{align*}
Equation (\ref{EqF3Mod5RankId}) now follows from
(\ref{EqF3Mod5RankProofId1}) and (\ref{EqF3Mod5RankProofId2}).
\end{proof}

\begin{proof}[Proof of (\ref{EqF3Mod7RankId})]

We begin with
\begin{align*}
	&1+(1-\zeta_7)(1-\zeta_7^{-1})\sum_{n=1}^\infty
	\frac{q^n(1+q^{2n})}{(1-\zeta_7q^{2n})(1-\zeta_7q^{2n})}
	\\
	&=
	1+(1-\zeta_7)(1-\zeta_7^{-1})\sum_{n=1}^\infty
		\frac{q^n(1+q^{2n})(1-q^{2n})(1-\zeta_7^2q^{2n})(1-\zeta_7^3q^{2n})
			(1-\zeta_7^4q^{2n})(1-\zeta_7^5q^{2n})}
		{1-q^{14n}}
	\\
	&=
	1
		+
		(2-\zeta_7-\zeta_7^6)\sum_{n=1}^\infty
			\frac{q^n-q^{13n}}{1-q^{14n}}
		+
		(\zeta_7-\zeta_7^2-\zeta_7^5+\zeta_7^6)\sum_{n=1}^\infty
			\frac{q^{3n}-q^{11n}}{1-q^{14n}}
		\\&\quad
		+
		(1+\zeta_7+2\zeta_7^2+2\zeta_7^5+\zeta_5^6)\sum_{n=1}^\infty
			\frac{q^{5n}-q^{9n}}{1-q^{14n}}
.
\end{align*}
We claim that
\begin{align}
	\label{EqF3Mod7RankProofId1}
	&\frac{\aqprod{q}{q^2}{\infty}}{\aqprod{q^2}{q^2}{\infty}}
	\left(
    	1+\sum_{n=1}^\infty \frac{2q^n + q^{5n} - q^{9n} - 2q^{13n}}
			{1-q^{14n}} 
	\right)
	\nonumber\\
	&=
	\frac{\aqprod{q^{14}}{q^{14}}{\infty}\jacprod{q^3,q^6}{q^{14}}}
		{\aqprod{q^{7}}{q^{14}}{\infty}\jacprod{q,q^4}{q^{14}}}
	\nonumber\\
	&=
		\frac{\aqprod{q^{49}}{q^{49}}{\infty}}{\jacprod{q^{14}}{q^{98}}}
		+
		q\frac{\aqprod{q^{98}}{q^{98}}{\infty}\jacprod{q^{35}}{q^{98}}}
			{\aqprod{q^{49}}{q^{98}}{\infty}\jacprod{q^{14}}{q^{98}}}
		+
		q^2\frac{\aqprod{q^{49}}{q^{49}}{\infty}\jacprod{q^{14}}{q^{49}}}
			{\jacprod{q^7}{q^{49}}\jacprod{q^{28}}{q^{98}}}
		+
		q^4\frac{\aqprod{q^{49}}{q^{49}}{\infty}}
			{\jacprod{q^{28}}{q^{98}}}
		\nonumber\\&\quad
		+
		q^5\frac{\aqprod{q^{49}}{q^{49}}{\infty}\jacprod{q^{21}}{q^{49}}}
			{\jacprod{q^{14}}{q^{49}}\jacprod{q^{42}}{q^{98}}}
.
\end{align}
For this we apply (\ref{LambertSeriesIdentALL}) with $q\mapsto q^{14}$
and $a=q$ to get that
\begin{align*}
	1+\sum_{n=1}^\infty \frac{2q^n + q^{5n} - q^{9n} - 2q^{13n}}{1-q^{14n}} 
	&=
		1
		+2\sum_{n=0}^\infty \frac{q^{14n+1}}{1-q^{14n+1}}
		-2\sum_{n=1}^\infty \frac{q^{14n-1}}{1-q^{14n-1}}
		-\sum_{n=0}^\infty \frac{q^{14n+9}}{1-q^{14n+9}}
		+\sum_{n=1}^\infty \frac{q^{14n-9}}{1-q^{14n-9}}
	\\
	&=
	\frac{\aqprod{q^{14}}{q^{14}}{\infty}^2\jacprod{q^2,q^8,q^8}{q^{14}}}
		{\jacprod{q,q,q^7,q^9}{q^{14}}}
	.
\end{align*}
Thus
\begin{align*}
	\frac{\aqprod{q}{q^2}{\infty}}{\aqprod{q^2}{q^2}{\infty}}
	\left(
    	1+\sum_{n=1}^\infty \frac{2q^n + q^{5n} - q^{9n} - 2q^{13n}}
			{1-q^{14n}} 
	\right)
	&=
	\frac{\aqprod{q}{q^2}{\infty}\aqprod{q^{14}}{q^{14}}{\infty}^2\jacprod{q^2,q^8,q^8}{q^{14}}}
		{\aqprod{q^2}{q^2}{\infty}\jacprod{q,q,q^7,q^9}{q^{14}}}
	\\
	&=
	\frac{\aqprod{q^{14}}{q^{14}}{\infty}\jacprod{q^3,q^6}{q^{14}}}
		{\aqprod{q^7}{q^{14}}{\infty}\jacprod{q,q^4}{q^{14}}}
.
\end{align*}
To verify the second identity of (\ref{EqF3Mod7RankProofId1}),
we divide both sides by
$\frac{\aqprod{q^{14}}{q^{14}}{\infty}\jacprod{q^3,q^6}{q^{14}}}
 {\aqprod{q^{7}}{q^{14}}{\infty}\jacprod{q,q^4}{q^{14}}}$,
to get an identity between modular functions on
$\Gamma_1(98)$. As we did in the proof of (\ref{EqF3Mod7CrankId}),
we examine the orders at the poles of various modular functions and find
that to prove the identity between modular functions we just need to
verify the identity in the $q$-series expansion past $q^{147}$. We do this in Maple.

Next we claim that
\begin{align}
	\label{EqF3Mod7RankProofId2}
	\frac{\aqprod{q}{q^2}{\infty}}{\aqprod{q^2}{q^2}{\infty}}
	\left(
    	\sum_{n=1}^\infty \frac{(-q^n + q^{3n} + q^{5n} - q^{9n} - q^{11n} + q^{13n})}
			{1-q^{14n}} 
	\right)
	&=
	-q\frac{\aqprod{q^{14}}{q^{14}}{\infty}}
		{\aqprod{q^{7}}{q^{14}}{\infty}}
.
\end{align}
This is actually Entry 17(i) of \cite{Berndt}, so there is nothing for us to
prove.
Lastly we claim that
\begin{align}
	\label{EqF3Mod7RankProofId3}
	&\frac{\aqprod{q}{q^2}{\infty}}{\aqprod{q^2}{q^2}{\infty}}
	\left(
    	\sum_{n=1}^\infty \frac{(-q^{3n} + 2q^{5n} - 2q^{9n} + q^{11n})}
			{1-q^{14n}} 
	\right)
	\nonumber\\
	&=
	-q^3\frac{\aqprod{q^{14}}{q^{14}}{\infty}\jacprod{q,q^2}{q^{14}}}
		{\aqprod{q^{7}}{q^{14}}{\infty}\jacprod{q^5,q^6}{q^{14}}}
	\nonumber\\
	&=
		-q^8\frac{\aqprod{q^{98}}{q^{98}}{\infty}\jacprod{q^{21}}{q^{98}}}
			{\aqprod{q^{49}}{q^{98}}{\infty}\jacprod{q^{28}}{q^{98}}}
		-
		q^3\frac{\aqprod{q^{49}}{q^{49}}{\infty}\jacprod{q^{7}}{q^{49}}}
			{\jacprod{q^{21}}{q^{49}}\jacprod{q^{14}}{q^{98}}}
		+
		q^4\frac{\aqprod{q^{49}}{q^{49}}{\infty}}
			{\jacprod{q^{28}}{q^{98}}}
		+
		q^5\frac{\aqprod{q^{49}}{q^{49}}{\infty}\jacprod{q^{21}}{q^{49}}}
			{\jacprod{q^{14}}{q^{49}}\jacprod{q^{42}}{q^{98}}}
		\nonumber\\&\quad
		-
		q^6\frac{\aqprod{q^{49}}{q^{49}}{\infty}}
			{\jacprod{q^{42}}{q^{98}}}
.
\end{align}
For this we apply (\ref{LambertSeriesIdentBerndt}) with $q\mapsto q^7$ and
$a=b=q^2$ to get that
\begin{align*}
	\sum_{n=1}^\infty \frac{(q^{3n} - 2q^{5n} + 2q^{9n} - q^{11n})}{1-q^{14n}} 
	&=
	q^3\frac{\aqprod{q^{14}}{q^{14}}{\infty}^4\jacprod{q^2,q^2,q^4}{q^{14}}}
		{\aqprod{q^7}{q^7}{\infty}^2\jacprod{q^9,q^9,q^{11}}{q^{14}}}
.
\end{align*}
Thus
\begin{align*}
	\frac{\aqprod{q}{q^2}{\infty}}{\aqprod{q^2}{q^2}{\infty}}
	\left(
    	\sum_{n=1}^\infty \frac{(-q^{3n} + 2q^{5n} - 2q^{9n} + q^{11n})}
			{1-q^{14n}} 
	\right)
	&=	
	-q^3\frac{\aqprod{q}{q^2}{\infty}\aqprod{q^{14}}{q^{14}}{\infty}^4\jacprod{q^2,q^2,q^4}{q^{14}}}
		{\aqprod{q^2}{q^2}{\infty}\aqprod{q^7}{q^7}{\infty}^2\jacprod{q^9,q^9,q^{11}}{q^{14}}}
	\\
	&=
	-q^3\frac{\aqprod{q^{14}}{q^{14}}{\infty}\jacprod{q,q^2}{q^{14}}}
		{\aqprod{q^{7}}{q^{14}}{\infty}\jacprod{q^5,q^6}{q^{14}}}
.
\end{align*}
To verify the second identity of (\ref{EqF3Mod7RankProofId3}),
we divide both sides by
$q^3\frac{\aqprod{q^{14}}{q^{14}}{\infty}\jacprod{q,q^2}{q^{14}}}
	{\aqprod{q^{7}}{q^{14}}{\infty}\jacprod{q^5,q^6}{q^{14}}}$,
to get an identity between modular functions on
$\Gamma_1(98)$. We examine the orders at the poles of various modular functions and find
that to prove the identity between modular functions we just need to
verify the identity in the $q$-series expansion past $q^{147}$. We do this in Maple.

Equation (\ref{EqF3Mod7RankId}) now follows from
(\ref{EqF3Mod7RankProofId1}), (\ref{EqF3Mod7RankProofId2}), and
(\ref{EqF3Mod7RankProofId3}).

\end{proof}

\section{Dissections for $S_{G4}(z,q)$ and $S_{AG4}(z,q)$  }

Here we find the $5$-dissections of $S_{G4}(\zeta_5,q)$ and $S_{AG4}(\zeta_5,q)$
using the techniques in \cite{JenningsShaffer}. Atkin and Swinnerton-Dyer
pioneered this method to study the rank of partitions \cite{AS}. Since then 
Lovejoy and Osburn used it to study the Dyson rank of 
overpartitions \cite{LO1}, the $M_2$-rank of overpartitions \cite{LO3}, and the 
$M_2$-rank of 
partitions without repeated odd parts \cite{LO2}. Also Ekin demonstrated that 
it could be used
for the crank of partitions \cite{Ekin}. 
However, it is quicker to derive these dissections
from the product and series 
forms of $S_{G4}(z,q)$ and $S_{AG4}(z,q)$. For this reason, we omit some of 
the details but do include the general identities that lead to the end results.
To begin we use Bailey's Lemma with $\rho_1=z$ and $\rho_2=z^{-1}$ to get that
\begin{align*}
	S_{G4}(z,q)
	&=
		\frac{\aqprod{q^4}{q^4}{\infty}\aqprod{-q}{q^2}{\infty}}
			{\aqprod{z,z^{-1}}{q^2}{\infty}}
		\sum_{n=0}^\infty \aqprod{z,z^{-1}}{q^2}{n}q^{2n}\beta^{G4}_n
		-
		\frac{\aqprod{q^4}{q^4}{\infty}\aqprod{-q}{q^2}{\infty}}
		{\aqprod{z,z^{-1}}{q^2}{\infty}}
	\\
	&=
		\frac{\aqprod{q^4}{q^4}{\infty}\aqprod{-q}{q^2}{\infty}}
		{(1-z)(1-z^{-1})\aqprod{q^2}{q^2}{\infty}^2}
		\left(
		1+
		\sum_{n=1}^\infty\frac{(1-z)(1-z^{-1})(-1)^nq^{\frac{n^2+3n}{2}}(1+q^n)}
			{(1-zq^{2n})(1-z^{-1}q^{2n})}
		\right)
		\\&\quad
		-
		\frac{\aqprod{q^4}{q^4}{\infty}\aqprod{-q}{q^2}{\infty}}
		{\aqprod{z,z^{-1}}{q^2}{\infty}}
	\\
	&=
		\frac{1}{(1-z)(1-z^{-1})\aqprod{q}{q}{\infty}}
		\left(
		1+
		\sum_{n=1}^\infty\frac{(1-z)(1-z^{-1})(-1)^nq^{\frac{n^2+3n}{2}}(1+q^n)}
			{(1-zq^{2n})(1-z^{-1}q^{2n})}
		\right)
		-
		\frac{\aqprod{q^2}{q^2}{\infty}}
		{\aqprod{q,z,z^{-1}}{q^2}{\infty}}
.
\end{align*}
Similarly, for $S_{AG4}(z,q)$, Bailey's Lemma gives that
\begin{align*}
	S_{AG4}(z,q)
	&=
		\frac{1}{(1-z)(1-z^{-1})\aqprod{q}{q}{\infty}}
		\left(
		1+
		\sum_{n=1}^\infty\frac{(1-z)(1-z^{-1})(-1)^nq^{\frac{n^2+n}{2}}(1+q^{3n})}
			{(1-zq^{2n})(1-z^{-1}q^{2n})}
		\right)
		-
		\frac{\aqprod{q^2}{q^2}{\infty}}
		{\aqprod{q,z,z^{-1}}{q^2}{\infty}}
.
\end{align*}

\begin{proposition}\label{G4CrankTerm5Dissection}
\begin{align*}
	\frac{\aqprod{q^2}{q^2}{\infty}}
	{\aqprod{q,\zeta_5q^2,\zeta_5^{-1}q^2}{q^2}{\infty}}
	&=
	\frac{\aqprod{q^{25}}{q^{25}}{\infty}\jacprod{q^{20}}{q^{50}}}
		{\jacprod{q^{10}}{q^{25}}\jacprod{q^{10}}{q^{50}}}
	+
	(\zeta_5+\zeta_5^{4})q^5\frac{\aqprod{q^{50}}{q^{50}}{\infty}}
		{\aqprod{q^{25}}{q^{50}}{\infty}\jacprod{q^{20}}{q^{50}}}
	+
	q\frac{\aqprod{q^{25}}{q^{25}}{\infty}}
		{\jacprod{q^{5}}{q^{25}}}
	\\&\quad
	+
	(\zeta_5+\zeta_5^4)q^2\frac{\aqprod{q^{25}}{q^{25}}{\infty}}
		{\jacprod{q^{10}}{q^{25}}}
	+
	(\zeta_5+\zeta_5^4)q^3\frac{\aqprod{q^{25}}{q^{25}}{\infty}\jacprod{q^{10}}{q^{50}}}
		{\jacprod{q^5}{q^{25}}\jacprod{q^{20}}{q^{50}}}
	\\&\quad
	+
	q^3\frac{\aqprod{q^{50}}{q^{50}}{\infty}}
		{\aqprod{q^{25}}{q^{50}}{\infty}\jacprod{q^{10}}{q^{50}}}
.
\end{align*}
\end{proposition}
\begin{proof}
By Gauss and the Jacobi Triple Product Identity we have
\begin{align*}
	\frac{\aqprod{q^2}{q^2}{\infty}}{\aqprod{q}{q^2}{\infty}}
	&=
	\frac{1}{2}\sum_{n=-\infty}^\infty q^{n(n+1)/2}
	=
	\frac{\aqprod{q^{25}}{q^{25}}{\infty}\jacprod{q^{20}}{q^{50}}}
		{\jacprod{q^{10}}{q^{25}}}
	+
	q\frac{\aqprod{q^{25}}{q^{25}}{\infty}\jacprod{q^{10}}{q^{50}}}
		{\jacprod{q^{5}}{q^{25}}}
	+
	q^3\frac{\aqprod{q^{50}}{q^{50}}{\infty}}{\aqprod{q^{25}}{q^{50}}{\infty}}
.
\end{align*}
By Lemma 3.9 of \cite{Garvan1}, with $q$ replaced by $q^2$ we have
\begin{align*}
	\frac{1}{\aqprod{\zeta_5q^2,\zeta_5^{-1}q^2}{q^2}{\infty}}
	&=
	\frac{1}{\jacprod{q^{10}}{q^{50}}}
	+
	(\zeta_5+\zeta_5^4)\frac{q^2}{\jacprod{q^{20}}{q^{50}}}
.
\end{align*}
Multiplying these two identities together gives the result.
\end{proof}

With Proposition \ref{G4CrankTerm5Dissection} we find (\ref{Dissection5G4})
and (\ref{Dissection5AG4})
to be equivalent to the following.
\begin{proposition}\label{TheoremDissectionPartsForG4AndAG4}
\begin{align*}
	&\frac{1}{\aqprod{q}{q}{\infty}}
	\left(1+ \sum_{n=1}^\infty\frac{(1-\zeta_5)(1-\zeta_5^{-1})(-1)^nq^{\frac{n^2+3n}{2}}(1+q^n)}
		{(1-\zeta_5q^{2n})(1-\zeta_5^{-1}q^{2n})}\right)
	\\
	&=
		\frac{\aqprod{q^{25}}{q^{25}}{\infty}\jacprod{q^{20}}{q^{50}}}
			{\jacprod{q^{10}}{q^{25}}\jacprod{q^{10}}{q^{50}}}
		+
		(1-2\zeta_5-2\zeta_5^4)q^{5}
			\frac{\aqprod{q^{50}}{q^{50}}{\infty}}
			{\aqprod{q^{25}}{q^{50}}{\infty}\jacprod{q^{20}}{q^{50}}}
		\\&\quad
		-
		(1+2\zeta_5+2\zeta_5^4)q^{10}
			\frac{\aqprod{q^{100}}{q^{100}}{\infty}\jacprod{q^{10}}{q^{200}}}
			{\jacprod{q^{10}}{q^{50}}\jacprod{q^{5}}{q^{100}}}
		+
		\frac{q\aqprod{q^{25}}{q^{25}}{\infty}}
			{\jacprod{q^{5}}{q^{25}}}
		\\&\quad
		-
		(2-\zeta_5-\zeta_5^4)q^{6}
			\frac{\aqprod{q^{100}}{q^{100}}{\infty}\jacprod{q^{30}}{q^{200}}}
			{\jacprod{q^{10}}{q^{50}}\jacprod{q^{15}}{q^{100}}}
		+
		(\zeta_5+\zeta_5^4)\frac{q^2\aqprod{q^{25}}{q^{25}}{\infty}}
			{\jacprod{q^{10}}{q^{25}}}
		\\&\quad	
		-
		(2-\zeta_5-\zeta_5^4)q^{12}
			\frac{\aqprod{q^{100}}{q^{100}}{\infty}\jacprod{q^{10}}{q^{200}}}
			{\jacprod{q^{20}}{q^{50}}\jacprod{q^{5}}{q^{100}}}
		+
		(-1+\zeta_5+\zeta_5^4)q^{3}
			\frac{\aqprod{q^{50}}{q^{50}}{\infty}}
			{\aqprod{q^{25}}{q^{50}}{\infty}\jacprod{q^{10}}{q^{50}}}
		\\&\quad
		+
		(\zeta_5+\zeta_5^4)\frac{q^3\aqprod{q^{25}}{q^{25}}{\infty}\jacprod{q^{10}}{q^{50}}}
			{\jacprod{q^5}{q^{25}}\jacprod{q^{20}}{q^{50}}}
		+
		(1-3\zeta_5-3\zeta_5^4)q^{8}
			\frac{\aqprod{q^{100}}{q^{100}}{\infty}\jacprod{q^{30}}{q^{200}}}
			{\jacprod{q^{20}}{q^{50}}\jacprod{q^{15}}{q^{100}}}
	,
	\\
	&\frac{1}{\aqprod{q}{q}{\infty}}
	\left(1+ \sum_{n=1}^\infty\frac{(1-\zeta_5)(1-\zeta_5^{-1})(-1)^nq^{\frac{n^2+n}{2}}(1+q^{3n})}
		{(1-\zeta_5q^{2n})(1-\zeta_5^{-1}q^{2n})}\right)
	\\
	&=
		\frac{\aqprod{q^{25}}{q^{25}}{\infty}\jacprod{q^{20}}{q^{50}}}
			{\jacprod{q^{10}}{q^{25}}\jacprod{q^{10}}{q^{50}}}
		+
		(\zeta_5+\zeta_5^{4})q^5\frac{\aqprod{q^{50}}{q^{50}}{\infty}}
			{\aqprod{q^{25}}{q^{50}}{\infty}\jacprod{q^{20}}{q^{50}}}
		+
		(\zeta_5+\zeta_5^4-2)q^{10}
			\frac{\aqprod{q^{100}}{q^{100}}{\infty}\jacprod{q^{10}}{q^{200}}}
			{\jacprod{q^{10}}{q^{50}}\jacprod{q^{5}}{q^{100}}}
		\\&\quad
		+
		(\zeta_5+\zeta_5^4-2)q
			\frac{\aqprod{q^{100}}{q^{100}}{\infty}\jacprod{q^{70}}{q^{200}}}
			{\jacprod{q^{10}}{q^{50}}\jacprod{q^{35}}{q^{100}}}
		+
		q\frac{\aqprod{q^{25}}{q^{25}}{\infty}}
			{\jacprod{q^{5}}{q^{25}}}
		\\&\quad
		-
		(1+2\zeta_5+2\zeta_5^4)q^{6}
			\frac{\aqprod{q^{100}}{q^{100}}{\infty}\jacprod{q^{30}}{q^{200}}}
			{\jacprod{q^{10}}{q^{50}}\jacprod{q^{15}}{q^{100}}}
		+
		(\zeta_5+\zeta_5^4)q^2\frac{\aqprod{q^{25}}{q^{25}}{\infty}}
			{\jacprod{q^{10}}{q^{25}}}
		\\&\quad
		+
		(1-3\zeta_5-3\zeta_5^4)q^{12}
			\frac{\aqprod{q^{100}}{q^{100}}{\infty}\jacprod{q^{10}}{q^{200}}}
			{\jacprod{q^{20}}{q^{50}}\jacprod{q^{5}}{q^{100}}}
		+
		(1-3\zeta_5-3\zeta_5^4)q^{3}
			\frac{\aqprod{q^{100}}{q^{100}}{\infty}\jacprod{q^{70}}{q^{200}}}
			{\jacprod{q^{20}}{q^{50}}\jacprod{q^{35}}{q^{100}}}
		\\&\quad
		+
		(\zeta_5+\zeta_5^4)q^3\frac{\aqprod{q^{25}}{q^{25}}{\infty}\jacprod{q^{10}}{q^{50}}}
			{\jacprod{q^5}{q^{25}}\jacprod{q^{20}}{q^{50}}}
		+
		q^3\frac{\aqprod{q^{50}}{q^{50}}{\infty}}
			{\aqprod{q^{25}}{q^{50}}{\infty}\jacprod{q^{10}}{q^{50}}}
		\\&\quad
		+
		(-2+\zeta_5+\zeta_5^4)q^8
			\frac{\aqprod{q^{100}}{q^{100}}{\infty}\jacprod{q^{30}}{q^{200}}}
			{\jacprod{q^{20}}{q^{50}}\jacprod{q^{15}}{q^{100}}}
.
\end{align*}
\end{proposition}

The major work in proving identities like this is to find identities that allow
us to see the series terms on the left hand sides as products. To this end, we note that
\begin{align*}
	&1+
	\sum_{n=1}^\infty\frac{(1-z)(1-z^{-1})(-1)^nq^{\frac{n^2+3n}{2}}(1+q^n)}
		{(1-zq^{2n})(1-z^{-1}q^{2n})}
	\\	
	&=
	\sum_{n=-\infty}^\infty \frac{(1-z)(1-z^{-1})(-1)^nq^{\frac{n^2+3n}{2}}}
		{(1-zq^{2n})(1-z^{-1}q^{2n})}
	\\
	&=
	\sum_{n=-\infty}^\infty \frac{(1-z)(1-z^{-1})q^{2n^2+3n}}
		{(1-zq^{4n})(1-z^{-1}q^{4n})}
	-
	\sum_{n=-\infty}^\infty \frac{(1-z)(1-z^{-1})q^{2n^2+5n+2}}
		{(1-zq^{4n+2})(1-z^{-1}q^{4n+2})}
.
\end{align*}
We then define
\begin{align*}
	V_\ell(b) &= \sum_{\substack{n=-\infty\\n\not=0}}^\infty \frac{q^{2n^2+bn}}{1-q^{4\ell n}}
	,
	&U_\ell(b) &= \sum_{n=-\infty}^\infty \frac{q^{2n^2+bn}}{1-q^{4\ell n+2\ell}}
.
\end{align*}
We note replacing $n$ by $-n$ gives that
\begin{align*}
	V_{\ell}(b) &= -V_{\ell}(4\ell-b)
	,
	&U_{\ell}(b) &= -q^{2\ell-b+2}U_\ell(4\ell+4-b)
.
\end{align*}

Next we have
\begin{align*}
	&
	\sum_{n=-\infty}^\infty \frac{(1-\zeta_5)(1-\zeta_5^{-1})q^{2n^2+3n}}
		{(1-\zeta_5q^{4n})(1-\zeta_5^{-1}q^{4n})}
	\\
	&=
		1+
		\sum_{n\not=0}^\infty
		\frac{(1-\zeta_5)(1-\zeta_5^{-1})(-1)^nq^{2n^2+3n}(1-q^{4n})(1-\zeta_5^2q^{4n})(1-\zeta_5^3q^{4n})}
			{(1-q^{20n})}
	\\
	&=
		1+
		(2-\zeta_5-\zeta_5^4)(V_5(3) - V_5(15))
		+
		(-1+3\zeta_5+3\zeta_5^4)(V_5(7) - V_5(11))
	\\
	&=
		1+
		(2-\zeta_5-\zeta_5^4)(V_5(3) + V_5(5))
		+
		(-1+3\zeta_5+3\zeta_5^4)(V_5(7) + V_5(9))
.
\end{align*}
Similarly,
\begin{align*}
	&
	\sum_{n=-\infty}^\infty \frac{(1-\zeta_5)(1-\zeta_5^{-1})q^{2n^2+5n+2}}
		{(1-\zeta_5q^{4n+2})(1-\zeta_5^{-1}q^{4n+2})}
	\\
	&=
		\sum_{n\not=-\infty}^\infty
		\frac{(1-\zeta_5)(1-\zeta_5^{-1})(-1)^nq^{2n^2+5n+2}(1-q^{4n+2})(1-\zeta_5^2q^{4n+2})(1-\zeta_5^3q^{4n+2})}
			{(1-q^{20n+10})}
	\\
	&=
		(2-\zeta_5-\zeta_5^4)(q^2U_5(5) - q^8U_5(17))
		+
		(-1+3\zeta_5+3\zeta_5^4)(q^4U_5(9) - q^6U_5(13))
	\\
	&=
		(2-\zeta_5-\zeta_5^4)(q^2U_5(5) + q^3U_5(7))
		+
		(-1+3\zeta_5+3\zeta_5^4)(q^4U_5(9) + q^5U_5(11))
.
\end{align*}
Thus
\begin{align*}
	&1+ \sum_{n=1}^\infty\frac{(1-\zeta_5)(1-\zeta_5^{-1})(-1)^nq^{\frac{n^2+3n}{2}}(1+q^n)}
		{(1-\zeta_5q^{2n})(1-\zeta_5^{-1}q^{2n})}
	\\
	&=
		1+
		(2-\zeta_5-\zeta_5^4)(V_5(3) - q^2U_5(5) + V_5(5) - q^3U_5(7) )
		\\&\quad
		+
		(-1+3\zeta_5+3\zeta_5^4)(V_5(7) - q^4U_5(9) + V_5(9) - q^5U_5(11))
.
\end{align*}
In the same fashion we deduce that
\begin{align*}
	&1+ \sum_{n=1}^\infty\frac{(1-\zeta_5)(1-\zeta_5^{-1})(-1)^nq^{\frac{n^2+n}{2}}(1+q^{3n})}
		{(1-\zeta_5q^{2n})(1-\zeta_5^{-1}q^{2n})}
	\\
	&=
		1+
		(2-\zeta_5-\zeta_5^4)(V_5(1) - qU_5(3) + V_5(7) - q^4U_5(9)  )
		\\&\quad+
		(-1+3\zeta_5+3\zeta_5^4)(V_5(5) -q^3U_5(7) - V_5(9) + q^5U_5(11))
.
\end{align*}

Similar to Ekin's work in \cite{Ekin}, we use the functions
\begin{align*}
	T(z,w,q) &= \sum_{n=-\infty}^\infty \frac{(-1)^nq^{n(n+1)/2}w^n}{1-zq^n}
	,\\
	T^*(w,q) &= \sum_{\substack{n=-\infty\\n\not=0}}^\infty \frac{(-1)^nq^{n(n+1)/2}w^n}{1-q^n}
	,\\
	h(z,q) &= T^*(z^{-1},q)+zT(z^2,z,q)
.
\end{align*}

\begin{lemma}\label{LemmaUV}
For $b$ and $\ell$ odd integers with $\ell>1$, we have
\begin{align*}
	&V_\ell(b) -q^{\frac{b+1}{2}}U_\ell(b+2)
	\\
	&=
	h(-q^{2\ell^2-b\ell}, q^{4\ell^2})
	+\aqprod{q^{4\ell^2}}{q^{4\ell^2}}{\infty}^2\jacprod{-q^{2\ell^2-b\ell}}{q^{4\ell^2}}
	\sum_{k=1}^{\ell-1}	
	q^{2k^2+bk}\frac{\jacprod{q^{2b\ell+8k\ell}}{q^{4\ell^2}}}
	{\jacprod{q^{4k\ell},q^{2b\ell+4k\ell},-q^{2\ell^2+b\ell+4k\ell} }{q^{4\ell^2}}}
.		
\end{align*}
\end{lemma}
\begin{proof}
We have
\begin{align*}
	&V_\ell(b) -q^{\frac{b+1}{2}}U_\ell(b+2)
	\\
	&=
	\sum_{\substack{n=-\infty\\n\not=0}}^\infty\frac{q^{2n^2+bn}}{1-q^{4\ell n}}
	-q^{\frac{b+1}{2}}
	\sum_{n=-\infty}^\infty\frac{q^{2n^2+bn+2n}}{1-q^{4\ell n + 2\ell}}
	\\
	&=
		\sum_{k=0}^{\ell-1}
		q^{2k^2+bk}
		\sum_{\substack{n=-\infty\\(n,k)\not=(0,0)}}^\infty
		\frac{q^{2\ell^2n^2 + b\ell n + 4k\ell n}}{1-q^{4\ell^2n + 4k\ell}}
		-
		\sum_{k=0}^{\ell-1}		
		q^{2k^2 + bk - 4k\ell -b\ell + 2\ell^2}
		\sum_{n=-\infty}^\infty
		\frac{ q^{2\ell^2n^2 + 4\ell^2n -4k\ell n -b\ell n} }
		{1-q^{4\ell^2n+4\ell^2-2b\ell-4k\ell}}
	\\
	&=
		\sum_{k=0}^{\ell-1}
		q^{2k^2+bk}
		\sum_{\substack{n=-\infty\\(n,k)\not=(0,0)}}^\infty
		\frac{(-1)^n q^{2\ell^2n(n+1)} (-q)^{b\ell n + 4k\ell n - 2\ell^2n}}
			{1-q^{4\ell^2n + 4k\ell}}
		\\&\quad
		-
		\sum_{k=0}^{\ell-1}		
		q^{2k^2 + bk - 4k\ell -b\ell + 2\ell^2}
		\sum_{n=-\infty}^\infty
		\frac{ (-1)^nq^{2\ell^2n(n+1)} (-q)^{2\ell^2n -4k\ell n -b\ell n} }
			{1-q^{4\ell^2n+4\ell^2-2b\ell-4k\ell}}
	\\
	&=
		T^*(-q^{b\ell-2\ell^2},q^{4\ell^2})
		-
		q^{2\ell^2-b\ell}T( q^{4\ell^2-2b\ell}, -q^{2\ell^2-b\ell}, q^{4\ell^2} )
		\\&\quad
		+
		\sum_{k=1}^{\ell-1}
		q^{2k^2+bk}
		\left(
			T(q^{4k\ell}, -q^{b\ell + 4k\ell - 2\ell^2}, q^{4\ell^2} )		
			-
			q^{- 4k\ell -b\ell + 2\ell^2}
			T(q^{4\ell^2-2b\ell-4k\ell}, -q^{2\ell^2 -4k\ell -b\ell}, q^{4\ell^2} )
		\right)
	\\
	&=
		h(-q^{2\ell^2-b\ell},q^{4\ell^2})
		\\&\quad
		+
		\sum_{k=1}^{\ell-1}
		q^{2k^2+bk}
		\left(
			T(q^{4k\ell}, -q^{b\ell + 4k\ell - 2\ell^2}, q^{4\ell^2} )		
			-
			q^{- 4k\ell -b\ell + 2\ell^2}
			T(q^{4\ell^2-2b\ell-4k\ell}, -q^{2\ell^2 -4k\ell -b\ell}, q^{4\ell^2} )
		\right)
.
\end{align*}
Here we have replaced $n$ by $\ell n+k$ in $V_\ell(b)$ and
replaced $n$ by $\ell n - k + \ell - \frac{b+1}{2}$ in $U_\ell(b+2)$.
By Lemma 2 of \cite{Lewis} we have that
\begin{align*}
	wT(zw,w,q) + T(zw^{-1},w^{-1},q)
	&=
	\frac{\aqprod{q}{q}{\infty}^2\jacprod{z,w^2}{q}}
		{\jacprod{zw^{-1},zw,w}{q}}
.
\end{align*}
We note one could also deduce this identity using Theorem 2.1 of \cite{Chan}.
Applying this with $q\mapsto q^{4\ell^2}$,
$w=-q^{2\ell^2 - b\ell - 4k\ell}$,
$z=-q^{2\ell^2 - b\ell}$ yields
\begin{align*}
	&T(q^{4k\ell}, -q^{b\ell + 4k\ell - 2\ell^2}, q^{4\ell^2} )		
	-
	q^{- 4k\ell -b\ell + 2\ell^2}
	T(q^{4\ell^2-2b\ell-4k\ell}, -q^{2\ell^2 -4k\ell -b\ell}, q^{4\ell^2} )
	\\
	&=
	\frac{\aqprod{q^{4\ell^2}}{q^{4\ell^2}}{\infty}^2
		\jacprod{-q^{2\ell^2-b\ell}, q^{4\ell^2-2b\ell-8k\ell}}{q^{4\ell^2}}}
	{\jacprod{q^{4k\ell}, q^{4\ell^2-2b\ell-4k\ell}, -q^{2\ell^2-b\ell-4k\ell}}{q^{4\ell^2}}}
	\\
	&=
	\frac{\aqprod{q^{4\ell^2}}{q^{4\ell^2}}{\infty}^2
		\jacprod{-q^{2\ell^2-b\ell}, q^{2b\ell+8k\ell}}{q^{4\ell^2}}}
	{\jacprod{q^{4k\ell}, q^{2b\ell+4k\ell}, -q^{2\ell^2+b\ell+4k\ell}}{q^{4\ell^2}}}
.
\end{align*}
This completes the proof.
\end{proof}

Also one can express combinations of $h(z,q)$ in terms of products by using
Lemma 1 of \cite{Ekin}.
For our purposes, we could use the following, the proof of which is basic algebra
and applying Lemma 1 of \cite{Ekin}.
\begin{proposition}
\begin{align*}
	&(3a-d)h(-q^{5},q^{100})
	+(3b-a)h(-q^{15},q^{100})
	+(3c-b)h(-q^{45},q^{100})
	+(3d-c)h(-q^{35},q^{100})
	+c+d
	\\
	&=
		a\frac{\aqprod{q^{100}}{q^{100}}{\infty}^2\jacprod{q^{10}}{q^{100}}^3}
			{\jacprod{-q^5}{q^{100}}^3\jacprod{-q^{15}}{q^{100}}}	
		+
		(c-a)\frac{\aqprod{q^{100}}{q^{100}}{\infty}^2\jacprod{q^{20}}{q^{100}}^3}
			{\jacprod{q^{10}}{q^{100}}^3\jacprod{q^{30}}{q^{100}}}	
		+
		b\frac{\aqprod{q^{100}}{q^{100}}{\infty}^2\jacprod{q^{30}}{q^{100}}^3}
			{\jacprod{-q^{15}}{q^{100}}^3\jacprod{-q^{45}}{q^{100}}}	
		\\&\quad
		+		
		(d-b)\frac{\aqprod{q^{100}}{q^{100}}{\infty}^2\jacprod{q^{40}}{q^{100}}^3}
			{\jacprod{q^{30}}{q^{100}}^3\jacprod{q^{10}}{q^{100}}}	
		+
		cq^{35}\frac{\aqprod{q^{100}}{q^{100}}{\infty}^2\jacprod{q^{10}}{q^{100}}^3}
			{\jacprod{-q^{45}}{q^{100}}^3\jacprod{-q^{35}}{q^{100}}}	
		+
		dq^{5}\frac{\aqprod{q^{100}}{q^{100}}{\infty}^2\jacprod{q^{30}}{q^{100}}^3}
			{\jacprod{-q^{35}}{q^{100}}^3\jacprod{-q^{5}}{q^{100}}}	
.
\end{align*}
\end{proposition}

Also we note that
\begin{align*}
	h(-q^{25},q^{100})
	&=
	\frac{\aqprod{q^{100}}{q^{100}}{\infty}^2\jacprod{q^{50}}{q^{100}}^4}
	{\jacprod{-q^{25}}{q^{100}}^4 }
	=
	\frac{\aqprod{q^{25}}{q^{25}}{\infty}^4}{\aqprod{q^{100}}{q^{100}}{\infty}^2}
.
\end{align*}

This would allow use to express the identities in 
Proposition \ref{TheoremDissectionPartsForG4AndAG4}
just in terms of infinite products. We could then rewrite the identities
strictly in terms of modular functions. The identity in terms of
modular functions could then be proved as we did for the identities
in (\ref{EqF3Mod7CrankId}), (\ref{EqF3Mod7RankProofId1}), and
(\ref{EqF3Mod7RankProofId3}). We do not include these calculations here.


\section{Proof of Corollary \ref{TheoremEqualCranks}}
We note that Theorem \ref{TheoremDissections} immediately gives that the
coefficients of
$q^{3n}$ in $S_{F3}(\zeta_3,q)$,
$q^{5n+1}$ in $S_{B2}(\zeta_5,q)$,
$q^{5n+4}$ in $S_{B2}(\zeta_5,q)$,
$q^{5n}$ in $S_{F3}(\zeta_5,q)$,
$q^{5n+4}$ in $S_{F3}(\zeta_5,q)$,
$q^{5n+4}$ in $S_{G4}(\zeta_5,q)$,
$q^{5n+4}$ in $S_{AG4}(\zeta_5,q)$,
$q^{7n+1}$ in $S_{B2}(\zeta_7,q)$,
$q^{7n+5}$ in $S_{B2}(\zeta_7,q)$,
$q^{7n}$ in $S_{F3}(\zeta_7,q)$,
$q^{7n+4}$ in $S_{F3}(\zeta_7,q)$, and
$q^{7n+6}$ in $S_{F3}(\zeta_7,q)$
are all zero. Thus the identities in Corollary \ref{TheoremEqualCranks} for
$B2$, $F3$, $G4$, and $AG4$ follow. 
We still need to prove that the coefficients of
$q^{3n+2}$ in $S_{J1}(\zeta_3,q)$,
$q^{3n}$ in $S_{J2}(\zeta_3,q)$, and
$q^{3n+1}$ in $S_{J3}(\zeta_3,q)$ are zero.
We prove this by using Theorem \ref{TheoremSingleSeries}.

For $J1$, we use (\ref{TheoremSeriesForJ1}) to see that
\begin{align*}
	&(1+\zeta_3)(1-\zeta_3)(1-\zeta_3^{-1})\aqprod{q^3}{q^3}{\infty}
	S_{J1}(\zeta_3,q)
	\\
	&=
	\sum_{j=2}^\infty 
	\frac{(1-\zeta_3^j)(1-\zeta_3^{j-1})\zeta_3^{1-j}(-1)^{j+1}q^{\frac{j(j-1)}{2}}
			(1 - q^{j} - q^{2j-2} + q^{4j-3} + q^{5j-2} - q^{6j-3})}  
		{(1-q^{3j-3})(1-q^{3j})}
	.
\end{align*}
Thus the non-zero terms occur only when $j\equiv 2\pmod{3}$, 
but one finds that
$q^{\frac{j(j-1)}{2}}(1 - q^{j} - q^{2j-2} + q^{4j-3} + q^{5j-2} - q^{6j-3})$ 
only contributes terms of the form $q^{3n}$ and $q^{3n+1}$
when $j\equiv 2\pmod{3}$. Thus $S_{J1}(\zeta_3,q)$ has no non-zero terms of the
form $q^{3n+2}$.

For $J2$, we use (\ref{TheoremSeriesForJ2}) to see that
\begin{align*}
	&(1+\zeta_3)(1-\zeta_3)(1-\zeta_3^{-1})\aqprod{q^3}{q^3}{\infty}S_{J2}(\zeta_3,q)
	\\
	&=
	\sum_{j=2}^\infty 
	\frac{(1-\zeta_3^j)(1-\zeta_3^{j-1})\zeta_3^{1-j}(-1)^{j+1}q^{\frac{j(j-1)}{2}}
			(1-q^{j-1}-q^{2j}+q^{4j-1} +q^{5j-3} - q^{6j-3})}  
		{(1-q^{3j-3})(1-q^{3j})}
	.
\end{align*}
Thus the non-zero terms occur only when $j\equiv 2\pmod{3}$, 
but one finds that
$q^{\frac{j(j-1)}{2}}(1-q^{j-1}-q^{j}-q^{2j}+q^{4j-1} +q^{5j-3} - q^{6j-3})$ 
only contributes terms of the form $q^{3n+1}$ and $q^{3n+2}$
when $j\equiv 2\pmod{3}$. Thus $S_{J2}(\zeta_3,q)$ has no non-zero terms of the
form $q^{3n}$.

For $J3$, we use (\ref{TheoremSeriesForJ3}) to see that
\begin{align*}
	&(1+\zeta_3)(1-\zeta_3)(1-\zeta_3^{-1})\aqprod{q^3}{q^3}{\infty}S_{J3}(\zeta_3,q)
	\\
	&=
	\sum_{j=2}^\infty 
	\frac{(1-\zeta_3^j)(1-\zeta_3^{j-1})\zeta_3^{1-j}(-1)^{j+1}q^{\frac{j(j-1)}{2}}
			(q^{j-1}-q^{j}-q^{2j-2}+q^{2j}+q^{4j-3} -q^{4j-1} - q^{5j-3} +q^{5j-2} )}  
		{(1-q^{3j-3})(1-q^{3j})}
	.
\end{align*}
Thus the non-zero terms occur only when $j\equiv 2\pmod{3}$, 
but one finds that
$q^{\frac{j(j-1)}{2}}(q^{j-1}-q^{j}-q^{2j-2}+q^{2j}+q^{4j-3} -q^{4j-1} - q^{5j-3} +q^{5j-2} )$ 
only contributes terms of the form $q^{3n}$ and $q^{3n+1}$
when $j\equiv 2\pmod{3}$. Thus $S_{J3}(\zeta_3,q)$ has no non-zero terms of the
form $q^{3n+1}$.

\section{Concluding Remarks}

We could also prove dissection identities for 
$S_{J2}(\zeta_3,q)$ and $S_{J3}(\zeta_3,q)$ in the same way we proved the 
dissections for $S_{G4}(\zeta_5,q)$ and $S_{AG4}(\zeta_5,q)$. It would require
defining functions similar to $V_\ell$, $U_\ell$, $T$, and $h$ and finding
the appropriate identities. Whereas $S_{G4}(z,q)$ and $S_{AG4}(z,q)$ use 
functions and formulas similar to those used for the crank,
$S_{J2}(z,q)$ and $S_{J3}(z,q)$ would use functions and formulas similar 
to those used for the rank. We save this for another time.

We might think to try the methods of this paper with the following Bailey pairs relative to 
$(1,q^2)$,
\begin{align*}
	\beta^{F1}_n &= \frac{1}{\aqprod{q,q^2}{q^2}{n}}
	,&
	\alpha^{F1}_n &= 
	\left\{\begin{array}{ll}
		1 & \mbox{ if } n=0
		\\
		q^{2n^2-n}(1+q^{2n}) & \mbox{ if } n \ge 1
	\end{array}\right.,
	\\
	\beta^{G*}_n &= \frac{q^{n^2-2n}}{\aqprod{q^4}{q^4}{n}\aqprod{q}{q^2}{n}}
	,&
	\alpha^{G*}_n &= 
	\left\{\begin{array}{ll}
		1 & \mbox{ if } n=0
		\\
		(-1)^{n(n-1)/2} q^{n(n-3)/2}(1+(-1)^nq^{3n}) & \mbox{ if } n \ge 1
	\end{array}\right.
	\\
	\beta^{G**}_n &= \frac{q^{n^2}}{\aqprod{q^4}{q^4}{n}\aqprod{q}{q^2}{n}}
	,&
	\alpha^{G**}_n &= 
	\left\{\begin{array}{ll}
		1 & \mbox{ if } n=0
		\\
		(-1)^{n(n+1)/2} q^{n(n-1)/2}(1+(-1)^nq^{n}) & \mbox{ if } n \ge 1
	\end{array}\right.
.
\end{align*}
These Bailey pairs are $F(1)$ and the first two Bailey pairs a
listed on page 470 of \cite{Slater1}. We would then define the following series.
\begin{align*}
	S_{F1}(z,q)
	&=
	\frac{\aqprod{q}{q^2}{\infty}}{\aqprod{z,z^{-1}}{q^2}{\infty}}
	\sum_{n=1}^\infty \aqprod{z,z^{-1}}{q^2}{n}q^{2n}\beta_n^{F1}
	=
	\frac{\aqprod{q}{q^2}{\infty}}{\aqprod{z,z^{-1}}{q^2}{\infty}}
	\sum_{n=1}^\infty \frac{\aqprod{z,z^{-1}}{q^2}{n}q^{2n}}
		{\aqprod{q,q^2}{q^2}{n}}
	,
	\\
	S_{G*}(z,q)
	&=
	\frac{\aqprod{q^4}{q^4}{\infty}\aqprod{q}{q^2}{\infty}}
		{\aqprod{z,z^{-1}}{q^2}{\infty}}
	\sum_{n=1}^\infty \aqprod{z,z^{-1}}{q^2}{n}q^{2n}\beta^{G*}_n
	=
	\frac{\aqprod{q^4}{q^4}{\infty}\aqprod{q}{q^2}{\infty}}
		{\aqprod{z,z^{-1}}{q^2}{\infty}}
	\sum_{n=1}^\infty \frac{\aqprod{z,z^{-1}}{q^2}{n}q^{n^2}}
		{\aqprod{q^4}{q^4}{n}\aqprod{q}{q^2}{n}}
	,
	\\
	S_{G**}(z,q)
	&=
	\frac{\aqprod{q^4}{q^4}{\infty}\aqprod{q}{q^2}{\infty}}
		{\aqprod{z,z^{-1}}{q^2}{\infty}}
	\sum_{n=1}^\infty \aqprod{z,z^{-1}}{q^2}{n}q^{2n}\beta^{G*}_n
	=
	\frac{\aqprod{q^4}{q^4}{\infty}\aqprod{q}{q^2}{\infty}}
		{\aqprod{z,z^{-1}}{q^2}{\infty}}
	\sum_{n=1}^\infty \frac{\aqprod{z,z^{-1}}{q^2}{n}q^{n^2+2n}}
		{\aqprod{q^4}{q^4}{n}\aqprod{q}{q^2}{n}}
	.
\end{align*}
At first it appears that these series may explain congruences for other
new spt functions, however, they are old functions in disguise. In particular
one finds that $S_{G*}(z,-q)=S_{AG4}(z,q)$ and $S_{G**}(z,-q)=S_{G4}(z,q)$.
Similarly $S_{F1}(z,-q) = 	S2(z,q)$, where $S2(z,q)$ is a two variable 
generalization for the M2spt function studied in \cite{GarvanJennings}. 

While this paper gives the last of the spt-crank-type functions
for Bailey pairs from \cite{Slater1} and \cite{Slater2},
with such simple linear congruences, we should expect there to be many more
interesting spt-crank-type functions. There are plenty of other Bailey pairs
from other sources that may lead to new functions. Also we have not used
all the Bailey pairs from \cite{Slater1} and \cite{Slater2}, we have only
used those that have simple congruences.
So far all Bailey pairs have been
relative to $(a,q)$ with $a=1$, but a slight change in the form of the 
spt-crank-type functions may allow for many useful functions coming from
other values of $a$. In a coming paper, we investigate Bailey pairs arising
from variations of Bailey's Lemma and conjugate Bailey pairs.

The functions studied here and in \cite{JenningsShaffer} and \cite{GarvanJennings2} may
have additional properties worth studying. While the $M_{A1}(m,n)$ were given
a combinatorial interpretation in \cite{JenningsShaffer} (in particular they
are non-negative), work on the other $M_{X}(m,n)$ still needs to be done.
Additionally, the original spt functions for partitions and overpartitions are 
known to be related to mock modular forms and harmonic Maass forms. 
Any of the other spt functions that can be expressed in terms of known
rank functions and infinite products will also lead to harmonic Maass forms,
so these functions can be studied from that aspect as well.

\bibliographystyle{abbrv}
\bibliography{BaileyGroupBFGJRef}

\end{document}